\documentclass[1 [leqno,11pt]{amsart}
\usepackage{amssymb, amsmath}
\allowdisplaybreaks[4]
\usepackage{color}
\usepackage{hyperref}

\let\pa=\partial
\let\f=\frac

\let\p=\partial
\let\om=\omega

\def\s{\sigma}

\def\oo{\infty}

\def\cB{{\mathcal B}}

\def\cL{{\mathcal L}}
\def\cM{{\mathcal M}}
\def\cN{{\mathcal N}}

\def\eqdef{\buildrel\hbox{\footnotesize def}\over =}
\def\Z{\mathop{\mathbb Z\kern 0pt}\nolimits}
\def\N{\mathop{\mathbb N\kern 0pt}\nolimits}
\def\Q{\mathop{\mathbb Q\kern 0pt}\nolimits}
\def\R{{\mathop{\mathbb R\kern 0pt}\nolimits}}
\def\T{{\mathop{\mathbb T\kern 0pt}\nolimits}}

\def\dive{{\mathop{\rm div}\nolimits}\,}

\def\tt{\langle t\rangle}

\def\e{\varepsilon}
\def\eqdefa{\buildrel\hbox{\footnotesize def}\over =}

\newcommand{\Rmnum}[1]{\uppercase\expandafter{\romannumeral #1} }

\newcommand{\beq}{\begin{equation}}
\newcommand{\eeq}{\end{equation}}
\newcommand{\ben}{\begin{eqnarray}}
\newcommand{\een}{\end{eqnarray}}
\newcommand{\beno}{\begin{eqnarray*}}
\newcommand{\eeno}{\end{eqnarray*}}

 \numberwithin{equation}{section}
\newcommand{\andf}{\quad\hbox{and}\quad}
\newcommand{\with}{\quad\hbox{with}\quad}

\newtheorem{defi}{Definition}[section]
\newtheorem{thm}{Theorem}[section]
\newtheorem{lem}{Lemma}[section]
\newtheorem{rmk}{Remark}[section]
\newtheorem{col}{Corollary}[section]
\newtheorem{prop}{Proposition}[section]

\begin{document}

\title[Nonlinear Evolution Toward the Linear Diffusive Profile]{Nonlinear Evolution Toward the Linear Diffusive Profile in the Presence of Couette Flow}

\author[N. Liu]{Ning Liu}
\address[N. Liu]{Academy of Mathematics $\&$ Systems Science, The Chinese Academy of Sciences, Beijing 100190, CHINA. }
\email{liuning16@mails.ucas.ac.cn}

\author[P. Zhang]{Ping Zhang}
\address[P. Zhang]{State Key Laboratory of Mathematical Sciences, Academy of Mathematics $\&$ Systems Science, The Chinese Academy of
	Sciences, Beijing 100190, China, and School of Mathematical Sciences,
	University of Chinese Academy of Sciences, Beijing 100049, China.} \email{zp@amss.ac.cn}

\author[W. Zhao]{Weiren Zhao}
\address[W. Zhao]{Department of Mathematics, New York University Abu Dhabi, Saadiyat Island, P.O. Box 129188, Abu Dhabi, United Arab Emirates.}
\email{zjzjzwr@126.com, wz19@nyu.edu}

\date{\today}

\begin{abstract}
In this paper, we investigate the long-time behavior of solutions to the two-dimensional Navier-Stokes equations with initial data evolving under the influence of the planar Couette flow. We focus on general perturbations, which may be large and of low regularity, including singular configurations such as point vortices, and show that the vorticity asymptotically approaches a constant multiple of the fundamental solution of the corresponding linearized vorticity equation after a long-time evolution determined by the relative Reynolds number.
\end{abstract}

\maketitle

\section{Introduction}
The dynamics of shear flows occupy a central role in the study of fluid mechanics, both for their mathematical elegance and their practical significance in understanding the onset of turbulence. Among these, the planar Couette flow $(y,0)$ stands out as a classical example. While Couette flow is known to be linearly stable for all Reynolds numbers, the nonlinear dynamics around this flow remain a rich source of subtle and complex behaviors, especially regarding the interplay of stability, enhanced dissipation, and the potential emergence of turbulence.

Physically, shear flows are often observed to transition to turbulent states even when linear theory predicts stability — a phenomenon sometimes described as ``subcritical transition." In many cases, this transition is triggered by the amplification of certain perturbations, including streamwise vortices, streaks, and coherent structures, which are sustained by the non-normal nature of the linearized operator and nonlinear feedback mechanisms. Understanding the long-time fate of perturbations near shear flows, particularly in the presence of viscosity, is key to bridge the gap between mathematical stability and physical instability.

In this paper, we consider the two-dimensional Navier-Stokes equations in $\R^2$:
\begin{equation*}
(NS)\quad \left\{\begin{array}{l}
\displaystyle \pa_t v-\nu\Delta v +v\cdot\nabla v+\nabla p=0, \qquad (t,x,y)\in\R^+\times\R^2, \\
\displaystyle \dive v = 0, \\
\displaystyle  v|_{t=0}=v_{0}(x,y),
\end{array}\right.
\end{equation*}
where $v(t,x,y)=(v^1,v^2)$ denotes the velocity field of the viscous fluid, $p$ denotes the scalar pressure function, and $\nu$ designates the viscosity coefficient.
Let $\omega_1\eqdefa \p_x v^2-\p_y v^1$ be the vorticity of the viscous fluid. Then it follows from $(NS)$ that
\begin{equation}\label{eqs:w_1}
\quad \left\{\begin{array}{l}
\displaystyle \pa_t \om_1-\nu\Delta \omega_1 +v\cdot\nabla \omega_1=0, \qquad (t,x,y)\in\R^+\times\R^2, \\
\displaystyle  v = \nabla^\bot \Delta^{-1} \omega_1 =(-\p_y, \p_x)\Delta^{-1} \omega_1, \\
\displaystyle  \omega_1|_{t=0}=\omega_{1,0}(x,y).
\end{array}\right.
\end{equation}

To study the asymptotic stability of $(NS)$ near  Coutte flow $(y,0),$ we denote
$u\eqdefa v-(y,0)$ to be the perturbation of the velocity, which satisfies
\begin{equation}\notag
\quad \left\{\begin{array}{l}
\displaystyle \pa_t u-\nu\Delta u +y\p_x u +(u^2,0) +u\cdot\nabla u+\nabla p=0, \qquad (t,x,y)\in\R^+\times\R^2, \\
\displaystyle \dive u = 0, \\
\displaystyle  u|_{t=0}=u_{0}(x,y).
\end{array}\right.
\end{equation}
Then the associated vorticity $\omega =\p_x u^2-\p_y u^1=\omega_1+1$ satisfies
\begin{equation}\label{eqs:w}
\quad \left\{\begin{array}{l}
\displaystyle \pa_t \om-\nu\Delta \omega +y\p_x \omega +u\cdot\nabla \omega=0, \qquad (t,x,y)\in\R^+\times\R^2, \\
\displaystyle  u = \nabla^\bot \Delta^{-1} \omega =(-\p_y, \p_x)\Delta^{-1} \omega, \\
\displaystyle  \omega|_{t=0}=\omega_{0}(x,y).
\end{array}\right.
\end{equation}

\smallskip

The stability of Couette flow has been a prominent topic in fluid mechanics since the seminal works of Kelvin \cite{Kelvin1887}, Rayleigh \cite{Rayleigh1880}, Orr \cite{Orr1907}, and Sommerfeld \cite{Sommerfeld1908}. While Couette flow is known to be linearly stable at all Reynolds numbers, recent studies in the physics literature \cite{bottin1998experimental,couliou2017growth,dauchot1995streamwise,zametaev2016evolution} have explored potential instability mechanisms, such as streamwise vortices and turbulence spots. Mathematically, nonlinear stability is closely related to the regularity and magnitude of the initial perturbations.

\smallskip

Our main goal is to establish precise asymptotic stability results for the vorticity field, without requiring the initial perturbation to be small or regular. We show that, after a time scale determined by the relative Reynolds number, the vorticity evolves toward a self-similar diffusive structure, characterized by the fundamental solution of the linearized equation. This reveals that, even in the nonlinear regime, viscous diffusion and the shear-induced mixing collaborate to suppress instability mechanisms and ensure convergence to a Gaussian profile determined by the conserved vorticity mass.

\smallskip

\subsection{Main results}
To study the long-time behavior of the vorticity to the 2D Navier-Stokes equations near Couette flow without any size restriction for the initial data, we shall first establish the global well-posedness of the system \eqref{eqs:w} with initial data in $L^1(\R^2)$, which can be proved by classical approaches used to the global well-posedness of \eqref{eqs:w_1} with initial data in $L^1$ (see for instance \cite{ben1994global,brezis1994remarks,kato1994navier}). Before introducing the mild solution of \eqref{eqs:w}, we first consider its linearized equation:
\beq\label{S1eq1}
\p_t \omega -\nu \Delta \omega +y \p_x \omega=0.
\eeq
It's easy to compute by using the Fourier transform that the fundamental solution of \eqref{S1eq1}  is
\begin{equation}\label{Green's function}
G_L(t,x,y)=
\f1{4\pi \nu t{\bigl(1+\f{t^2}{12}\bigr)^{\f12}}}\exp{\Bigl(-\f{x^2}{4\nu t(1+\f{t^2}{3})}-\f{\bigl((1+\f{t^2}3)y-\f{t}2 x\bigr)^2}{4\nu t(1+\f{t^2}3)(1+\f{t^2}{12})}}\Bigr),
\end{equation}
which solves the equation \eqref{S1eq1}  with initial data $\delta$.

\begin{defi}\label{S1def1}
{\sl
We call $\omega(t)$ a mild solution of the system \eqref{eqs:w} on $(0,T)$ if
\begin{equation}
\omega(t)=S(t-t_0)\omega(t_0) -\int_{t_0}^t S(t-s)\dive (u(s)\omega(s))ds, \qquad \forall\ 0<t_0<t<T,
\end{equation}
where $S(t)$ is the semigroup defined as
\begin{equation}\label{def:S(t)}
[S(t) f](x,y)=\int_{\R^2} G_L(t,x-x'.y-y')f(x',y')dx'dy',
\end{equation}
with kernel given by \eqref{Green's function}.
}\end{defi}

\smallskip

Our first result is concerned with the well-posedness of \eqref{eqs:w} in $L^1$:
\begin{thm}\label{Thm1}
{\sl For any $\omega_0\in L^1(\R^2)$,  the equations \eqref{eqs:w}  has a unique global mild  solution
$$
\omega\in C\bigl([0,+\oo);L^1(\R^2)\bigr)\cap C\bigl((0,+\oo);L^\oo(\R^2)\bigr),
$$
so that $\|\omega(t)\|_{L^1}$ is non-increasing, and for all $1\leq p\leq +\oo$,
\begin{equation}\label{eq1.5}
\|\omega(t)\|_{L^p}\leq C (\nu t)^{\f1p-1}\|\omega_0\|_{L^1}.
\end{equation}
}\end{thm}

\begin{rmk}
{\sl We remark that the proof of the decay estimate \eqref{eq1.5}  does not use any influence from Couette flow.
 Yet for small initial data in the sense that
$$
\|\omega_0\|_{L^1}\leq \e_0 \nu
$$
with some small universal constant $\e_0$, the proof of Theorem \ref{Thm1} will imply the better decay
$$
\|\omega(t)\|_{L^p}\leq C \left(\nu t\tt\right)^{\f1p-1} \|\omega_0\|_{L^1}.
$$
This decay rate coincides with the fundamental solution \eqref{Green's function}.
}
\end{rmk}

For the long-time behavior of the solution to the system \eqref{eqs:w}, we aim to demonstrate that the solution of \eqref{eqs:w} approaches the solution of \eqref{S1eq1}. Precisely, noticing that the mass of vorticity is preserved,
$$
M(\omega(t))\eqdef \int_{\R^2} \omega(t,x,y)dxdy=M(\omega_0),
$$
we are going  to show  that $\omega(t)\rightarrow M(\omega_0)G_L$, as $t\rightarrow+\oo$.
Inspired by \eqref{Green's function}, we introduce the new variables:
\begin{equation}\label{coordinate}
X=\f{x}{\sqrt{\nu t(1+\f{t^2}3)}} \andf Y=\f{(1+\f{t^2}3)y-\f{t}2 x}{\sqrt{\nu t (1+\f{t^2}3)(1+\f{t^2}{12})}},
\end{equation}
and write
\begin{equation}\label{def:Om}
\omega(t,x,y)=\f{\Omega(t,X,Y)}{\nu t \sqrt{1+\f{t^2}{12}}} .
\end{equation}
Then it follows from a
direct calculations that
\begin{equation}\label{eq:Om}
t\p_t \Omega = \cL_t \Omega +\cN_t\Omega,
\end{equation}
where we denote
\beq\label{eq:Oma}
\begin{split}
\Delta_t\eqdefa &\bigl(1+\f{t^2}3\bigr)^{-1}\Bigl(\p_X-\f{t}2\bigl({1+\f{t^2}{12}}\bigr)^{-\f12}\p_Y\Bigr)^2
+\bigl({1+\f{t^2}3}\bigr)\bigl({1+\f{t^2}{12}}\bigr)^{-1} \p_Y^2,  \\
\cL_t\eqdefa &\Delta_t+\f12\bigl(1+\f{t^2}3\bigr)^{-1}\Bigl(X-\f{t}2\bigl(1+\f{t^2}{12}\bigr)^{-\f12}Y\Bigr)
\Bigl(\p_X-\f{t}2\bigl(1+\f{t^2}{12}\bigr)^{-\f12}\p_Y\Bigr)\\
&+\f12{\bigl(1+\f{t^2}3\bigr)}{\bigl(1+\f{t^2}{12}\bigr)^{-1}} Y\p_Y+\f{t}4{\bigl(1+\f{t^2}{12}\bigr)^{-\f12}} \f{9+t^2}{3+t^2} (X\p_Y-Y\p_X) \\ &+\f{12+2t^2}{12+t^2},\\
\cN_t\Omega\eqdefa &\nu^{-1} \bigl(1+\f{t^2}{12}\bigr)^{-1} \Bigl( \p_Y \Delta_t^{-1}\Omega
\Bigl(\p_X-\f{t}2\bigl({1+\f{t^2}{12}}\bigr)^{-\f12}\p_Y\Bigr) \Omega\\
&\qquad\qquad\qquad \ \ - \Bigl(\p_X-\f{t}2\bigl({1+\f{t^2}{12}}\bigr)^{-\f12}\p_Y\Bigr)\Delta_t^{-1}\Omega \p_Y \Omega\Bigr)\\
=&\nu^{-1} \bigl(1+\f{t^2}{12}\bigr)^{-1} \bigl( \p_Y \Delta_t^{-1}\Omega\p_X \Omega
- \p_X\Delta_t^{-1}\Omega \p_Y \Omega\bigr).
\end{split}\eeq

The coefficients in equation \eqref{eq:Om} look complicated. Yet the formal limit of \eqref{eq:Om} as $t$ going to infinity
becomes
$$
t\p_t \Omega= 4\p_Y^2\Omega+2Y\p_Y\Omega+2\Omega+\f{\sqrt{3}}2 (X\p_Y-Y\p_X) \Omega.
$$
This is the equation of the forward semigroup generated by the `Fokker-Planck' operator:
\begin{equation}\label{def:Fokker-Planck operator}
\cL_\oo \eqdefa 4\p_Y^2+2Y\p_Y+2+\f{\sqrt{3}}2 (X\p_Y-Y\p_X).
\end{equation}
(It becomes the standard Fokker-Planck operator after commutating with the Gaussian weight.) When one considers the operator on Gaussian weighted space $L^2(G)$ with inner product $\bigl( f \big| g \bigr)_{L^2(G)}=\int_{\R^2} f \bar{g} {G}^{-1}dX dY$,
it is well-known that the spectral for $\cL_\oo$ on $L^2(G)$ is $\sigma(\cL_\oo)=-\f\N2$ and the semigroup satisfies that for any $0<\alpha<\f12$,
\begin{equation}\label{eq1.11}
\|e^{t\cL_\oo} f - M(f)G\|_{L^2(G)}\leq C_\alpha e^{-\alpha t}\|f\|_{L^2(G)},  \qquad \text{for}\quad t\geq0,
\end{equation}
where $M(f)=\int f(X,Y)\, dX\,dY$ denotes the mass, and $G(X,Y)=\f1{4\pi}e^{-\f{X^2+Y^2}4}$ stands for the Gaussian function which is the kernel of $\cL_\oo$.

To study the asymptotic stability of the solution to  \eqref{eqs:w}, we introduce some weighted spaces. Given any $m\geq 0$, we introduce the inner product via
$$
\bigl( f\big|g\bigr)_{L^2(m)}=\int_{\R^2} f(X,Y) g(X,Y) \langle X,Y\rangle^{2m} \,dX\,dY,
$$
where $\langle X,Y\rangle\eqdefa \bigl(1+|X|^2+|Y|^2\bigr)^\f12$. When $m>1$, the space $L^2(m)$ can be embedded into $L^1(\R^2)$.
When $m\geq 3$, we shall prove a similar version of  \eqref{eq1.11} in the framework of $L^2(m)$:
\begin{equation}\label{eq1.11a}
\|e^{t\cL_\oo}f-M(f)G\|_{L^2(m)}\leq C_m e^{-\f{t}2}\|f-M(f)G\|_{L^2(m)}.
\end{equation}

Our second result is concerned with the estimate $\Omega(t)$ in the weighted space $L^2(m)$:

\begin{thm} \label{Thm2}
{\sl Let $1<m\leq m_0,$  $\delta>0$, and $\Omega(1)\in L^2(m_0)$. Then the unique solution $\Omega$ of \eqref{eq:Om} satisfies
\begin{equation}\label{eq1.12}
\|\Omega(t)\|_{L^2(m)}\leq C_{m,m_0,\delta} \bigl( 1 +{\nu}^{-1}{\|\Omega(1)\|_{L^2(m_0)}}\bigr)^{m} \tt^{\f12 +\f{m}{m_0}\left(\f12+\delta\right)}\|\Omega(1)\|_{L^2(m_0)}.
\end{equation}
Moreover, given $0<\sigma<\f12$, for any $|(a,b)|\leq 2$ and $t\geq T_0$ for some universal time $T_0$, there holds
\begin{equation}
\begin{aligned}\label{eq1.14}
\|\p_X^a\p_Y^b\Omega(t)\|_{L^2(m)}
\leq &C_{m,m_0,\delta,\sigma} \bigl( 1 +{\nu}^{-1}{\|\Omega(1)\|_{L^2(m_0)}}\bigr)^{m+\f{2(3a+b)(m+1)}{1+2\sigma}} \\
&\qquad\times\tt^{\f12 +\f{m(1+2\delta)}{m_0}\left(\f12+\f{3a+b}{1+2\sigma}\right)}\|\Omega(1)\|_{L^2(m_0)}.
\end{aligned}
\end{equation}
}\end{thm}

\begin{rmk}
{\sl When we take $m_0$ to be large enough, $\delta$ to be small enough and $\sigma$ close enough to $\f12$, these estimates show that the $H^2(m)$ norm of $\Omega(t)$ is bounded by the nearly $\tt^\f12$ time growth multiplying some power of relative Reynolds number $\nu^{-1}\|\Omega(1)\|_{L^2(m_0)}$.
}\end{rmk}

\begin{rmk}
{\sl In \eqref{eq1.14}, we only present the estimates of the higher regularities for large time $t\geq T_0$, which is enough to study the long-time behavior. It will be much easier to gain regularities for finite time $t\leq T_0$, since the coefficients in $\cL_t$ do not vanish.
}\end{rmk}

With such a time growth estimate for $\Omega(t),$ our third result is to prove that the solution of \eqref{eq:Om} will approach the equilibrium:
\begin{thm}\label{Thm3}
{\sl Let $3\leq m,$  $m_0>\f72m$, and $\Omega(1)\in L^2(m_0)$. For any $\epsilon>0$, the solution $\Omega(t)$ of \eqref{eq:Om} satisfies
\begin{equation}\label{eq1.13}
\|{\Omega}(t)-M(\Omega(1))G\|_{L^2(m)} \leq C_{m,m_0,\epsilon} t^{-\f12} \bigl(  1 +\nu^{-1}\|\Omega(1)\|_{L^2(m_0)}\bigr)^{\f{8(m+1)}{1-\f{7m}{2m_0}}-1+\epsilon}\|\Omega(1)\|_{L^2(m_0)},
\end{equation}
for all $t\geq T_1=C_{m,m_0,\epsilon} \bigl( 1 +\nu^{-1}\|\Omega(1)\|_{L^2(m_0)}\bigr)^{\f{7(m+1)}{1-\f{7m}{2m_0}}+\epsilon}$.
}
\end{thm}

\begin{rmk}
{\sl Noticing that $L^2(m)\hookrightarrow L^1$ for $m>1$, we deduce from \eqref{coordinate} and \eqref{eq1.13} that
$$
\|\omega(t)-M(\omega(1))G_L(t) \|_{L^1} =O(t^{-\f12}),
$$
which means that the solution of \eqref{eqs:w} returns to the mass multiplying $G_L$ in $L^1$ with the optimal decay rate $t^{-\f12}$.
}\end{rmk}

\begin{rmk}\label{rmk1.4}
{\sl Although we prove the convergence rate in \eqref{eq1.13} after $T_1$, this estimate will only be powerful  after $T_2=C_{m,m_0,\epsilon} \bigl( 1 +\nu^{-1}\|\Omega(1)\|_{L^2(m_0)}\bigr)^{\f{16(m+1)}{1-\f{7m}{2m_0}}-2+2\epsilon}$ when the right hand side of \eqref{eq1.13} becomes smaller than $\|\Omega(1)\|_{L^2(m_0)}$. When we take $m_0$ large enough and $m=3$, this time is nearly $T_2\approx \bigl( 1 +\nu^{-1}\|\Omega(1)\|_{L^2(m_0)}\bigr)^{62}$. We believe that such  a power is far from being sharp.
}
\end{rmk}

\begin{rmk}
{\sl In both Theorem \ref{Thm2} and Theorem \ref{Thm3}, we always assume the  initial data belongs to $L^2(m_0)$ at time $t=1$, because the change of coordinate \eqref{coordinate} is singular near $t=0$. If we assume $\Omega(0)\in L^2(m_0)$, one may expect $\omega_0$ is a Dirac measure supported at the origin from \cite{gallay2005global}, since Couette flow needs time to influence the perturbation. Therefore, as we are focusing on the long-time behavior, putting initial data at a positive time can include more general cases.
}\end{rmk}

\subsection{Previous literature and some comparison}
We now review some related literature and point out the connections with this work.

\begin{itemize}
\item {\it Global well-posedness of the 2-D Navier-Stokes equations:}
\end{itemize}

The global well-posedness of $(NS)$ with initial velocity field in $L^2(\R^2)$ is well-known since Leray's celebrated work \cite{leray1933etude}. In \cite{kato1984strong}, Kato proved that d-dimensional Navier-Stokes equations on $\R^d(\R^d)$ are locally well-posed for arbitrary $L^d$ initial data, and globally well-posed for sufficiently small initial data. These results exploit the scaling invariant of $(NS)$, given by $u(t,x,y)\rightarrow \lambda u(\lambda^2 t,\lambda x,\lambda y)$, which makes $L^\infty(\R^+; L^d(\R^d))$  scaling invariant functional space.

 One speciality of $2D$ Navier-Stokes is that its vorticity, a scalar quantity, satisfies the transport-diffusion equation \eqref{eqs:w_1}. A natural scale-invariant space for the vorticity is $L^\infty(\R^+;L^1(\R^2))$. The Cauchy problem of \eqref{eqs:w_1} with
  initial vorticity in $L^1$ was studied in \cite{ben1994global,brezis1994remarks,kato1994navier}, where results similar to Leray's and Kato's theorems for velocity were obtained. However, $L^1(\R^2)$ is not the largest scaling invariant space for vorticity: Cottet in \cite{cottet1986equations} and Giga, Miyakawa and Osada in \cite{giga1988two} independently established the existence of solutions to \eqref{eqs:w_1} with finite measure initial data in $\cM(\R^2)$. The uniqueness in $\cM(\R^2)$ is more challenging; following partial results in \cite{giga1988two,kato1994navier,gallay2005global,gallagher2005dirac}, Gallagher and Gallay proved the uniqueness in $\cM(\R^2)$ in \cite{gallagher2005uniqueness}.

In this paper, we follow the approach of  \cite{ben1994global,brezis1994remarks,kato1994navier} to establish Theorem \ref{Thm1}. First, by using Kato's method and Brezis's observations, we shall prove local existence and uniqueness of solution to \eqref{eqs:w} with $L^1$ initial data. Then, a standard argument based on the maximum principle and the conservation of mass $M(\omega)$ implies that the $L^1$ norm of the vorticity is non-increasing with respect to time, which leads to a global existence solution to \eqref{eqs:w}. The estimate \eqref{eq1.5} follows from  Nash-Moser iteration, which exploits the divergence-free structure to cancel all the nonlinear terms.

An interesting question is whether \eqref{eqs:w} is well-posed in the space $\cM(\R^2)$ of finite measures. With slight modifications of Theorem \ref{Thm1}, one can easily obtain the existence and uniqueness for initial data with some smallness condition on the atomic part, as in \cite{kato1994navier}. However, the well-posedness (especially, uniqueness) for \eqref{eqs:w} with general initial data in $\cM(\R^2)$ seems more complicated. We conjecture that such a result holds, as Couette flow is unlikely to have a significant impact over very short time scales.

\begin{itemize}
\item {\it Long-time behaviour of the 2-D Navier-Stokes equations:}
\end{itemize}

Many previous studies have focused on proving that the solutions of \eqref{eqs:w_1} in $L^1$ asymptotically behave like the solutions of the linear heat equation with the same initial data. Giga and Kambe \cite{giga1988large} established such a result, and another approach using ideas from dynamical systems can be found in \cite{gallay2002invariant}. In \cite{carpio1994asymptotic}, Carpio observed a deep connection between this asymptotic behavior and the uniqueness of the fundamental solution for \eqref{eqs:w_1}. Finally, Gallay and Wayne \cite{gallay2005global} proved this approximation for arbitrary initial data in $L^1$.

The approach in \cite{gallay2005global} begins by introducing the self-similar variables and transforming \eqref{eqs:w_1} into the following form:
\begin{equation}\label{eq1.15}
t\p_t \Omega_1 +V\cdot \nabla_\xi \Omega_1=\cL \Omega_1, \with \cL=\Delta_\xi+\f{\xi}2\cdot \nabla_\xi +1 \andf V=\nabla^\perp_\xi \Delta_\xi^{-1}\Omega_1.
\end{equation}
Noticing this is an autonomous equation with the Gaussian function $G$ as a steady state, the authors constructed a pair of Lyapunov functions to establish
$$
\lim_{t\rightarrow +\oo}\|\Omega_1(t)-M(\Omega_1) G\|_{L^1}=0.
$$
Moreover, they showed that if $\Omega_1(1)$ belongs to $L^2(m)$ for some $m>2$, then the Oseen vortices $G$ attract all the solutions in $L^2(m)$ at the rate $\|\Omega_1(t)-M(\Omega_1) G\|_{L^2(m)}=O(t^{-\f12})$.

The primary goal of this work is to extend the results of \cite{gallay2005global} to the case of \eqref{eqs:w} with Couette flow. As a first step, we construct a new 'self-similar' variable based on the explicit formula of the fundamental solution $G_L$ of the linear equation \eqref{S1eq1}, which captures the dominant scale of Couette flow at large times. Compared with \eqref{eq1.15}, the equation \eqref{eq:Om} in our setting contains many complicated coefficients depending on time, which makes it very difficult to obtain some results for general $L^1$ initial data with the Lyapunov method.

On the other hand, in \eqref{eq:Om}, the coefficient before the nonlinear parts is of order $\tt^{-2}$, which introduces some smallness when $t$ is sufficiently large. This key observation marks a fundamental difference between our Theorem \ref{Thm3} and the results of \cite{gallay2005global}: in our case, the estimate \eqref{eq1.13} is quantitative, meaning it holds explicitly after a long but finite time $T_2$. In contrast, in \cite{gallay2005global}, the $t^{-\f12}$ decay emerges only after some unspecified time when $\|\Omega_1(t)-M(\Omega_1) G\|_{L^2(m)}$ becomes sufficiently small, without an explicit estimate for when this occurs.

\begin{itemize}
\item {\it Stability near 2-D Couette flow:}
\end{itemize}

During recent years, several stability mechanisms have been established to explain the asymptotic stability of \eqref{eqs:w}, including the enhanced dissipation and the inviscid damping, see \cite{albritton2022enhanced, BM2015, CWZZ2025, IJ2018, IJ2020, MasmoudiZhao2020, zhao2025} and references therein.
Mathematically, Bedrossian, Germain and Masmoudi \cite{BGM2017} formulated the following so-called transition threshold problem for \eqref{eqs:w} with small viscosity $\nu>0$,

{\it
Given a norm $\|\cdot\|_X$, find a $\beta=\beta(X)$ so that
\begin{itemize}
\item[] $\|u_{0}\|_{X}\leq \nu^\beta$ $\Rightarrow$ stability, enhanced dissipation and inviscid damping,
\item[] $\|u_{0}\|_X \gg \nu^\beta$ $\Rightarrow$ instability.
\end{itemize}
We then call $\beta$ the transition threshold.
}

For 2-D Couette flow on the domain $\T_x\times\R_y$, the following important results about the transition threshold of 2-D Couette flow are known:
\begin{itemize}
\item If $X$ is Gevrey class $2_-$, then \cite{BMV2016} shows $\beta\leq 0$, and \cite{DM2023} showed $\beta\geq 0$.
\item If $X$ is Sobolev space $H^{\log}_xL^2_y$, then \cite{BVW2018,MasmoudiZhao2020cpde} showed $\beta\leq \f12$ and \cite{LiMasmoudiZhao2022critical} showed $\beta\geq \f12$.
\item If $X$ is Sobolev space $H^\sigma$ ($\sigma\geq 2$), then \cite{MasmoudiZhao2019,wei2023nonlinear} showed $\beta\leq \f13$.
\item If $X$ is Gevrey class $\f1s$ with $s\in[0,\f12]$, then \cite{LMZ2022G} showed $\beta\leq \f{1-2s}{3(1-s)}$.
\end{itemize}

Very recent works \cite{wang2024transition}, \cite{arbon2024}, and \cite{li2025stability} have extended the transition threshold problem of \eqref{eqs:w} to the whole plane $\R^2$. In particular, Li, the first author and the third author proved in \cite{li2025stability} that under smallness of (where $m>\f12$ and $\epsilon>0$)
\begin{equation}\label{smallness1}
\|\langle D_x\rangle^m \langle \f1{D_x}\rangle^\epsilon \omega_{0} \|_{L^2}\leq \e \nu^\f12,
\end{equation}
then the solution of \eqref{eqs:w} exhibits enhanced dissipation and inviscid damping. Moreover, they established the sharp transition threshold $\f13$ in certain Fourier-weighted Sobolev spaces.

Our Theorem \ref{Thm3} demonstrates that the solution asymptotically behaves like the fundamental solution \eqref{Green's function}, which naturally satisfies enhanced dissipation and inviscid damping. This suggests a deep connection between our results and the transition threshold. For example, if we want the estimate \eqref{eq1.13} of Theorem \ref{Thm3} to cover all the time after a large time (the $T_2$ in Remark \ref{rmk1.4}) independent of $\nu$, we need the initial data to be $\nu$-small, i.e. $\|\Omega(1)\|_{L^2(m_0)}\leq C\nu$. Expressing this in the original variables via \eqref{coordinate}, we obtain
$$
\|\omega(1)\|_{L^2}= \sqrt[4]{12} \nu^{-\f12} \|\Omega(1)\|_{L^2}\leq C \nu^\f12.
$$
Such smallness condition coincides with the transition threshold $\f12$ in nearly $L^2$ spaces, see \cite{BVW2018,MasmoudiZhao2020cpde} and Theorem 1.1 of \cite{li2025stability}.

\begin{itemize}
\item {\it Open problems:}
\end{itemize}

As mentioned earlier, we assume $\Omega\in L^2(m_0)$ in this paper to establish the asymptotic behaviour of the solutions. An important open problem is to study the long-time behaviour of general $L^1$ mild solutions given by Theorem \ref{Thm1}.  Due to the non-autonomous nature of \eqref{eq:Om}, the Lyapunov function approach used in \cite{gallay2005global} cannot be directly applied.

Another interesting question is to improve the quantitative bounds in Theorem \ref{Thm2} and Theorem \ref{Thm3}. In this paper, we provide an upper bound of the growth rate, which is nearly $\tt^\f12$ up to a large time scale $T_1$ and followed by a decay to the equilibrium at the rate of $t^{-\f12}$. Since the $\tt^\f12$ time growth originates from applying the Nash-Moser iteration to estimate the $L^2$ norm, improving this bound appears to be a difficult task.  Understanding whether a sharper estimate can be achieved remains an open and significant question.

Moreover, the long-time convergence results focus on the final equilibrium regime, but many physically relevant flows exhibit rich and complex dynamics at intermediate times. In particular, the presence of localized transient chaotic behavior — characterized by the temporary growth and interaction of coherent vortical structures before eventual decay — is a hallmark of two-dimensional turbulence. Understanding the precise mechanisms that govern this transient phase, especially the formation, stretching, and dissipation of vorticity patches, remains a significant open problem. A more complete mathematical description of these dynamics would bridge the gap between the linear stability theory and the observed turbulent behavior in shear-dominated flows.

Very recently, Donati and Gallay \cite{DG2025} studied the evolution of a point vortex advected by a smooth, divergence-free velocity field in two space dimensions. Although Couette flow does not satisfies the assumption there (since it is not bounded), one can expect that the method therein can describe the evolution of \eqref{eqs:w} with initial data $\omega_{0}$ to be some Dirac measure of strength $\Gamma$, before a timespan slightly smaller than $\nu^{-1}\Gamma$. It is also interesting to study whether we improve our long-time behaviour result to such a time scale for this kind of special initial data.

\subsection{Outline of the paper and Notations}
We now sketch the structure of this paper.

In section \ref{section 2}, we present the proof of Theorem \ref{Thm1} by first proving the local well-posedness and then showing the solution is global.

In section \ref{section 3}, we present the proof of Theorem \ref{Thm2} by the energy method.

In section \ref{section 4}, we first study the semi-group generated by $\cL_\oo$, and then present the proof of Theorem \ref{Thm3}.

Let us end this section with some notations that will be used
throughout this paper.

\noindent{\bf Notations:}   For $a\lesssim b$,
we mean that there is a uniform constant $C,$ which may be different in each occurrence, such that $a\leq Cb$. In the estimates of weighted $L^2(m)$ norms, we shall denote $a(X,Y)=\sqrt{|X|^2+|Y|^2}$ and $b(X,Y)=\langle X,Y\rangle=\sqrt{1+|X|^2+|Y|^2}$.
Given a function $f(X,Y)$ on $\R_X\times\R_Y$, we shall denote $\hat{f}(\xi,\eta)$ the Fourier transform of $f$ with respect to both $X$ and $Y$ variables.
 Finally, we denote $L^r_T(L^p)$ the space $L^r([0,T];
L^p(\R^2)),$ and denote $L^r_{[T_1,T_2]}(L^p)$ the space $L^r([T_1,T_2];
L^p(\R^2))$. We will denote $\int f$ to be the integral $\int_{\R^2} f(x,y)dxdy$ or $\int_{\R^2} f(X,Y)dXdY$, depending on the variables of $f$.

\section{Global well-posedness in $L^1$}\label{section 2}

In this section, we study the well-posedness theory of the system \eqref{eqs:w}.
In subsection \ref{subsection2.1}, we shall use a fixed point argument to build a local solution of \eqref{eqs:w} in a Kato-type space, and then use Brezis's argument to show the uniqueness of such a solution in a large class of $L^1$ mild solutions.
In subsection \ref{subsection 2.2}, we shall use the maximal principle to show the $L^1$ norm of the solution constructed in subsection \ref{subsection2.1} is non-decreasing with respect to the time $t$ and derive \eqref{eq1.5} via classical Nash-Moser iteration.

\subsection{Local wellposedness}\label{subsection2.1}
Given initial data $\omega_0$, we can equivalently reformulate the $\omega$ equation of \eqref{eqs:w} to the following integral equation:
\begin{equation}\label{eq:integral eq:w}
\omega(t)=S(t)\omega_0 -\int_0^t S(t-s)\dive (u(s)\omega(s))ds,
\end{equation}
where $S(t)$ is the linear semigroup defined in \eqref{def:S(t)}.

\begin{lem}\label{S2lem1}
{\sl The family $S(t)$ given by \eqref{def:S(t)} is a strongly continuous semigroup of bounded linear operators in $L^p(\R^2)$ for any $p\in [1,\oo)$. Moreover, if $1\leq p\leq q\leq \oo$, there hold:
\begin{subequations} \label{S2eq1}
\begin{gather}
\|S(t)f\|_{L^q}\leq C \left(\nu t\tt\right)^{-\f1p+\f1q} \|f\|_{L^p},\label{eq2.3}\\
\|S(t)\dive f\|_{L^q}\leq C (\nu t)^{-\f12}\left(\nu t\tt\right)^{-\f1p+\f1q} \|f\|_{L^p}.\label{eq2.4}
\end{gather}
\end{subequations}
}\end{lem}
\begin{proof} It is easy to observe that \eqref{S2eq1}
follow from Young's inequality and
$$
\|G_L(t)\|_{L^r}\lesssim \left(\nu t\tt\right)^{-1+\f1r} \andf
\|\nabla G_L(t)\|_{L^r}\lesssim (\nu t)^{-\f12} \left(\nu t\tt\right)^{-1+\f1r},
$$
with $r$ satisfying $\f1q+1=\f1p+\f1r$.
\end{proof}

\begin{prop}\label{S2prop1}
{\sl Let $\omega\in L^1(\R^2)$, there exists a positive time $T=T(\omega_0)$ such that the integral equation \eqref{eq:integral eq:w} has a unique mild solution in the sense of Definition \ref{S1def1} with
$$
\omega\in C([0,T];L^1(\R^2))\cap C((0,T];L^\oo(\R^2)).
$$
Moreover, the solution satisfies
\begin{equation}\label{eq2.5}
\lim_{t\rightarrow 0_+}\Bigl( t^{1-\f1p}\|\omega(t)\|_{L^p}\Bigr)=0, \qquad \text{for} \quad 1<p\leq \oo.
\end{equation}
}\end{prop}
\begin{proof}
For any  $T>0$, we introduce the following Kato-type function space
$$
X_T \eqdefa \bigl\{\ \omega\in C((0,T];L^\f43(\R^2))\ \big|\  \|\omega\|_{X_T}\eqdef \sup_{0<t\leq T} \left(\nu t\tt\right)^\f14\|\omega(t)\|_{L^\f43}<\oo \ \bigr\}
$$
 Since $L^1(\R^2)\cap L^\f43(\R^2)$ is dense in $L^1(\R^2)$, we deduce from \eqref{eq2.3}
 that for any $T<\infty,$
\beq\label{S2eq2}
\begin{split}
&\|S(t)\omega_0\|_{X_T}=\sup_{0<t\leq T}\Bigl(\left(\nu t\tt\right)^\f14\|S(t)\omega_0\|_{L^\f43}\Bigr) \leq C\|\omega_0\|_{L^1} \andf \lim_{T\to 0_+} \|S(t)\omega_0\|_{X_T}=0.
\end{split}
\eeq

Given $\omega_1, \omega_2\in X_T$ and $p\in[1,2)$, we can define the bilinear map $\cB(\omega_1,\omega_2)$ via
$$
\cB(\omega_1, \omega_2)(t)=-\int_0^t S(t-s)\dive (\nabla^\perp \Delta^{-1}\omega_1(s)\cdot\omega_2(s))\,ds.
$$
Then the equation \eqref{eq:integral eq:w} can reformulated as
\beq\label{S2eq3}
\omega(t)=S(t)\omega_0 +\cB(\omega, \omega)(t).
\eeq

Observing that  $\|\nabla^{\perp}\Delta^{-1}\omega\|_{L^4}\lesssim \|\omega\|_{L^\f43}$, we get,
by using \eqref{eq2.4}, that
\begin{align*}
&\left(\nu t\tt\right)^{1-\f1p}\|\cB(\omega_1, \omega_2)(t)\|_{L^p} \\
&\lesssim \left(\nu t\tt\right)^{1-\f1p}\int_0^t (\nu(t-s))^{-\f12}\left(\nu(t-s)\langle t-s\rangle\right)^{-1+\f1p} \|\nabla^\perp \Delta^{-1}\omega_1(s)\cdot \omega_2(s)\|_{L^1} ds\\
&\lesssim \left(\nu t\tt\right)^{1-\f1p}\int_0^t (\nu(t-s))^{-\f12}\left(\nu(t-s)\langle t-s\rangle\right)^{-1+\f1p} \|\nabla^\perp \Delta^{-1}\omega_1(s)\|_{L^4}\|\omega_2(s)\|_{L^\f43} ds\\
&\lesssim \left(\nu t\tt\right)^{1-\f1p}\int_0^t (\nu(t-s))^{-\f12}\left(\nu(t-s)\langle t-s\rangle\right)^{-1+\f1p} \|\omega_1(s)\|_{L^\f43}
\|\omega_2(s)\|_{L^\f43} ds \\
&\lesssim \left(\nu t\tt\right)^{1-\f1p}\int_0^t (\nu(t-s))^{-\f12}\left(\nu(t-s)\langle t-s\rangle\right)^{-1+\f1p} (\nu s\langle s \rangle)^{-\f12} ds\|\omega_1\|_{X_T}\|\omega_2\|_{X_T}\\
&\lesssim \nu^{-1} t^{1-\f1p}\int_0^t (t-s)^{\f1p-\f32} s^{-\f12} ds\|\omega_1\|_{X_T}\|\omega_2\|_{X_T},
\end{align*}
where we used the fact that $\tt^{1-\f1p}\cdot \langle t-s\rangle^{-1+\f1p}\cdot\langle s \rangle^{-\f12}\lesssim 1$ if $p\leq 2.$ This
 implies
\begin{equation}\label{eq2.6}
\left(\nu t\tt\right)^{1-\f1p}\|\cB(\omega_1, \omega_2)(t)\|_{L^p}\leq C_1\nu^{-1}\|\omega_1\|_{X_T}\|\omega_2\|_{X_T}.
\end{equation}
By taking $p=\f43$ in \eqref{eq2.6}, we obtain
\beno \|\cB(\omega_1, \omega_2)\|_{X_T}\leq C_1\nu^{-1}\|\omega_1\|_{X_T}\|\omega_2\|_{X_T}.
\eeno
While by virtue of \eqref{S2eq2}, we can take $T_0$ to be so small that
\beno
\|S(t)\omega_0\|_{X_{T_0}}\leq \frac{\nu}{4C_1},
\eeno
then we deduce from Lemma 5.5 of \cite{bcd} that  \eqref{eq:integral eq:w} has a unique solution $\omega\in X_{T_0}.$ Furthermore,
for any $T\leq T_0,$ we deduce from \eqref{S2eq3} and \eqref{eq2.6} that
\beno
\|\om\|_{X_T}\leq \|S(t)\om_0\|_{X_T}+C_1\nu^{-1}\|\omega\|_{X_T}^2,
\eeno
which along with \eqref{S2eq2} ensures that there exits $T_1\leq T_0$ so that for $T\leq T_1,$
\beq\label{S2eq4}
\|\om\|_{X_T}\leq 2\|S(t)\om_0\|_{X_T} \andf \lim_{T\to 0_+}\|\om\|_{X_T}=0.
\eeq
This leads to \eqref{eq2.5} for $p=\f43$.
Then, we get, by taking $p=1$ in \eqref{eq2.6}, that
$$ \sup_{t\in]0,T]} \|\cB(\omega,\om)\|_{L^1}\leq C_1\nu^{-1}\|\omega\|_{X_T}^2\rightarrow 0, \qquad \text{as}\quad T\rightarrow 0_+.
$$
 Observing that $S(t)\omega_0\in C([0,T];L^1(\R^2))$, we thus proved that $\omega\in C([0,T];L^1(\R^2))$.

  Finally, to prove that $\omega\in C((0,T];L^p(\R^2))$ for $p\in (1,\oo]$ and \eqref{eq2.5} holds, for $p\in(1,\oo]$, we denote
$$
M_p(T)\eqdefa\sup_{t\in(0,T]} \left(\nu t\tt\right)^{1-\f1p}\|S(t)\omega_0\|_{L^p} \andf
N_p(T)\eqdefa \sup_{t\in(0,T]} \left(\nu t\tt\right)^{1-\f1p}\|\omega(t)\|_{L^p}.
$$
Along the same line to the derivation of \eqref{S2eq2},
we find that $M_p(T)\lesssim \|\omega_0\|_{L^1}$ and $M_p(T)\rightarrow 0$, as $T\rightarrow 0_+$. To control $N_p(T)$, we
 separate  the integral in $\cB(\omega,\om)$ into two parts and use
$$
\|\nabla^{\perp}\Delta^{-1}\omega \omega\|_{L^r}
\lesssim \|\nabla^{\perp}\Delta^{-1}\omega \|_{L^s}\|\omega\|_{L^{q_2}}
\lesssim \|\omega\|_{L^{q_1}} \|\omega\|_{L^{q_2}}, \with s=\f{2q_1}{2-q_1},
$$
where $\f1r=\f1{q_1}+\f1{q_2}-\f12$, $\f43\leq q_1<2$ and $\f43\leq q_2\leq \oo$. So that we deduce from Lemma \ref{S2lem1}
that
\begin{align*}
\|\omega(t)\|_{L^p} \leq &\|S(t)\omega\|_{L^p}+C\int_0^{\f{t}2}(\nu (t-s))^{-\f12} (\nu (t-s)\langle t-s\rangle)^{\f1p-\f2q+\f12} \|\omega(s)\|_{L^q}^2 ds \\
&+C \int_{\f{t}2}^t (\nu (t-s))^{-\f12} (\nu (t-s)\langle t-s\rangle)^{\f1p-\f1{q_1}-\f1{q_2}+\f12} \|\omega(s)\|_{L^{q_1}}\|\omega(s)\|_{L^{q_2}} ds \\
\leq& \left(\nu t\tt\right)^{\f1p-1}M_p(T)+C\int_0^{\f{t}2}(\nu(t-s))^{\f1p-\f2q}\langle t-s\rangle^{\f1p-\f2q+\f12} (\nu s\langle s\rangle)^{\f2q-2}ds N_q^2(T)\\
&+C\int_{\f{t}2}^t(\nu(t-s))^{\f1p-\f1{q_1}-\f1{q_2}}\langle t-s\rangle^{\f1p-\f1{q_1}-\f1{q_2}+\f12} (\nu s\langle s\rangle)^{\f1{q_1}+\f1{q_2}-2}ds N_{q_1}(T)N_{q_2}(T)\\
\leq &(\nu t \tt)^{\f1p-1}\Bigl(M_p(T)+ C\nu^{-1} \tt^{\f32-\f2q}N_q^2(T)+ C\nu^{-1}\tt^{-1-\f1p+\f1{q_1}+\f1{q_2}} N_{p_1}(T)N_{p_2}(T) \Bigr),
\end{align*}
where $p\in [1,\oo]$, $q,q_1\in[\f43,2)$ and $q_2\in[\f43,\oo]$ are assumed to satisfy that
$$
\f12\leq \f2q-\f1p \andf \f12\leq \f1{q_1}+\f1{q_2}-\f1p<1.
$$
In particular, by taking $q=\f43$ and multiplying the inequality by $\left(\nu t\tt\right)^{1-\f1p}$ and then taking the supremum of
the resulting inequality over $t\in [0,T],$ we  obtain
\begin{equation}\label{eq2.7}
N_p(T)\leq M_p(T) +C\nu^{-1} N_{\f43}(T)^2+ C\nu^{-1} N_{q_1}(T) N_{q_2}(T).
\end{equation}
If we choose $q_1=q_2=\f43$, we deduce from \eqref{eq2.7} that for $p\in [1,2),$ $N_p(T)\lesssim \|\omega_0\|_{L^1}$ and $N_p(T)\rightarrow 0$, as $T\rightarrow 0$. Similarly, by taking $q_1=q_2$ to be sufficiently close to $2,$ we obtain a similar result for any $p<\oo$. Finally, by choosing $q_1=\f32$ and $q_2=4$, we deduce from \eqref{eq2.7} that $N_\oo(T)\lesssim \|\omega_0\|_{L^1}$ and $N_\oo(T)\rightarrow 0_+$, as $T\rightarrow 0$.
This completes the proof of \eqref{eq2.5}.

It remains to prove the uniqueness part of Proposition \ref{S2prop1}. It follows from  the fixed point argument in $X_T$ that
 the solution just constructed is unique  in $X_T$ which satisfies $\lim_{t\rightarrow 0_+} \bigl(t^\f14\|\omega(t)\|_{L^\f43}\bigr)=0$.
   Next, we shall use a nice argument of Brezis in \cite{brezis1994remarks} to prove the uniqueness part of Proposition \ref{S2prop1}. Let us
    assume that $\omega\in C([0,T];L^1(R^2))\cap C((0,T];L^\oo(\R^2))$ is a mild solution of \eqref{eqs:w} in the sense
     of Definition \ref{S1def1}. Due to $\omega\in C([0,T];L^1(R^2)),$ the set $K\eqdefa \left\{ \omega(t) | t\in [0,T] \right\}$ is compact in $L^1(\R^2)$. Recalling that $(t,f) \rightarrow \|S(t)f\|_{X_T}$ is a continuous map on $[0,T]\times L^1$, the fixed point argument allows us to construct a local solution of \eqref{eqs:w} in $X_{\tilde{T}}$ for all initial data $\tilde{\omega}_0\in K$ with a common existence time $\tilde{T}>0$. (Without loss of generality, we take $\tilde{T}\leq \f{T}2$) We denote the  solution of \eqref{eqs:w} with initial data $\tilde{\omega}_0$ thus constructed by $\tilde{\omega}(t)\eqdefa \Sigma(t)\tilde{\omega}_0.$  For any $t_0\in (0,\tilde{T}]$, since $\|\omega(t)\|_{L^\f43}$ is bounded around $t_0$, it's easy to deduce from $\lim_{t\rightarrow t_0}\bigl((t-t_0)^\f14\|\omega(t)\|_{L^\f43}\bigr)=0$ that
$$
\omega(t)=\Sigma(t-t_0)\omega(t_0),\qquad t\in [t_0,t_0+\tilde{T}].
$$
Now, for any fixed $t\in (0,\tilde{T}]$, it follows from the $L^1$ boundedness of $\Sigma(t-t_0)$ that
$$
\|\Sigma(t-t_0)\omega(t_0)-\Sigma(t-t_0)\omega(0)\|_{L^1}\leq C\|\omega(t_0)-\omega(0)\|_{L^1}\rightarrow 0, \qquad \text{as}\quad t\rightarrow 0.
$$
Also, from the continuous dependence of time, $\|\Sigma(t-t_0)\omega(0)-\Sigma(t)\omega(0)\|_{L^1}\rightarrow0$, as $t\rightarrow 0$. Therefore, we take $t_0\rightarrow 0$ to obtain $\omega(t)=\Sigma(t)\omega_0$ for all $t\in[0,T]$. This means the mild solution coincides with the unique solution constructed by the fixed point argument.
\end{proof}

\begin{rmk}
{\sl
If the initial data is so small that $\|\omega_0\|_{L^1}\leq \e_0\nu$ for some universal small constant $\e_0$, $\|S(t)\omega_0\|_{X_T}\lesssim \e_0\nu$ will enable us to take $T$ to be arbitrarily large. In such a case, one can deduce from the estimates near \eqref{eq2.7} that
\begin{equation}\label{rmk2.1eq}
\sup_{t>0} \left(\nu t\tt\right)^{1-\f1p}\|\omega(t)\|_{L^p}\lesssim \|\omega_0\|_{L^1}, \qquad p\in[1,\oo]
\end{equation}
which implies in the coordinates \eqref{coordinate}, the function $\Omega$ defined in \eqref{def:Om} is bounded in $L^1\cap L^\oo$.
However, such estimates for large data seem to be too difficult.
}\end{rmk}

\subsection{A priori estimates and global existence}\label{subsection 2.2}
To extend the local well-posedness result obtained in the previous subsection to the global one, we only need to prove
 that the $L^1$ norm of the solution is non-increasing with respect to time $t.$

\begin{lem}\label{S2lem2}
{\sl Any smooth solution of \eqref{eqs:w} on $[0,T]$ satisfies $\|\omega(t)\|_{L^1}\leq \|\omega_0\|_{L^1}$ for all $t\in[0,T]$. Moreover, if $\omega_0$ changes signs, then the map $t\rightarrow \|\omega(t)\|_{L^1}$ is strictly decreasing.
}\end{lem}
\begin{proof}
We first consider the case when $\omega_0\geq 0$ and $\omega_0\neq 0$. Then, we get, by using the strong maximum principle for
 the equation \eqref{eqs:w}, that the solution $\omega(t)$ is strictly positive for $t\in (0,T]$. Recalling the mass  conservation for the solution of \eqref{eqs:w}, we find
$$
\|\omega(t)\|_{L^1}=\int_{\R^2} \omega(t,x,y)=\int_{\R^2} \omega_0(x,y)=\|\omega_0\|_{L^1}.
$$
We thus proved the lemma for non-negative initial data, and the negative case follows along the same line.

In  general, we decompose the initial data as $\omega_0=\omega_0^+-\omega_0^-$ so that $$
\|\omega_0\|_{L^1}=\|\omega_0^+\|_{L^1}+\|\omega_0^-\|_{L^1}=\int_{\R^2} \bigl(\omega^+_0+\omega^-_0\bigr).$$
We define $\omega^\pm$ through the following transport-diffusion equation:
\begin{equation}\notag
\quad \left\{\begin{array}{l}
\displaystyle \pa_t \om^\pm-\nu\Delta \omega^\pm +y\p_x \omega^\pm +u\cdot\nabla \omega^\pm=0, \qquad (t,x,y)\in\R^+\times\R^2, \\
\displaystyle  \omega^\pm|_{t=0}=\omega_{0}^\pm(x,y).
\end{array}\right.
\end{equation}
These two equations are locally well-posed, and both solutions will be non-negative due to the maximum principle.
Then we get,  by using mass conservation of $\omega^\pm$ that,
$$
\|\omega(t)\|_{L^1}=\int_{\R^2} |\omega^+-\omega^-|\leq \int_{\R^2}\left(\omega^++\omega^-\right)=\int_{\R^2}\left(\omega^+_0+\omega^-_0\right)=\|\omega_0\|_{L^1}.
$$
The above inequality can only become an equality if at least one of $\omega_0^\pm$ is trivial.
This completes the proof of Lemma \ref{S2lem2}.
\end{proof}


\begin{lem}\label{S2lem3}
{\sl For any $\omega_0\in L^1$, the solution of \eqref{eqs:w} satisfies that for $1\leq p\leq \oo$,
\begin{equation}\label{eq2.8}
\|\omega(t)\|_{L^p}\lesssim (\nu t)^{\f1p-1}\|\omega_0\|_{L^1}.
\end{equation}
}\end{lem}
\begin{proof} We first consider the case when $p\in [2,\infty).$ We get,
by multiplying \eqref{eqs:w} by $|\omega|^{p-2}\omega$   and integrating the resulting equation over $\R^2$, and then
using integration by parts, that
\beq\label{S2eq5}
\begin{split}
\f1p \f{d}{dt}\|\omega(t)\|_{L^p}^p=&\nu\int \Delta\omega |\omega|^{p-2}\omega\\
=&-\nu(p-1)\int |\nabla \omega|^2 |\omega|^{p-2}=-4\nu\f{p-1}{p^2}\|\nabla|\omega|^{\f{p}2}\|_{L^2}^2.
\end{split}
\eeq
It follows  from $\|f\|_{L^2}\lesssim \|f\|_{L^1}^\f12\|\nabla f\|_{L^2}^\f12$ that
\beq\label{S2eq6}
\|\omega\|_{L^p}^p=\||\omega|^\f{p}2\|_{L^2}^2\lesssim \||\omega|^\f{p}2\|_{L^1}\|\nabla |\omega|^\f{p}2\|_{L^2}\lesssim \|\omega\|_{L^{\f{p}2}}^{\f{p}2}\|\nabla |\omega|^\f{p}2\|_{L^2},
\eeq
from which and $\|\om\|_{L^{\f{p}2}}\leq \|\om\|_{L^1}^{\f1{p-1}}\|\om\|_{L^p}^{\f{p-2}{p-1}},$ we infer
\beno
\|w\|_{L^p}^{\f{p^2}{2(p-1)}}\leq \|w\|_{L^1}^{\f{p}{2(p-1)}}\|\nabla|\omega|^{\f{p}2}\|_{L^2}.
\eeno
By inserting the above inequality into \eqref{S2eq5}, we obtain
$$
\f{d}{dt}\|\omega(t)\|_{L^p}^p \leq -C_0\nu\f{p-1}p \|\omega_0\|_{L^1}^{-\f{p}{p-1}} \|\omega\|_{L^p}^{\f{p^2}{p-1}},
$$
from which we deduce that
\beq \label{S2eq7}
\|\omega(t)\|_{L^p}^{-\f{p}{p-1}}\geq \f{C_0\nu t}{p}\|\omega_0\|_{L^1}^{-\f{p}{p-1}}\andf \|\omega(t)\|_{L^p}\leq (C_0\nu t)^{\f1p-1}p^{\f{p-1}p}\|\omega_0\|_{L^1}.
\eeq
In particular, one has
$$
\|\omega(t)\|_{L^\f{p}2}\leq C_\f{p}2 (\nu t)^{\f2p-1}\|\omega_0\|_{L^1}.
$$
Then we get, by inserting \eqref{S2eq6} into \eqref{S2eq5}, that
\begin{align*}
\f{d}{dt}\|\omega(t)\|_{L^p}^p \leq &-C_0\nu\f{p-1}p \|\omega\|_{L^\f{p}2}^{-p} \|\omega\|_{L^p}^{2p}\\
\leq & -C_0\nu\f{p-1}p C_{\f{p}2}^{-p} (\nu t)^{p-2} \|\omega_0\|_{L^1}^{-p} \|\omega\|_{L^p}^{2p}.
\end{align*}
Solving the inequality leads to
$$
\|\omega(t)\|_{L^p}^p\leq \f1{\f1{\|\omega_0\|_{L^p}^p}+C_0\f{1}{p}(\nu t)^{p-1} C_{\f{p}2}^{-p} \|\omega_0\|_{L^1}^{-p}} \leq \f{p}{C_0}C_{\f{p}2}^p (\nu t)^{1-p}\|\omega_0\|_{L^1}^p,
$$
which implies
\beq \label{S2eq8}
\|\omega(t)\|_{L^p} \leq C_p (\nu t)^{\f1p-1}\|\omega_0\|_{L^1}, \with C_p=(\f{p}{C_0})^{\f1p} C_{\f{p}2}.
\eeq
Observing from \eqref{S2eq7} that  $C_2=\sqrt{2/C_0},$  we deduce that
$$
C_{2^k}=\prod_{j=1}^k (\f{2^j}{C_0})^{2^{-j}}\leq 2^{\sum_{j\geq1} j2^{-j}} C_0^{-\sum_{j\geq1}2^{-j}}=C_\oo,
$$
which is bounded by that quantity $C_\oo$. By taking $p=2^k$  in \eqref{S2eq8} with $k$ going to infinity, we  proved \eqref{eq2.8} with $p=\oo$. The other cases with $1<p<\oo$ follow from the interpolation between $L^1$ bounds and $L^\oo$ bounds. We thus finish the proof of Lemma \ref{S2lem3}.
\end{proof}

\begin{rmk}
{\sl By comparing \eqref{eq2.8} with \eqref{rmk2.1eq}, we find that the estimate \eqref{eq2.8}  does not show the influence of Couette flow. In coordinates \eqref{coordinate}, this estimate becomes
$$
\|\Omega(t)\|_{L^p}\lesssim \tt^{1-\f1p}\|\omega_0\|_{L^1},
$$
which means the $L^p$ norm of $\Omega$ may grow in time.
}\end{rmk}

\section{Estimates of $\Omega$ in weighted spaces}\label{section 3}

In this section, we present the proof of Theorem \ref{Thm2}.
 For simplicity, we just present the  {\it a priori} estimates in $L^2(m)$ for smooth enough solutions of \eqref{eq:Om}.
In subsection \ref{subsection 3.1}, we shall derive the estimate \eqref{eq1.12} by using the energy estimate. Due to the difficulties of large data, we shall
appeal to the $L^p$ norms we obtained in \eqref{eq2.8} to deal with part of the weighted estimates of nonlinear terms.

In subsection \ref{subsection 3.2}, we prove the bound \eqref{eq1.14} for higher regularities. When time $t$ is large enough, some of the coefficients in $\cL_t$ vanish. To gain the horizontal regularities, we have to use the hypo-ellipticity property of the operator. The smallness in the nonlinear problem comes from $\ln {t}/{t_0}$ by taking $t$ close enough to the starting point $t_0$ of the estimates.

\subsection{$L^2(m)$ estimates of $\Omega$} \label{subsection 3.1}

We first prove the following energy estimates in $L^2(m)$.

\begin{prop}\label{prop3.1}
{\sl Let $m>1$ and $\delta>0,$ let $\Omega(t)$  be a smooth solution \eqref{eq:Om} with $\Omega(1)\in L^2(m)$. Then, for any $t\geq 1$, there holds
\begin{equation}\label{eq:prop3.1}
\| \Omega(t)\|_{L^2(m)} \leq C_{m,\delta}\tt^{1+\delta}  \bigl( 1+{\nu}^{-1}{\|\Omega(1)\|_{L^2(m)}}\bigr)^{m}\|\Omega(1)\|_{L^2(m)}.
\end{equation}
}\end{prop}

\begin{proof} We first remark that due to
 $m>1$, $L^2(m)$ can be embedded  into $L^1(\R^2)$. In view of \eqref{coordinate} and \eqref{def:Om}, we find
 \beno
 \|\Omega(t)\|_{L^p}=\Bigl(\nu t\sqrt{1+\f{t^2}{12}}\Bigr)^{1-\f1p}\|\omega(t)\|_{L^p}\approx \left(\nu t\tt\right)^{1-\f1p}\|\omega(t)\|_{L^p},
 \eeno
 from which and \eqref{eq2.8}, we deduce that  for all $1\leq p\leq\oo$,
\begin{equation}\label{eq3.1}
\begin{aligned}
\|\Omega(t)\|_{L^p}\leq& C \left(\nu t\tt\right)^{1-\f1p}\bigl(\nu (t-1)\bigr)^{\f1p-1}\|\omega(1)\|_{L^1}\\
\leq &C\tt^{1-\f1p} \Bigl(\f{t}{t-1}\Bigr)^{1-\f1p}\|\Omega(1)\|_{L^1}
\leq C_m \tt^{1-\f1p} \Bigl(\f{t}{t-1}\Bigr)^{1-\f1p}\|\Omega(1)\|_{L^2(m)} .
\end{aligned}\end{equation}
On the other hand, noticing that the $L^p$ norm of $\omega(t)$ is non-increasing with respect to time $t,$ we infer that for all $1\leq p\leq 2$,
\begin{align*}
\|\Omega(t)\|_{L^p}=&\left(\nu t\tt\right)^{1-\f1p}\|\omega(t)\|_{L^p}\leq \left(\nu t\tt\right)^{1-\f1p}\|\omega(1)\|_{L^p}\\
 =&\tt^{1-\f1p} t^{1-\f1p} \|\Omega(1)\|_{L^p}
\leq  C_m \tt^{1-\f1p} t^{1-\f1p} \|\Omega(1)\|_{L^2(m)}.
\end{align*}
Then, by discussing separately whether $t-1$ is small or not, we can improve \eqref{eq3.1} for $1\leq p\leq 2$ to
\begin{equation}\label{eq3.1a}
\|\Omega(t)\|_{L^p}\leq C_m \tt^{1-\f1p}\|\Omega(1)\|_{L^2(m)},
\end{equation}
Similarly, when $2\leq p\leq +\oo$, we get, by using a similar version of \eqref{eq2.8} with initial data in $L^2(\R^2),$ that
\begin{align*}
\|\Omega(t)\|_{L^p}=&\left(\nu t\tt\right)^{1-\f1p}\|\omega(t)\|_{L^p} \lesssim\left(\nu t\tt\right)^{1-\f1p}\bigl(\nu (t-1)\bigr)^{\f1p-\f12}\|\omega(1)\|_{L^2}\\
\lesssim & \left(\nu t\tt\right)^{1-\f1p}\bigl(\nu (t-1)\bigr)^{\f1p-\f12}\nu^{-\f12}\|\Omega(1)\|_{L^2} \lesssim \tt^{1-\f1p} \Bigl(\f{t}{t-1}\Bigr)^{\f12-\f1p} t^\f12\|\Omega(1)\|_{L^2}
\end{align*}
which in particular improves \eqref{eq3.1} to
\begin{equation}\label{eq3.1b}
\|\Omega(t)\|_{L^p}\leq \tt^{1-\f1p} \Bigl(\f{t}{t-1}\Bigr)^{\f12-\f1p} \|\Omega(1)\|_{L^2(m)}\quad\mbox{for}\ p\in [2,\infty].
\end{equation}

Let us now consider the estimates with homogeneous weight $a^m(X,Y)=(X^2+Y^2)^{\f{m}2}$.  Indeed, by using integration by parts, one has
\begin{align*}
\int \Delta_t\Omega \, \Omega \, a^{2m}=&-\bigl(1+\f{t^2}3\bigr)^{-1}\bigl\|a^m\Bigl(\p_X-\f{t}2(1+\f{t^2}{12})^{-\f12}\p_Y\Bigr)\Omega\bigr\|_{L^2}^2-\f{1+\f{t^2}3}{1+\f{t^2}{12}}\|a^m\p_Y\Omega\|_{L^2}^2\\
&-\f{2m}{1+\f{t^2}3}\int a^{2m-2}\Bigl(X-\f{t}2\bigl(1+\f{t^2}{12}\bigr)^{-\f12}Y\Bigr)\Bigl(\p_X-\f{t}2\bigl(1+\f{t^2}{12}\bigr)^{-\f12}\p_Y\Bigr)\Omega \,\Omega\\
&-2m\f{1+\f{t^2}3}{1+\f{t^2}{12}}\int a^{2m-2} Y\p_Y \Omega\,\Omega\\
=&-\bigl(1+\f{t^2}3\bigr)^{-1}\bigl\|\Bigl(\p_X-\f{t}2\bigl(1+\f{t^2}{12}\bigr)^{-\f12}\p_Y\Bigr)(a^m\Omega)\bigr\|_{L^2}^2-\f{1+\f{t^2}3}{1+\f{t^2}{12}}\|\p_Y(a^m\Omega)\|_{L^2}^2\\
&+\f{m^2}{1+\f{t^2}3}\int a^{2m-4} \Bigl(X-\f{t}2\bigl(1+\f{t^2}{12}\bigr)^{-\f12}Y\Bigr)^2 \Omega^2+\f{m^2(1+\f{t^2}3)}{1+\f{t^2}{12}}\int a^{2m-4} Y^2 \Omega^2.
\end{align*}
Then in view of \eqref{eq:Oma}, we find
\begin{align*}
\int \cL_t\Omega \, \Omega \, a^{2m}
=&-\bigl(1+\f{t^2}3\bigr)^{-1}\bigl\|\Bigl(\p_X-\f{t}2\bigl(1+\f{t^2}{12}\bigr)^{-\f12}\p_Y\Bigr)(a^m\Omega)\bigr\|_{L^2}^2-\f{1+\f{t^2}3}{1+\f{t^2}{12}}\|\p_Y(a^m\Omega)\|_{L^2}^2\\
&+{m^2}\bigl(1+\f{t^2}3\bigr)^{-1}\int a^{2m-4} \Bigl(X-\f{t}2\bigl(1+\f{t^2}{12}\bigr)^{-\f12}Y\Bigr)^2 \Omega^2+\f{m^2(1+\f{t^2}3)}{1+\f{t^2}{12}}\int a^{2m-4} Y^2 \Omega^2\\
&+\f{6+t^2}{12+t^2}\|a^m \Omega\|_{L^2}^2-\f{m}2{\bigl(1+\f{t^2}3\bigr)^{-1}}\int a^{2m-2}\Bigl(X-\f{t}2\bigl(1+\f{t^2}{12}\bigr)^{-\f12}Y\Bigr)^2\Omega^2\\
&-\f{m(1+\f{t^2}3)}{2(1+\f{t^2}{12})}\int a^{2m-2}Y^2\Omega^2 .
\end{align*}

 On the other hand, by separating the cases whether  $a(X,Y)$ is larger than $\epsilon^{-\f12}$ or not, we observe that for any $\epsilon>0$,
$$
 a^{2m-4} \Bigl(X-\f{t}2\bigl(1+\f{t^2}{12}\bigr)^{-\f12}Y\Bigr)^2 \leq \epsilon a^{2m-2}\Bigl(X-\f{t}2\bigl(1+\f{t^2}{12}\bigr)^{-\f12}Y\Bigr)^2 +C \epsilon^{1-m},
$$
and
$$
a^{2m-4} Y^2 \leq \epsilon a^{2m-2}Y^2 +C \epsilon^{1-m}.
$$
Then we get, by taking $\epsilon=\f{1}{4m}$ in the above inequalities, that
\begin{equation}\label{estimates:L}
\begin{aligned}
&\int \cL_t\Omega \, \Omega \, a^{2m}
\leq \f{6+t^2}{12+t^2}\|a^m \Omega\|_{L^2}^2+C_m\|\Omega\|_{L^2}^2 .
\end{aligned}
\end{equation}

Next, we consider the weighted estimates on the nonlinear terms. Indeed, by using integration by parts, we write
\begin{align*}
\int\cN_t \Omega\, \Omega a^{2m}= \f{m}\nu \Bigl(1+\f{t^2}{12}\Bigr)^{-1}\Bigl(\int a^{2m-2} Y \Omega^2 \p_X\Delta_t^{-1}\Omega
-\int a^{2m-2}X\Omega^2 \p_Y \Delta_t^{-1}\Omega \Bigr).
\end{align*}
For the terms with Biot-Savart law, we use the following lemma, whose proof will be postponed until the end of this subsection:
\begin{lem}\label{lem:BS law}
{\sl There holds
\begin{equation}\label{eq:BS law}
\|\p_X\Delta_t^{-1}\Omega\|_{L^\oo}+\tt\|\p_Y\Delta_t^{-1}\Omega\|_{L^\oo}
\leq C\tt^\f32\|\Omega\|_{L^\f{2m}{m-1}}^\f12\|\Omega\|_{L^{\f{2m}{m+1}}}^\f12.
\end{equation}
}\end{lem}
\noindent We use this inequality and observe from  \eqref{eq3.1a} and \eqref{eq3.1b} that
\begin{equation}\label{eq3.4}
\begin{aligned}
\|\p_Y\Delta_t^{-1}\Omega\|_{L^\oo}
\leq &C\tt^\f12\|\Omega\|_{L^\f{2m}{m-1}}^\f12\|\Omega\|_{L^{\f{2m}{m+1}}}^\f12
\leq C_m\tt \Bigl(\f{t}{t-1}\Bigr)^\f1{4m} \|\Omega(1)\|_{L^2(m)},
\end{aligned}
\end{equation}
and similarly,
\begin{equation}\label{eq3.5}
\|\p_X\Delta_t^{-1}\Omega\|_{L^\oo}
\leq C_m \tt^2 \Bigl(\f{t}{t-1}\Bigr)^\f1{4m}\|\Omega(1)\|_{L^2(m)},
\end{equation}
from which we infer
$$
\int\cN_t \Omega\, \Omega a^{2m}\leq\f{C_m}\nu \Bigl(\f{t}{t-1}\Bigr)^\f1{4m}\|\Omega(1)\|_{L^2(m)} \Bigl(\int a^{2m-2} |Y|\Omega^2+\tt^{-1}\int a^{2m-2}|X|\Omega^2  \Bigr).
$$
Observing that for any  $\delta>0$, one has
\begin{align*}
\f{C_m}\nu \Bigl(\f{t}{t-1}\Bigr)^\f1{4m}\|\Omega(1)\|_{L^2(m)} a^{2m-2}\bigl(|X| +|Y|\bigr) \leq &\f{C_m}\nu \Bigl(\f{t}{t-1}\Bigr)^\f1{4m}\|\Omega(1)\|_{L^2(m)}a^{2m-1} \\
\leq &\delta a^{2m}+C_{m,\delta}\Bigl(\f{t}{t-1}\Bigr)^\f12\bigl({\nu}^{-1}{\|\Omega(1)\|_{L^2(m)}} \bigr)^{2m}.
\end{align*}
As a consequence, we obtain
\begin{equation}\label{estimates:NL}
\begin{aligned}
&\int\cN_t \Omega\, \Omega a^{2m} \leq \delta \|a^{2m} \Omega\|_{L^2}^2+ C_{m,\delta} \Bigl(\f{t}{t-1}\Bigr)^\f12 \bigl({\nu}^{-1}{\|\Omega(1)\|_{L^2(m)}}\bigr)^{2m} \|\Omega\|_{L^2}^2.
\end{aligned}
\end{equation}

By virtue of \eqref{estimates:L} and \eqref{estimates:NL}, we get, by
 taking $L^2$ inner product of \eqref{eq:Om} with $a^m\Omega$, that
\begin{align*}
\f{t}2\f{d}{dt}\|a^m\Omega(t)\|_{L^2}^2 \leq &\Bigl(\f{6+t^2}{12+t^2}+\delta\Bigr)\|a^m \Omega(t)\|_{L^2}^2\\
&+C_{m,\delta}  \Bigl(1+\Bigl(\f{t}{t-1}\Bigr)^\f12\bigl({\nu}^{-1}{\|\Omega(1)\|_{L^2(m)}}\bigr)^{2m}\Bigr)\|\Omega\|_{L^2}^2.
\end{align*}
Noticing that $\f{6+t^2}{12+t^2}\leq 1$, we get, by using \eqref{eq3.1a} with $p=2$, that
\begin{align*}
t\f{d}{dt}\|a^m\Omega(t)\|_{L^2}^2\leq &2(1+\delta)\|a^m\Omega(t)\|_{L^2}^2 \\
&+C_{m,\delta} \bigl(1+{\nu}^{-1}{\|\Omega(1)\|_{L^2(m)}}\bigr)^{2m}  \|\Omega(1)\|_{L^2(m)}^2 \Bigl(\f{t}{t-1}\Bigr)^\f12\tt .
\end{align*}
By using Gronwall's inequality, we get that for any $t\geq 1$,
\begin{align*}
\|a^m \Omega(t)\|_{L^2}^2\leq &t^{2(1+\delta)}\|a^m\Omega(1)\|_{L^2}^2\\
&+C_{m,\delta}  \bigl(1+{\nu}^{-1}{\|\Omega(1)\|_{L^2(m)}}\bigr)^{2m} \|\Omega(1)\|_{L^2(m)}^2 t^{2(1+\delta)}\int_1^t s^{-1-2\delta}\Bigl(\f{s}{s-1}\Bigr)^\f12\,ds \\
\leq &C_{m,\delta} \bigl(1+\nu^{-1}{\|\Omega(1)\|_{L^2(m)}}\bigr)^{2m}\|\Omega(1)\|_{L^2(m)}^2 \tt^{2(1+\delta)},
\end{align*}
 By taking a square root, we arrive at
$$
\|a^m \Omega(t)\|_{L^2} \leq C_{m,\delta}\tt^{1+\delta}\bigl(1+{\nu}^{-1}{\|\Omega(1)\|_{L^2(m)}}\bigr)^{m}\|\Omega(1)\|_{L^2(m)},
$$
which together with \eqref{eq3.1a} ensures \eqref{eq:prop3.1}. This finishes the proof of Proposition \ref{prop3.1}.
\end{proof}

Noticing that the $L^2(m)$ estimate \eqref{eq:prop3.1} of $\Omega$ grows as  $\tt^{1+\delta}$ with respect to the time $t,$ which is faster than the $L^2$ norm  growth estimate $\tt^\f12$. Below, we are going to use an interpolation argument to improve the time growth:

\begin{col}\label{col3.1}
{\sl Let $m_0>1$ and $\delta>0$. Suppose \eqref{eq:Om} has a solution $\Omega(t)\in L^2(m_0)$. Then, for any $0\leq m\leq m_0$, there holds
\begin{equation}\label{eq:col3.1}
\|\Omega(t)\|_{L^2(m)}\leq C_{m,m_0,\delta} \bigl( 1 +{\nu}^{-1}{\|\Omega(1)\|_{L^2(m_0)}}\bigr)^{m} \tt^{\f12 +\f{m}{m_0}(\f12+\delta)}\|\Omega(1)\|_{L^2(m_0)}.
\end{equation}
}\end{col}
\begin{proof} We deduce from  \eqref{eq3.1a} with $p=2$ and \eqref{eq:prop3.1} with $m=m_0$ that for any $0<m<m_0$,
\begin{align*}
&\|\Omega(t)\|_{L^2(m)}\leq \|\Omega(t)\|_{L^2}^{1-\f{m}{m_0}}\|\Omega(t)\|_{L^2(m_0)}^{\f{m}{m_0}} \\
&\leq C_{m,m_0,\delta} \tt^{\f12-\f{m}{2m_0}} \|\Omega(1)\|_{L^2(m_0)}^{1-\f{m}{m_0}}
\Bigl( 1 +\bigl({\nu}^{-1}{\|\Omega(1)\|_{L^2(m_0)}}\bigr)^{m_0}\Bigr)^{\f{m}{m_0}} \tt^{\f{m}{m_0}(1+\delta)} \|\Omega(1)\|_{L^2(m_0)}^{\f{m}{m_0}}\\
&\leq C_{m,m_0,\delta} \Bigl( 1 +\bigl(\nu^{-1}{\|\Omega(1)\|_{L^2(m_0)}}\bigr)^{m}\Bigr) \tt^{\f12 +\f{m}{m_0}(\f12+\delta)}\|\Omega(1)\|_{L^2(m_0)},
\end{align*}
which finishes the proof of \eqref{eq:col3.1}.
\end{proof}

Now, we present the proof of Lemma \ref{lem:BS law}:
\begin{proof}[Proof of Lemma \ref{lem:BS law}]
The basic idea is to change the variables from $(X,Y)$ back to $(x,y)$. By using \eqref{coordinate} and \eqref{def:Om}, one can easily compute
$$
\p_y\Delta^{-1}\omega(t,x,y)=\f{\sqrt{1+\f{t^2}3}}{\sqrt{\nu t}(1+\f{t^2}{12})} \p_Y\Delta_t^{-1} \Omega(t,X,Y)
$$
and
$$
\p_x\Delta^{-1}\omega(t,x,y)=\f1{\sqrt{\nu t(1+\f{t^2}3)(1+\f{t^2}{12})}}(\p_X-\f{t}{2\sqrt{1+\f{t^2}{12}}}\p_Y)\Delta_t^{-1}\Omega(t,X,Y).
$$
Therefore, we apply the classical Biot-Savart law $\|\nabla\Delta^{-1}\omega\|_{L^\oo}\lesssim \|\omega\|_{L^{\f{2m}{m-1}}}^\f12\|\omega\|_{L^{\f{2m}{m+1}}}^\f12$ and change variables with \eqref{coordinate} to conclude
\begin{align*}
&\tt\|\p_Y\Delta_t^{-1}\Omega\|_{L^\oo}+\|\p_X\Delta_t^{-1}\Omega\|_{L^\oo}
\lesssim (\nu t)^\f12 \tt^2\|\nabla\Delta^{-1}\omega\|_{L^\oo}\\
\lesssim &(\nu t)^\f12 \tt^2 \bigl( (\nu t\tt)^{-\f{m+1}{2m}} \|\Omega\|_{L^{\f{2m}{m+1}}} \bigr)^\f12 \bigl( (\nu t\tt)^{-\f{m-1}{2m}} \|\Omega\|_{L^{\f{2m}{m-1}}} \bigr)^\f12 \\
\lesssim &\tt^\f32 \|\Omega\|_{L^{\f{2m}{m+1}}}^\f12 \|\Omega\|_{L^{\f{2m}{m-1}}}^\f12,
\end{align*}
which finishes the proof of \eqref{eq:BS law}.
\end{proof}

\subsection{Higher regularity estimates}\label{subsection 3.2}
 The goal of this subsection is to handle the derivative estimate of $\Omega$ in $L^2(m)$. When $t$ is not large, the linear operator $\Delta_t$ can gain the regularity in both $X$ and $Y$ variables. However, as $t\rightarrow \oo$, the formal limit of $\Delta_t$ is  $\p_Y^2$, which can not gain any horizontal regularity. For such reason, we shall make use of the hypo-ellipticity of the Fokker-Planck operator, see for instance \cite{hormander1967hypoelliptic,villani2009hypocoercivity}.

We first introduce the following energy functional:
\begin{equation}\label{def:E}
\begin{aligned}
E(t)\eqdefa &\|\Omega(t)\|_{L^2(m)}^2+c_1\bigl(\ln{t}/{t_0}\bigr)\|\p_Y \Omega(t)\|_{L^2(m)}^2+c_2\bigl(\ln{t}/{t_0}\bigr)^2 \bigl(\p_X \Omega (t)\big|\p_Y \Omega(t) \bigr)_{L^2(m)} \\
&+c_3\bigl(\ln{t}/{t_0}\bigr)^3\|\p_X \Omega(t)\|_{L^2(m)}^2+c_4\bigl(\ln{t}/{t_0}\bigr)^2\|\p_Y^2\Omega(t)\|_{L^2(m)}^2\\
&+c_5\bigl(\ln{t}/{t_0}\bigr)^4\|\p_X\p_Y\Omega(t)\|_{L^2(m)}^2+c_6\bigl(\ln{t}/{t_0}\bigr)^5 \bigl(\p_X^2 \Omega(t) \big|\p_X\p_Y \Omega(t) \bigr)_{L^2(m)}\\
& +c_7\bigl(\ln{t}/{t_0}\bigr)^6\|\p_X^2\Omega(t)\|_{L^2(m)}^2 ,
\end{aligned}
\end{equation}
and the associated dissipation energy functional:
\begin{equation}\label{def:D}
\begin{aligned}
D(t)\eqdefa &\tt^{-2}\|\p_X\Omega(t)\|_{L^2(m)}^2+\|\p_Y\Omega(t)\|_{L^2(m)}^2+c_1\bigl(\ln{t}/{t_0}\bigr)\Bigl(\tt^{-2}\|\p_X\p_Y\Omega(t)\|_{L^2(m)}^2\\
&+\|\p_Y^2\Omega(t)\|_{L^2(m)}^2\Bigr)
+c_2\bigl(\ln{t}/{t_0}\bigr)^2\|\p_X\Omega(t)\|_{L^2(m)}^2\\
&+c_3\bigl(\ln{t}/{t_0}\bigr)^3\Bigl(\tt^{-2}\|\p_X^2\Omega(t)\|_{L^2(m)}^2+\|\p_X\p_Y\Omega(t)\|_{L^2(m)}^2\Bigr)\\
&+c_4\bigl(\ln{t}/{t_0}\bigr)^2\Bigl(\tt^{-2}\|\p_X\p_Y^2\Omega(t)\|_{L^2(m)}^2+\|\p_Y^3\Omega(t)\|_{L^2(m)}^2\Bigr) \\
&+c_5\bigl(\ln{t}/{t_0}\bigr)^4\Bigl(\tt^{-2}\|\p_X^2\p_Y\Omega(t)\|_{L^2(m)}^2+\|\p_X\p_Y^2\Omega(t)\|_{L^2(m)}^2\Bigr)\\
&
+c_6\bigl(\ln{t}/{t_0}\bigr)^5\|\p_X^2\Omega(t)\|_{L^2(m)}^2
+c_7\bigl(\ln{t}/{t_0}\bigr)^6\Bigl(\tt^{-2}\|\p_X^3\Omega(t)\|_{L^2(m)}^2\\
&+\|\p_X^2\p_Y\Omega(t)\|_{L^2(m)}^2\Bigr).
\end{aligned}
\end{equation}

In the above energy functionals,  $t_0$ will be taken to be an arbitrary large time and $t$ to belong to $[t_0,2t_0]$ so that $\ln {t}/{t_0}$ is small. All the estimates below hold for $t_0\geq T_0$ for some universal $T_0$. The small constants from $c_1$ to $c_7$ are chosen to satisfy
\begin{equation}\label{assumptions on c1-7}
\quad \left\{\begin{array}{l}
\displaystyle 1\gg c_1\gg c_2=c_4 \gg c_3\gg c_5\gg c_6\gg c_7, \\
\displaystyle  c_1^2\ll c_2, \qquad c_2^2\ll_m c_1c_3, \qquad c_5^2\ll c_3c_6,\qquad c_6^2\ll_m c_5c_7,
\end{array}\right.
\end{equation}
where  $f\ll_m g$ means that there is a large constant $C_m$ depending on $m$ so that $f\leq {C_m} g$.
Below we present one example of $c_1$ to $c_7$  which can satisfy \eqref{assumptions on c1-7}. We first take a large enough constant $A_m$ and then take
$$
(c_1,c_2,c_3,c_4,c_5,c_6,c_7)=(A_m^{-3},A_m^{-5},A_m^{-6},A_m^{-5} ,A_m^{-9},A_m^{-11},A_m^{-12}).
$$

Before proceeding, we first introduce some useful anisotropic inequalities.

\begin{lem}\label{S3lem1}
{\sl Let $m>1$ and $t>2$, one has
\begin{equation}\label{eq:lem3.1a}
\|\p_X\Delta_t^{-1}\Omega(t)\|_{L^\oo}+\tt\|\p_Y\Delta_t^{-1}\Omega(t)\|_{L^\oo}\lesssim \tt^2 \|\Omega(1)\|_{L^2(m)},
\end{equation}
and
\begin{equation}\label{eq:lem3.1b}
\|\p_X\p_Y\Delta_t^{-1}\Omega(t)\|_{L^2}+\tt\|\p_Y^2\Delta_t^{-1}\Omega(t)\|_{L^2}\lesssim \tt^{\f32} \|\Omega(1)\|_{L^2(m)},
\end{equation}
}\end{lem}
\begin{proof}
 \eqref{eq:lem3.1a} follows from \eqref{eq3.4} and \eqref{eq3.5}. While in view of \eqref{eq:Oma}, we have
 \begin{align*}
 \|\p_X\p_Y\Delta_t^{-1}\Omega(t)\|_{L^2}+\tt\|\p_Y^2\Delta_t^{-1}\Omega(t)\|_{L^2}\lesssim \tt\|\Omega(t)\|_{L^2},
 \end{align*}
 which together with \eqref{eq3.1a} for $p=2$ ensures
 \eqref{eq:lem3.1b}.
\end{proof}

\begin{lem}\label{S3lem2}
{\sl For any $0<\sigma<\f12$, one has
\begin{equation}\label{eq3.6}
\|\p_X\Delta_t^{-1} f\|_{L^\oo}+\tt\|\p_Y\Delta_t^{-1}f\|_{L^\oo}
\leq C_\sigma \tt^{1+\sigma} \|f\|_{L^2(1)}^{\f12+\sigma}\|\p_X f\|_{L^2(1)}^{\f12-\sigma}.
\end{equation}
}\end{lem}
\begin{proof}
Let us denote $\hat{f}(\xi,\eta)$ as the Fourier transform of $f(X,Y).$  It is easy to observe that
\begin{align*}
\|\p_X\Delta_t^{-1} f\|_{L^\oo}+\tt\|\p_Y\Delta_t^{-1} f\|_{L^\oo}
\leq & \bigl\|{\xi}\Bigl({\f{1}{1+\f{t^2}3}\bigl(\xi-\f{t}2
(1+\f{t^2}{12})^{-\f12}\eta\bigr)^2+\f{1+\f{t^2}3}{1+\f{t^2}{12}}\eta^2}\Bigr)^{-1}\hat{f}\bigr\|_{L^1}\\
&+\tt\bigl\|{\eta}\Bigl({\f{1}{1+\f{t^2}3}\bigl(\xi-\f{t}2(1+\f{t^2}{12})^{-\f12}\eta\bigr)^2+\f{1+\f{t^2}3}{1+\f{t^2}{12}}\eta^2}\Bigr)^{-1}
\hat{f}\bigr\|_{L^1} \\
&\lesssim \tt^{1+\sigma}\bigl\|{\hat{f}}{|\xi|^{-\sigma}|\eta|^{\sigma-1}}\bigr\|_{L^1}.
\end{align*}

Notice that due to $\s\in (0,\f12),$ for any $R>0,$ we have
\begin{align*}
\||\xi|^{-\sigma}g\|_{L^1_{\xi}}\leq &\Bigl(\int_{|\xi|\leq R}|\xi|^{-2\s}\Bigr)^{\f12}\|g\|_{L^2_\xi}+
\Bigl(\int_{|\xi|\geq R}|\xi|^{-2(1+\s)}\Bigr)^{\f12}\|\xi g\|_{L^2_\xi}
\\
\lesssim &R^{\f12-\s}\|g\|_{L^2_\xi}+R^{-\f12-\s}\|\xi g\|_{L^2_\xi}.
\end{align*}
Taking $R=\f{\|\xi g\|_{L^2_\xi}}{\|g\|_{L^2_\xi}}$ in the above inequality leads to
\beq \label{S3eq1} \||\xi|^{-\sigma}g\|_{L^1_{\xi}}\lesssim \|g\|_{L^2_\xi}^{\f12+\sigma}\|\xi g\|_{L^2_\xi}^{\f12-\sigma},.
\eeq
Similarly, one can show that
\beq \label{S3eq2}
\||\eta|^{\sigma-1}g\|_{L^1_{\eta}} \lesssim \|g\|_{L^\oo_{\eta}}^{1-2\sigma}\|g\|_{L^2_\eta}^{2\sigma}.
\eeq

Thanks to \eqref{S3eq1} and \eqref{S3eq2}, we get, by using
 the fact that $L^2(1)$ can be continuous embedded into $L^1_Y(L^2_X),$ that
\begin{align*}
\bigl\|{\hat{f}}{|\xi|^{-\sigma}|\eta|^{\sigma-1}}\bigr\|_{L^1}
\lesssim &\bigl\| \|\hat{f}\|_{L^\oo_{\eta}}^{1-2\sigma}\|\hat{f}\|_{L^2_\eta}^{2\sigma}\bigr\|_{L^2_\xi}^{\f12+\sigma}
\bigl\| \xi \|\hat{f}\|_{L^\oo_{\eta}}^{1-2\sigma}\|\xi\hat{f}\|_{L^2_\eta}^{2\sigma}\bigr\|_{L^2_\xi}^{\f12-\sigma} \\
\lesssim &\|\hat{f}\|_{L^2_\xi(L^\oo_{\eta})}^{(1-2\sigma)(\f12+\sigma)}\|\hat{f}\|_{L^2}^{2\sigma(\f12+\sigma)}
\|\xi\hat{f}\|_{L^2_\xi(L^\oo_{\eta})}^{(1-2\sigma)(\f12-\sigma)}\|\xi\hat{f}\|_{L^2}^{2\sigma(\f12-\sigma)}\\
\lesssim &\|f\|_{L^1_Y(L^2_{X})}^{(1-2\sigma)(\f12+\sigma)}\|f\|_{L^2}^{2\sigma(\f12+\sigma)}
\|\p_X f\|_{L^1_Y(L^2_{X})}^{(1-2\sigma)(\f12-\sigma)}\|\p_X f\|_{L^2}^{2\sigma(\f12-\sigma)}\\
\lesssim &\|f\|_{L^2(1)}^{\f12+\sigma}\|\p_X f\|_{L^2(1)}^{\f12-\sigma}.
\end{align*}
This finishes the proof of \eqref{eq3.6}.
\end{proof}

Let us turn to the estimates of the terms appearing in \eqref{def:E}.

\begin{lem}\label{S3lem3}
{\sl Let $m>1$ and $t_0$ be a large enough positive constant. Then for $t\geq t_0$, one has
\begin{equation}\label{eq:lem1}
t\f{d}{dt}\|\Omega(t) \|_{L^2(m)}^2 +2\tt^{-2}\|\p_X\Omega\|_{L^2(m)}^2+4\|\p_Y\Omega\|_{L^2(m)}^2
\leq C_{m} \bigl(1+{\nu}^{-1}{\|\Omega(1)\|_{L^2(m)}}\bigr) E(t).
\end{equation}
}\end{lem}
\begin{proof}
The proof of this lemma will follow along the same line as that of   Proposition \ref{prop3.1}. For simplicity, we shall denote $\langle X,Y\rangle$ by $b(X,Y).$ We first get, by a similar derivation of \eqref{estimates:L}, that
\begin{align*}
\int \Delta_t\Omega \, \Omega \, b^{2m}=&-\bigl(1+\f{t^2}3\bigr)^{-1}\bigl\|\Bigl(\p_X-\f{t}2\bigl(1+\f{t^2}{12}\bigr)^{-\f12}\p_Y\Bigr)\Omega\bigr\|_{L^2(m)}^2-\f{1+\f{t^2}3}{1+\f{t^2}{12}}\|\p_Y\Omega\|_{L^2(m)}^2\\
&-\f{2m}{1+\f{t^2}3}\int b^{2m-2}\Bigl(X-\f{t}2\bigl(1+\f{t^2}{12}\bigr)^{-\f12}Y\Bigr)\Bigl(\p_X-\f{t}2\bigl(1+\f{t^2}{12}\bigr)^{-\f12}\p_Y\Bigr)\Omega \,\Omega\\
&-2m\f{1+\f{t^2}3}{1+\f{t^2}{12}}\int b^{2m-2} Y\p_Y \Omega\,\Omega\\
\leq&-\tt^{-2}\|\p_X\Omega\|_{L^2(m)}^2-2\|\p_Y\Omega\|_{L^2(m)}^2
+C_m\|\Omega\|_{L^2(m-1)}^2,
\end{align*}
where we used that for $t\geq t_0$ with $t_0$ being large enough,
\begin{align*}
&-\bigl(1+\f{t^2}3\bigr)^{-1}\bigl\|\Bigl(\p_X-\f{t}2\bigl(1+\f{t^2}{12}\bigr)^{-\f12}\p_Y\Bigr)\Omega\bigr\|_{L^2(m)}^2-\f{1+\f{t^2}3}{1+\f{t^2}{12}}\|\p_Y\Omega\|_{L^2(m)}^2 \\
\leq &-2\tt^{-2}\bigl( \|\p_X \Omega\|_{L^2(m)}^2 -4\|\p_X\Omega\|_{L^2(m)}\|\p_Y\Omega\|_{L^2(m)} \bigr) -3\|\p_Y \Omega\|_{L^2(m)}^2 \\
\leq &-\tt^{-2}\|\p_X\Omega\|_{L^2(m)}^2-2\|\p_Y\Omega\|_{L^2(m)}^2.
\end{align*}
Then in view of \eqref{eq:Oma}, we find
\begin{equation}\label{eq3.8}
\int \cL_t\Omega \, \Omega \, b^{2m}\leq-\tt^{-2}\|\p_X\Omega\|_{L^2(m)}^2-2\|\p_Y\Omega\|_{L^2(m)}^2+C_m\|\Omega\|_{L^2(m)}^2.
\end{equation}

While we get, by using  integration by parts, that
\begin{align*}
\int\cN_t \Omega\, \Omega b^{2m}= \f{m}\nu (1+\f{t^2}{12})^{-1}\Bigl(\int b^{2m-2} Y \Omega^2\p_X\Delta_t^{-1}\Omega -\int b^{2m-2}X\Omega^2 \p_Y \Delta_t^{-1}\Omega \Bigr)\\
\leq \f{C_m}\nu \tt^{-2}\|\Omega\|_{L^2(m)}^2 \bigl( \|\p_X\Delta_t^{-1}\Omega \|_{L^\oo}+\|\p_Y\Delta_t^{-1}\Omega \|_{L^\oo} \bigr),
\end{align*}
from which and \eqref{eq:lem3.1a}, we infer
\begin{equation}\label{eq3.9}
\int\cN_t \Omega\, \Omega b^{2m} \leq  \f{C_{m}}{\nu}{\|\Omega(1)\|_{L^2(m)}} \|\Omega\|_{L^2(m)}^2.
\end{equation}

Thanks to \eqref{eq3.8} and \eqref{eq3.9}, by taking $L^2(m)$ inner product of \eqref{eq:Om} with $\Omega,$
we conclude the proof of \eqref{eq:lem1}.
\end{proof}

\begin{lem}\label{S3lem4}
{\sl Under the assumptions of Lemma \ref{S3lem3}, for $t_0\leq t\leq 2t_0 $, we have
\begin{equation}\label{eq:lem2}
\begin{aligned}
&t \f{d}{dt}\Bigl(\bigl(\ln{t}/{t_0}\bigr)\|\p_Y\Omega\|_{L^2(m)}^2\Bigr) +\bigl(\ln{t}/{t_0}\bigr)\bigl(\tt^{-2}\|\p_X\p_Y\Omega\|_{L^2(m)}^2+\|\p_Y^2\Omega\|_{L^2(m)}^2\bigr)\\
&\leq \|\p_Y\Omega\|_{L^2(m)}^2+4\bigl(\ln{t}/{t_0}\bigr)\|\p_X\Omega\|_{L^2(m)}\|\p_Y\Omega\|_{L^2(m)}\\
&\quad+C_m \bigl(1+\nu^{-1}{\|\Omega(1)\|_{L^2(m)}}\bigr)\bigl(\ln{t}/{t_0}\bigr)^{\f12}  D(t).
\end{aligned}\end{equation}
}\end{lem}
\begin{proof}
By applying $\p_Y$ to \eqref{eq:Om} and then taking  $L^2(m)$ inner product of the resulting equation with $\ln(\f{t}{t_0})\p_Y\Omega$, we compute
\begin{equation}\label{eq3.15}
\f{t}2 \f{d}{dt}\Bigl(\bigl(\ln{t}/{t_0}\bigr)\|\p_Y\Omega(t)\|_{L^2(m)}^2\Bigr) =\f12 \|\p_Y\Omega\|_{L^2(m)}^2+\bigl(\ln{t}/{t_0}\bigr)\int \p_Y\bigl(\cL_t\Omega+\cN_t \Omega\bigr) \p_Y\Omega b^{2m}.
\end{equation}

It is easy to observe from \eqref{eq:Oma} that
\begin{equation}\label{[py;L]}
\p_Y\cL_t=\cL_t\p_Y -\f{t\sqrt{1+\f{t^2}{12}}}{{1+\f{t^2}3}}\p_X+\Bigl( \f{t^2}{8(1+\f{t^2}3)(1+\f{t^2}{12})}+\f{1+\f{t^2}3}{2(1+\f{t^2}{12})}\Bigr)\p_Y,
\end{equation}
from which and \eqref{eq3.8}, we deduce that for $t\geq t_0$ with $t_0$ being large enough,
\begin{equation}\label{eq3.17}
\begin{aligned}
\int \p_Y\cL_t\Omega \p_Y\Omega b^{2m}
&\leq \int \cL_t\p_Y\Omega \p_Y\Omega b^{2m} dXdY +2\int \bigl(|\p_X\Omega|+|\p_Y\Omega| \bigr)|\p_Y\Omega| b^{2m}\\
&\leq -\tt^{-2}\|\p_X\p_Y\Omega\|_{L^2(m)}^2-2\|\p_Y^2\Omega\|_{L^2(m)}^2\\
&\quad +C_m \|\p_Y\Omega\|_{L^2(m)}^2+2\|\p_X\Omega\|_{L^2(m)}\|\p_Y\Omega\|_{L^2(m)}.
\end{aligned}
\end{equation}

To deal with  the nonlinear part in \eqref{eq3.15}, we first decompose it to be
\begin{align*}
\int \p_Y \cN_t\Omega \p_Y\Omega b^{2m}=&\f{1}{\nu (1+\f{t^2}{12})} \Bigl(\int \bigl(\p_Y \Delta_t^{-1}\Omega\p_X \p_Y\Omega
- \p_X\Delta_t^{-1}\Omega \p_Y^2\Omega\bigr)\p_Y\Omega b^{2m}\\
&+\int \bigl(\p_Y^2 \Delta_t^{-1}\Omega\p_X \Omega
- \p_X\p_Y\Delta_t^{-1}\Omega \p_Y\Omega\bigr)\p_Y\Omega b^{2m}\Bigr)
\eqdef I_1+I_2.
\end{align*}
 It follows from a similar derivation of \eqref{eq3.9} that
$$
I_1\leq \f{C_{m}}{\nu} {\|\Omega(1)\|_{L^2(m)}} \|\p_Y\Omega\|_{L^2(m)}^2.
$$
For the second integral $I_2$, we write
\begin{align*}
I_2 \leq \f1\nu \tt^{-2}\Bigl(&\|\p_Y^2 \Delta_t^{-1}\Omega\|_{L^2}\|b^m\p_X\Omega\|_{L^2_X(L^\oo_Y)}\|b^m\p_Y\Omega\|_{L^2_Y(L^\oo_X)}\\
&+\|\p_X\p_Y\Delta_t^{-1}\Omega\|_{L^2}\|b^m\p_Y\Omega\|_{L^2_X(L^\oo_Y)}\|b^m\p_Y\Omega\|_{L^2_Y(L^\oo_X)}\Bigr).
\end{align*}
Whereas  we get, by using one-dimensional interpolation inequality, that
\begin{equation}\label{Sobolev in x}
\begin{aligned}
\|b^m f\|_{L^2_Y(L^\oo_X)}&\lesssim \|b^m f\|_{L^2}^\f12
\|\p_X(b^m f)\|_{L^2}^\f12\\
&\lesssim \|f\|_{L^2(m)}^\f12\Bigl(\|f\|_{L^2(m-1)}+\|\p_X f\|_{L^2(m)} \Bigr)^\f12\\
&\lesssim \|f\|_{L^2(m)}^\f12\|\p_X f\|_{L^2(m)}^\f12
\end{aligned}
\end{equation}
and similarly
\begin{equation}\label{Sobolev in y}
\| b^m f\|_{L^2_X(L^\oo_Y)}\lesssim \|f\|_{L^2_m}^\f12\|\p_Y f\|_{L^2_m}^\f12
\end{equation}
which together with \eqref{eq:lem3.1b} ensures that
\begin{align*}
I_2  \lesssim &\f1\nu \tt^{-2} \Bigl( \tt^\f12 \|\Omega(1)\|_{L^2(m)}\|\p_X\Omega\|_{L^2(m)}^\f12\|\p_Y\Omega\|_{L^2(m)}^\f12 \|\p_X\p_Y\Omega\|_{L^2(m)}\\
&\qquad\quad + \tt^\f32 \|\Omega(1)\|_{L^2(m)}\|\p_Y\Omega\|_{L^2(m)}\|\p_X\p_Y\Omega\|_{L^2(m)}^\f12\|\p_Y^2\Omega\|_{L^2(m)}^\f12\Bigr)\\
\lesssim  & \nu^{-1}\|\Omega(1)\|_{L^2(m)} \bigl(\ln{t}/{t_0}\bigr)^{-\f12}  \Bigl( \bigl(\tt^{-1}\|\p_X\Omega\|_{L^2(m)}\bigr)^\f12\|\p_Y\Omega\|_{L^2(m)}^\f12 \\
&\qquad\qquad\qquad\qquad\qquad\qquad\times\bigl( (\ln{t}/{t_0})^\f12\tt^{-1}\|\p_X\p_Y\Omega\|_{L^2(m)}\bigr)\\
&\qquad\quad + \|\p_Y\Omega\|_{L^2(m)}\bigl( (\ln{t}/{t_0})^\f12\tt^{-1}\|\p_X\p_Y\Omega\|_{L^2(m)}\bigr)^\f12\bigl((\ln{t}/{t_0})^\f12\|\p_Y^2\Omega\|_{L^2(m)}\bigr)^\f12\Bigr)\\
\lesssim & \nu^{-1}\|\Omega(1)\|_{L^2(m)} \bigl(\ln{t}/{t_0}\bigr)^{-\f12}  D(t).
\end{align*}
By summarizing the estimates of $I_1$ and $I_2$, we arrive at
\begin{equation}\label{eq3.18}
\begin{aligned}
&\int \p_Y \cN_t\Omega \p_Y\Omega b^{2m} \leq \f{C_m}\nu \|\Omega(1)\|_{L^2(m)} \bigl(\ln{t}/{t_0}\bigr)^{-\f12}  D(t).
\end{aligned}
\end{equation}

By substituting \eqref{eq3.17} and \eqref{eq3.18} into \eqref{eq3.15},  we conclude the proof of \eqref{eq:lem2}.
\end{proof}

\begin{lem}\label{S3lem5}
{\sl Under the assumptions of Lemma \ref{S3lem4}, we have
\begin{equation}\label{eq:lem3}
\begin{aligned}
&t\f{d}{dt}\Bigl( \bigl(\ln{t}/{t_0}\bigr)^2 \bigl(\p_X \Omega \big|\p_Y \Omega \bigr)_{L^2(m)} \Bigr)+\bigl(\ln{t}/{t_0}\bigr)^2\|\p_X\Omega\|_{L^2(m)}^2\\
&\leq 2\bigl(\ln{t}/{t_0}\bigr)\bigl(\p_X \Omega \big|\p_Y \Omega \bigr)_{L^2(m)}+C_m\bigl(\ln{t}/{t_0}\bigr)^2\Bigl(\tt^{-2}\|\p_X^2\Omega\|_{L^2(m)}\|\p_X\p_Y\Omega\|_{L^2(m)}
\\
&\qquad+\|\p_Y^2\Omega\|_{L^2(m)}\|\p_X\p_Y\Omega\|_{L^2(m)}\Bigr) +C_m\bigl(1+\nu^{-1}{\|\Omega(1)\|_{L^2(m)}}\bigr)\bigl(\ln{t}/{t_0}\bigr)^\f12 D(t) .
\end{aligned}\end{equation}
}\end{lem}
\begin{proof} We first get, by a direct computation of the time derivative and using the equation of \eqref{eq:Om}, that
\begin{equation}\label{eq3.22}
\begin{aligned}
&t\f{d}{dt}\Bigl( \bigl(\ln{t}/{t_0}\bigr)^2 \bigl(\p_X \Omega \big|\p_Y \Omega \bigr)_{L^2(m)} \Bigr)=2\bigl(\ln{t}/{t_0}\bigr)\bigl(\p_X \Omega \big|\p_Y \Omega \bigr)_{L^2(m)} \\
&\qquad+\bigl(\ln{t}/{t_0}\bigr)^2 \Bigl(\int\p_X \bigl(\cL_t\Omega+\cN_t\Omega \bigr)\p_Y\Omega b^{2m}
+ \int\p_X \Omega \p_Y\bigl(\cL_t\Omega+\cN_t\Omega \bigr)b^{2m}\Bigr).
\end{aligned}
\end{equation}

It is easy to observe from \eqref{eq:Oma} that
\begin{equation}\label{[px;L]}
\p_X \cL_t=\cL_t\p_X +\f1{2(1+\f{t^2}{3})} \p_X +\f{t(2+\f{t^2}3)}{4(1+\f{t^2}3)\sqrt{1+\f{t^2}{12}}}\p_Y
\end{equation}
from which and \eqref{[py;L]}, we infer
\begin{align*}
\int&\p_X \cL_t\Omega\p_Y\Omega b^{2m}
+ \int\p_X \Omega \p_Y\cL_t\Omega b^{2m}\\
=&\int\cL_t\p_X \Omega\p_Y\Omega b^{2m}
+ \int\p_X \Omega\cL_t \p_Y\Omega b^{2m} -\f{t\sqrt{1+\f{t^2}{12}}}{{1+\f{t^2}3}}\|\p_X\Omega\|_{L^2(m)}^2 \\
&+\f{t(2+\f{t^2}3)}{4(1+\f{t^2}3)\sqrt{1+\f{t^2}{12}}}\|\p_Y\Omega\|_{L^2(m)}^2 +\f{2+\f{t^2}3}{2(1+\f{t^2}{12})}\int \p_X\Omega \p_Y\Omega b^{2m}.
\end{align*}
While we get, by using integration by parts, that
\begin{align*}
\int\cL_t\p_X \Omega\p_Y\Omega b^{2m}
&+ \int\p_X \Omega\cL_t \p_Y\Omega b^{2m}
\lesssim_m \tt^{-2}\|\p_X^2\Omega\|_{L^2(m)}\|\p_X\p_Y\Omega\|_{L^2(m)}\\
&+\tt^{-2}\|\p_X\p_Y\Omega\|_{L^2(m)}^2+\|\p_Y^2\Omega\|_{L^2(m)}\|\p_X\p_Y\Omega\|_{L^2(m)}\\
&+\|\p_X\Omega\|_{L^2(m)}\|\p_Y\Omega\|_{L^2(m)} +\tt^{-2}\|\p_X\Omega\|_{L^2(m)}^2.
\end{align*}
As a result, for $t\geq t_0$ with $t_0$ being large enough, we obtain
\begin{equation}\label{eq3.23}
\begin{aligned}
\int&\p_X \cL_t\Omega\p_Y\Omega b^{2m}+ \int\p_X \Omega \p_Y\cL_t\Omega b^{2m}\\
\leq & -\|\p_X\Omega\|_{L^2(m)}^2+C_m\bigl(\tt^{-2}\|\p_X\Omega\|_{L^2(m)}^2+\|\p_Y\Omega\|_{L^2(m)}^2+\tt^{-2}\|\p_X\p_Y\Omega\|_{L^2(m)}^2\bigr)\\
&+C_m\bigl(\tt^{-2}\|\p_X^2\Omega\|_{L^2(m)}\|\p_X\p_Y\Omega\|_{L^2(m)}+\|\p_Y^2\Omega\|_{L^2(m)}\|\p_X\p_Y\Omega\|_{L^2(m)}\bigr).
\end{aligned}
\end{equation}

For the nonlinear terms in \eqref{eq3.22}, we get, by using integration by parts, that
\begin{align*}
\int&\p_X \cN_t\Omega\p_Y\Omega b^{2m}
+ \int\p_X \Omega \p_Y\cN_t\Omega b^{2m} \\
\leq &C_m\|\cN_t\Omega\|_{L^2(m)}\bigl(\|\p_X\p_Y\Omega\|_{L^2(m)}+\|\p_X\Omega\|_{L^2(m-1)}+\|\p_Y\Omega\|_{L^2(m-1)} \bigr)\\
\leq &C_m \nu^{-1}\tt^{-2}\bigl(\|\p_X\Delta_t^{-1}\Omega\|_{L^\oo}\|\p_Y\Omega\|_{L^2(m)}+\|\p_Y\Delta_t^{-1}\Omega\|_{L^\oo}\|\p_X\Omega\|_{L^2(m)} \bigr)\|\p_X\p_Y\Omega\|_{L^2(m)},
\end{align*}
from which and \eqref{eq:lem3.1a}, we deduce that
\begin{equation}\label{eq3.24}
\begin{aligned}
\int&\p_X \cN_t\Omega\p_Y\Omega b^{2m}
+ \int\p_X \Omega \p_Y\cN_t\Omega b^{2m}\\
\leq &  \f{C_m}{\nu}{\|\Omega(1)\|_{L^2(m)}}\bigl(\|\p_Y\Omega\|_{L^2(m)}+\tt^{-1}\|\p_X\Omega\|_{L^2(m)} \bigr)\|\p_X\p_Y\Omega\|_{L^2(m)}.
\end{aligned}
\end{equation}

By substituting \eqref{eq3.23} and \eqref{eq3.24} into \eqref{eq3.22}, we achieve
\begin{align*}
t\f{d}{dt}&\Bigl( \bigl(\ln{t}/{t_0}\bigr)^2 \bigl(\p_X \Omega \big|\p_Y \Omega \bigr)_{L^2(m)} \Bigr)+\bigl(\ln{t}/{t_0}\bigr)^2\|\p_X\Omega\|_{L^2(m)}^2\leq 2\bigl(\ln{t}/{t_0}\bigr)\bigl(\p_X \Omega \big|\p_Y \Omega \bigr)_{L^2(m)}\\
&+C_m\bigl(\ln{t}/{t_0}\bigr)^2\tt^{-2}\bigl(\|\p_X^2\Omega\|_{L^2(m)}\|\p_X\p_Y\Omega\|_{L^2(m)}+\|\p_Y^2\Omega\|_{L^2(m)}\|\p_X\p_Y\Omega\|_{L^2(m)}
\bigr)\\
&+ C_m \bigl(\ln{t}/{t_0}\bigr)^\f12(1+{\nu}^{-1}\|\Omega(1)\|_{L^2(m)})\Bigl(\|\p_Y\Omega\|_{L^2(m)}^2+\tt^{-1}\|\p_X\Omega\|_{L^2(m)}^2 \\
&\qquad +\bigl(\ln{t}/{t_0}\bigr)\tt^{-2}\|\p_X\p_Y\Omega\|_{L^2(m)}^2+\bigl(\ln{t}/{t_0}\bigr)^3\|\p_X\p_Y\Omega\|_{L^2(m)}^2\Bigr),
\end{align*}
which leads to \eqref{eq:lem3}. This finishes the proof of Lemma \ref{S3lem5}.
\end{proof}

\begin{lem}\label{S3lem6}
{\sl Under the assumptions of Lemma \ref{S3lem4}, for  $0<\sigma<\f12,$ we have
\begin{equation}\label{eq:lem4}
\begin{aligned}
t \f{d}{dt}&\Bigl(\bigl(\ln{t}/{t_0}\bigr)^3\|\p_X\Omega\|_{L^2(m)}^2\Bigr) +\bigl(\ln{t}/{t_0}\bigr)^3\tt^{-2}\|\p_X^2\Omega\|_{L^2(m)}^2+\bigl(\ln{t}/{t_0}\bigr)^3\|\p_X\p_Y\Omega\|_{L^2(m)}^2\\
\leq &3\bigl(\ln{t}/{t_0}\bigr)^2\|\p_X\Omega\|_{L^2(m)}^2+C_m\bigl(1+\nu^{-1}{\|\Omega(1)\|_{L^2(m)}}\bigr)E(t)\\
&+C_{m,\sigma}\nu^{-1} \tt^{-\f12} \bigl(\ln{t}/{t_0}\bigr)^{\f\sigma2+\f{1}4} E(t)^\f12 D(t).
\end{aligned}\end{equation}
}\end{lem}
\begin{proof} We get,
by first applying $\p_X$ to \eqref{eq:Om} and then taking  $L^2(m)$ inner product of the resulting equation with $\bigl(\ln{t}/{t_0}\bigr)^3\p_X\Omega$, that
\begin{equation}\label{eq3.27}\begin{aligned}
\f{t}2 \f{d}{dt}\Bigl(\bigl(\ln{t}/{t_0}\bigr)^3\|\p_X\Omega(t)\|_{L^2(m)}^2\Bigr) =&\f32 \bigl(\ln{t}/{t_0}\bigr)^2\|\p_X\Omega\|_{L^2(m)}^2\\
&+\bigl(\ln{t}/{t_0}\bigr)^3\int \p_X\bigl(\cL_t\Omega+\cN_t \Omega\bigr) \p_X\Omega b^{2m}.
\end{aligned}
\end{equation}
It follows from \eqref{[px;L]} and \eqref{eq3.8}  that for $t\geq t_0$ with $t_0$ being large enough,
\begin{equation}\label{eq3.28}
\begin{aligned}
\int \p_X\cL_t\Omega \p_X\Omega b^{2m}
\leq &\int \cL_t\p_X\Omega \p_X\Omega b^{2m} +2\int \bigl(|\p_X\Omega|+|\p_Y\Omega| \bigr)|\p_X\Omega| b^{2m} \\
\leq & -\tt^{-2}\|\p_X^2\Omega\|_{L^2(m)}^2-2\|\p_X\p_Y\Omega\|_{L^2(m)}^2\\
&+C_m \bigl(\|\p_X\Omega\|_{L^2(m)}^2+\|\p_Y\Omega\|_{L^2(m)}^2\bigr).
\end{aligned}
\end{equation}

For the nonlinear part in \eqref{eq3.27}, we first decompose it as
\begin{align*}
\int \p_X \cN_t\Omega \p_X\Omega b^{2m}=&\f{1}{\nu (1+\f{t^2}{12})} \Bigl(\int \bigl(\p_Y \Delta_t^{-1}\Omega\p_X^2\Omega
- \p_X\Delta_t^{-1}\Omega \p_Y\p_X\Omega\bigr)\p_X\Omega b^{2m}\\
&+\int \bigl(\p_X\p_Y \Delta_t^{-1}\Omega\p_X \Omega
- \p_X^2\Delta_t^{-1}\Omega \p_Y\Omega\bigr)\p_X\Omega b^{2m} \Bigr)
\eqdef A_1+A_2.
\end{align*}
We deduce from a similar derivation of \eqref{eq3.9} that
$$
A_1\leq \f{C_{m}}{\nu} \|\Omega(1)\|_{L^2(m)} \|\p_X\Omega\|_{L^2(m)}^2.
$$
For $A_2$, we write
\begin{align*}
A_2 \leq \nu^{-1} \tt^{-2}\bigl(\|\p_X\p_Y \Delta_t^{-1}\Omega\|_{L^\oo}\|\p_X\Omega\|_{L^2(m)}^2+\|\p_X^2\Delta_t^{-1}\Omega\|_{L^\oo}\|\p_Y\Omega\|_{L^2(m)}\|\p_X\Omega\|_{L^2(m)}\bigr),
\end{align*}
from which and   \eqref{eq3.6}, we infer
\begin{align*}
A_2  \lesssim &\nu^{-1} \tt^{\sigma-1} \Bigl(  \|\p_Y\Omega\|_{L^2(m)}^{\f12+\sigma}\|\p_X\p_Y\Omega\|_{L^2(m)}^{\f12-\sigma}\|\p_X\Omega\|_{L^2(m)}^2\\
& + \|\p_X\Omega\|_{L^2(m)}^{\f12+\sigma}\|\p_X^2\Omega\|_{L^2(m)}^{\f12-\sigma}\|\p_Y\Omega\|_{L^2(m)}\|\p_X\Omega\|_{L^2(m)}\Bigr)\\
\lesssim & \nu^{-1} \tt^{-\f12}\bigl(\ln{t}/{t_0}\bigr)^{\f\sigma2-\f{11}4} \Bigl(  \bigl(\bigl(\ln{t}/{t_0}\bigr)^\f12\|\p_Y\Omega\|_{L^2(m)}\bigr)^{\f12+\sigma}\bigl(\bigl(\ln{t}/{t_0}\bigr)^\f12\tt^{-1}\|\p_X\p_Y\Omega\|_{L^2(m)}\bigr)^{\f12-\sigma}\\
&\quad\times\bigl(\bigl(\ln{t}/{t_0}\bigr)^\f32\|\p_X\Omega\|_{L^2(m)}\bigr)^{\f12-\sigma}\bigl(\bigl(\ln{t}/{t_0}\bigr)\|\p_X\Omega\|_{L^2(m)}\bigr)^{\f32+\sigma}\\
&+\bigl(\bigl(\ln{t}/{t_0}\bigr) \|\p_X\Omega\|_{L^2(m)}\bigr)^{\f32+\sigma}\bigl(\bigl(\ln{t}/{t_0}\bigr)^\f32\tt^{-1}\|\p_X^2\Omega\|_{L^2(m)}\bigr)^{\f12-\sigma}\bigl(\ln{t}/{t_0}\bigr)^\f12\|\p_Y\Omega\|_{L^2(m)}\Bigr)\\
\lesssim & \nu^{-1} \tt^{-\f12} \bigl(\ln{t}/{t_0}\bigr)^{\f\sigma2-\f{11}4} \Bigl( \bigl(\ln{t}/{t_0}\bigr)\|\p_Y\Omega\|_{L^2(m)}^2+ \bigl(\ln{t}/{t_0}\bigr)^3\|\p_X\Omega\|_{L^2(m)}^2 \Bigr)^\f12 \\
&\quad \times \Bigl(\bigl(\ln{t}/{t_0}\bigr)\tt^{-2}\|\p_X\p_Y\Omega\|_{L^2(m)}^2+\bigl(\ln{t}/{t_0}\bigr)^2 \|\p_X\Omega\|_{L^2(m)}^2+ \bigl(\ln{t}/{t_0}\bigr)^3\tt^{-2}\|\p_X^2\Omega\|_{L^2(m)}^2\Bigr)
\end{align*}
By combining the estimates of $A_1$ and $A_2$, we arrive at
\begin{equation}\label{eq3.31}
\begin{aligned}
\int \p_X \cN_t\Omega \p_X\Omega b^{2m} \leq &\f{C_m}{\nu}{\|\Omega(1)\|_{L^2(m)}}\|\p_X\Omega\|_{L^2(m)}^2\\
&+\f{C_{m,\sigma}}{\nu}\tt^{-\f12} \bigl(\ln{t}/{t_0}\bigr)^{\f\sigma2-\f{11}4} E(t)^\f12 D(t).
\end{aligned}
\end{equation}

By substituting  \eqref{eq3.28} and \eqref{eq3.31} into \eqref{eq3.27}, we conclude
\begin{align*}
t &\f{d}{dt}\Bigl(\bigl(\ln{t}/{t_0}\bigr)^3\|\p_X\Omega\|_{L^2(m)}^2\Bigr) +\bigl(\ln{t}/{t_0}\bigr)^3\tt^{-2}\|\p_X^2\Omega\|_{L^2(m)}^2+\bigl(\ln{t}/{t_0}\bigr)^3\|\p_X\p_Y\Omega\|_{L^2(m)}^2\\
\leq &3\bigl(\ln{t}/{t_0}\bigr)^2\|\p_X\Omega\|_{L^2(m)}^2+C_m\bigl(1+\nu^{-1}{\|\Omega(1)\|_{L^2(m)}}\bigr)\bigl(\ln{t}/{t_0}\bigr)^3
\bigl(\|\p_X\Omega\|_{L^2(m)}^2+\|\p_Y\Omega\|_{L^2(m)}^2\bigr)\\
&+C_{m,\sigma} \nu^{-1} \tt^{-\f12} \bigl(\ln{t}/{t_0}\bigr)^{\f\sigma2+\f{1}4} E(t)^\f12 D(t),
\end{align*}
which gives rise to \eqref{eq:lem4}. This finishes the proof of Lemma \ref{S3lem6}.
\end{proof}

\begin{lem}\label{S3lem7}
{\sl  Under the assumptions of Lemma \ref{S3lem6}, we have
\begin{equation}\label{eq:lem5}
\begin{aligned}
t &\f{d}{dt}\Bigl(\bigl(\ln{t}/{t_0}\bigr)^2\|\p_Y^2\Omega\|_{L^2(m)}^2\Bigr) +\bigl(\ln{t}/{t_0}\bigr)^2\tt^{-2}\|\p_X\p_Y^2\Omega\|_{L^2(m)}^2\\
&+\bigl(\ln{t}/{t_0}\bigr)^2\|\p_Y^3\Omega\|_{L^2(m)}^2
\leq8\bigl(\ln{t}/{t_0}\bigr)^2\|\p_X\p_Y\Omega\|_{L^2(m)}\|\p_Y^2\Omega\|_{L^2(m)}\\
&+ 2\bigl(\ln{t}/{t_0}\bigr)\|\p_Y^2\Omega\|_{L^2(m)}^2+C_m\bigl(1+{\nu}^{-1}{\|\Omega(1)\|_{L^2(m)}}\bigr) \bigl(\ln {t}/{t_0}\bigr)^\f12 D(t) \\
&+C_{m,\sigma}{\nu}^{-1}\tt^{-\f12} \bigl(\ln {t}/{t_0}\bigr)^{\f\sigma2+\f14} E(t)^\f12D(t).
\end{aligned}\end{equation}
}\end{lem}
\begin{proof}
By applying $\p_Y^2$ to \eqref{eq:Om} and then taking $L^2(m)$ inner product of the resulting equation with $\bigl(\ln{t}/{t_0}\bigr)^2\p_Y^2\Omega$, we compute
\begin{equation}\label{eq3.33}
\begin{split}
\f{t}2 \f{d}{dt}\Bigl(\bigl(\ln{t}/{t_0}\bigr)^2\|\p_Y\Omega\|_{L^2(m)}^2\Bigr) =&\bigl(\ln{t}/{t_0}\bigr)\|\p_Y^2\Omega\|_{L^2(m)}^2\\
&+\bigl(\ln{t}/{t_0}\bigr)^2\int \p_Y^2\Bigl(\cL_t\Omega+\cN_t \Omega\Bigr) \p_Y^2\Omega b^{2m}.
\end{split}
\end{equation}

By using \eqref{[py;L]} twice, we find
$$
\p_Y^2\cL_t=\cL_t\p_Y^2 -\f{2t\sqrt{1+\f{t^2}{12}}}{{1+\f{t^2}3}}\p_X\p_Y+\Bigl( \f{t^2}{4(1+\f{t^2}3)(1+\f{t^2}{12})}+\f{1+\f{t^2}3}{(1+\f{t^2}{12})}\Bigr)\p_Y^2,
$$
from which and \eqref{eq3.8}, we deduce that for $t\geq t_0$ with $t_0$ being large enough,
\begin{equation}\label{eq3.34}
\begin{aligned}
\int \p_Y^2\cL_t\Omega \p_Y^2\Omega b^{2m}
\leq &\int \cL_t\p_Y^2\Omega \p_Y^2\Omega b^{2m}+4\int \bigl(|\p_X\p_Y\Omega|+|\p_Y^2\Omega| \bigr)|\p_Y^2\Omega| b^{2m}\\
\leq &-\tt^{-2}\|\p_X\p_Y^2\Omega\|_{L^2(m)}^2-2\|\p_Y^3\Omega\|_{L^2(m)}^2\\
&+C_m \|\p_Y^2\Omega\|_{L^2(m)}^2+4\|\p_X\p_Y\Omega\|_{L^2(m)}\|\p_Y^2\Omega\|_{L^2(m)}.
\end{aligned}
\end{equation}

To deal with the nonlinear part in \eqref{eq3.33}, we  decompose it as
\begin{align*}
\int \p_Y^2 \cN_t\Omega \p_Y^2\Omega b^{2m}=&\f{1}{\nu (1+\f{t^2}{12})} \Bigl(\int \bigl(\p_Y \Delta_t^{-1}\Omega\p_X\p_Y^2\Omega
- \p_X\Delta_t^{-1}\Omega \p_Y^3\Omega\bigr)\p_Y^2\Omega b^{2m}\\
&+2\int \bigl(\p_Y^2 \Delta_t^{-1}\Omega\p_X\p_Y \Omega
- \p_X\p_Y\Delta_t^{-1}\Omega \p_Y^2\Omega\bigr)\p_Y^2\Omega b^{2m}\\
&+\int \bigl(\p_Y^3 \Delta_t^{-1}\Omega\p_X \Omega
- \p_X\p_Y^2\Delta_t^{-1}\Omega \p_Y\Omega\bigr)\p_Y^2\Omega b^{2m} \Bigr)
\eqdef B_1+B_2+B_3.
\end{align*}
Along the same line as the derivation of \eqref{eq3.9}, we find
$$
B_1\leq C_{m}{\nu}^{-1} {\|\Omega(1)\|_{L^2(m)}} \|\p_Y^2\Omega\|_{L^2(m)}^2.
$$
For  $B_2$, we write
\begin{align*}
B_2 \leq \f2\nu \tt^{-2}\|b^m\p_Y^2\Omega\|_{L^2_Y(L^\oo_X)}\bigl(&\|\p_Y^2 \Delta_t^{-1}\Omega\|_{L^2}\|b^m\p_X\p_Y\Omega\|_{L^2_X(L^\oo_Y)}\\
&+\|\p_X\p_Y\Delta_t^{-1}\Omega\|_{L^2}\|b^m\p_Y^2\Omega\|_{L^2_X(L^\oo_Y)}\bigr),
\end{align*}
from which,  \eqref{eq:lem3.1b}, \eqref{Sobolev in x} and \eqref{Sobolev in y}, we infer
\begin{align*}
B_2  \leq  &\f{C}\nu \tt^{-2} \|\p_Y^2\Omega\|_{L^2(m)}^\f12\|\p_X\p_Y^2\Omega\|_{L^2(m)}^\f12
\Bigl( \tt^\f12\|\Omega(1)\|_{L^2(m)}\|\p_X\p_Y\Omega\|_{L^2(m)}^{\f12}\|\p_X\p_Y^2\Omega\|_{L^2(m)}^{\f12}\\
&\qquad+  \tt^\f32\|\Omega(1)\|_{L^2(m)} \|\p_Y^2\Omega\|_{L^2(m)}^\f12\|\p_Y^3\Omega\|_{L^2(m)}^\f12\Bigr)\\
\leq  &\f{C}\nu{\|\Omega(1)\|_{L^2(m)}} \bigl(\ln{t}/{t_0}\bigr)^{-\f32}\Bigl( \bigl(\ln{t}/{t_0}\bigr)^\f12 \|\p_Y^2\Omega\|_{L^2(m)}\Bigr)^\f12\Bigl(\bigl(\ln{t}/{t_0}\bigr)\tt^{-1}\|\p_X\p_Y^2\Omega\|_{L^2(m)}\Bigr)^\f12\\
&\quad\times\Bigl( \bigl((\ln{t}/{t_0})^\f12\tt^{-1}\|\p_X\p_Y\Omega\|_{L^2(m)}\bigr)^{\f12}\bigl( (\ln{t}/{t_0})\tt^{-1}\|\p_X\p_Y^2\Omega\|_{L^2(m)}\bigr)^{\f12}\\
&\qquad+  \bigl( (\ln{t}/{t_0})^\f12 \|\p_Y^2\Omega\|_{L^2(m)}\bigr)^\f12\bigl( (\ln{t}/{t_0})\|\p_Y^3\Omega\|_{L^2(m)}\bigr)^\f12\Bigr)\\
\leq &\f{ C_m}\nu  {\|\Omega(1)\|_{L^2(m)}}  \bigl(\ln{t}/{t_0}\bigr)^{-\f32} D(t).
\end{align*}
Finally, we get, by using   \eqref{eq3.6} that
\begin{align*}
B_3  \leq & \f1\nu \tt^{-2}\bigl(\|\p_Y^3 \Delta_t^{-1}\Omega\|_{L^\oo}\|\p_X\Omega\|_{L^2(m)}+\|\p_X\p_Y^2\Delta_t^{-1}\Omega\|_{L^\oo}\|\p_Y\Omega\|_{L^2(m)}\bigr)\|\p_Y^2\Omega\|_{L^2(m)}\\
\leq &\f{C_\sigma}\nu \tt^{\sigma-1} \|\p_Y^2\Omega\|_{L^2(m)}\|\p_Y^2\Omega\|_{L^2(m)}^{\f12+\sigma}\|\p_X\p_Y^2\Omega\|_{L^2(m)}^{\f12-\sigma}\bigl(  \tt^{-1}\|\p_X\Omega\|_{L^2(m)}+\|\p_Y\Omega\|_{L^2(m)}\bigr)\\
\leq &\f{C_\sigma}\nu \tt^{-\f12}\bigl(\ln{t}/{t_0}\bigr)^{\f\sigma2-\f{7}4} \Bigl(\bigl(\ln{t}/{t_0}\bigr)\|\p_Y^2\Omega\|_{L^2(m)}\Bigr)
 \Bigl(\bigl(\ln{t}/{t_0}\bigr)^\f12\|\p_Y^2\Omega\|_{L^2(m)}\Bigr)^{\f12+\sigma}\\
&\quad\times\Bigl(\bigl(\ln{t}/{t_0}\bigr)\tt^{-1}\|\p_X\p_Y^2\Omega\|_{L^2(m)}\Bigr)^{\f12-\sigma}\bigl(  \tt^{-1}\|\p_X\Omega\|_{L^2(m)}+\|\p_Y\Omega\|_{L^2(m)} \bigr)\\
\leq &\f{C_{m,\sigma}}\nu \tt^{-\f12} \bigl(\ln{t}/{t_0}\bigr)^{\f\sigma2-\f{7}4} E(t)^\f12 D(t).
\end{align*}
As a consequence, we obtain
\begin{equation}\label{eq3.35}
\begin{aligned}
\int \p_Y^2 \cN_t\Omega \p_Y^2\Omega b^{2m} \leq & \f{C_m}{\nu} {\|\Omega(1)\|_{L^2(m)}}\bigl(\ln {t}/{t_0}\bigr)^{-\f32} D(t) \\
&+\f{C_{m,\sigma}}{\nu}\tt^{-\f12} \bigl(\ln {t}/{t_0}\bigr)^{\f\sigma2-\f74} E(t)^\f12D(t) .
\end{aligned}
\end{equation}

By substituting \eqref{eq3.34} and \eqref{eq3.35} into \eqref{eq3.33}, we conclude the proof of \eqref{eq:lem5}.
\end{proof}

\begin{lem}\label{S3lem8}
{\sl Let $m>1$. For $t_0\leq t\leq 2t_0$, one has
\begin{equation}\label{eq:lem6}
\begin{aligned}
t &\f{d}{dt}\Bigl(\bigl(\ln{t}/{t_0}\bigr)^4\|\p_X\p_Y\Omega\|_{L^2(m)}^2\Bigr) +\bigl(\ln{t}/{t_0}\bigr)^4\tt^{-2}\|\p_X^2\p_Y\Omega\|_{L^2(m)}^2\\
&+\bigl(\ln{t}/{t_0}\bigr)^4\|\p_X\p_Y^2\Omega\|_{L^2(m)}^2
\leq 4\bigl(\ln{t}/{t_0}\bigr)^3\|\p_X\p_Y\Omega\|_{L^2(m)}^2\\
&+4\bigl(\ln{t}/{t_0}\bigr)^4\|\p_X^2\Omega\|_{L^2(m)}\|\p_X\p_Y\Omega\|_{L^2(m)}+C_m \bigl(1+\nu^{-1}{\|\Omega(1)\|_{L^2(m)}}\bigr) \bigl(\ln{t}/{t_0}\bigr)^{\f12} D(t).
\end{aligned}\end{equation}
}\end{lem}
\begin{proof} We first get,
by applying $\p_X\p_Y$ to \eqref{eq:Om} and then taking  $L^2(m)$ inner product of the resulting equation with $\bigl(\ln{t}/{t_0}\bigr)^4\p_X\p_Y\Omega$, that
\begin{equation}\label{eq3.37}\begin{aligned}
\f{t}2 \f{d}{dt}\Bigl(\bigl(\ln{t}/{t_0}\bigr)^4\|\p_X\p_Y\Omega(t)\|_{L^2(m)}^2\Bigr) =&2 \bigl(\ln{t}/{t_0}\bigr)^3\|\p_X\p_Y\Omega\|_{L^2(m)}^2\\
&+\bigl(\ln{t}/{t_0}\bigr)^4\int \p_X\p_Y\bigl(\cL_t\Omega+\cN_t \Omega\bigr) \p_X\p_Y\Omega b^{2m}.
\end{aligned}
\end{equation}

It follows from \eqref{[px;L]} and \eqref{[py;L]} that
\begin{equation}\label{[pypx;L]}
\begin{aligned}
\p_X\p_Y\cL_t=\cL_t\p_X\p_Y -\f{t\sqrt{1+\f{t^2}{12}}}{{1+\f{t^2}3}}\p_X^2+\f{2+\f{t^2}3}{2(1+\f{t^2}{12})}\p_X\p_Y+\f{t(2+\f{t^2}3)}{4(1+\f{t^2}3)\sqrt{1+\f{t^2}{12}}}\p_Y^2,
\end{aligned}
\end{equation}
from which and \eqref{eq3.8}, we deduce that that for $t\geq t_0$ with $t_0$ being large enough,
\begin{equation}\label{eq3.38}
\begin{aligned}
&\int \p_X\p_Y\cL_t\Omega \p_X\p_Y\Omega b^{2m} \\
&\leq \int \cL_t\p_X\p_Y\Omega \p_X\p_Y\Omega b^{2m} +2\int \bigl(|\p_X^2\Omega|+|\p_X\p_Y\Omega|+|\p_Y^2\Omega| \bigr)|\p_X\p_Y\Omega| b^{2m} \\
&\leq -\tt^{-2}\|\p_X^2\p_Y\Omega\|_{L^2(m)}^2-2\|\p_X\p_Y^2\Omega\|_{L^2(m)}^2+C_m \bigl(\|\p_X\p_Y\Omega\|_{L^2(m)}^2+\|\p_Y^2\Omega\|_{L^2(m)}^2\bigr)\\
&\quad+2\|\p_X^2\Omega\|_{L^2(m)}\|\p_X\p_Y\Omega\|_{L^2(m)}.
\end{aligned}
\end{equation}

To deal with the nonlinear part in \eqref{eq3.37}, we get, by using integration by parts, that
\begin{align*}
\int \p_X \p_Y\cN_t\Omega \p_X\p_Y\Omega b^{2m}
&=-\int \p_Y\cN_t\Omega \bigl(\p_X^2\p_Y\Omega b^{2m}+2m\p_X\p_Y \Omega b^{2m-2}X\bigr)\\
&\leq C_m\|\p_Y\cN_t\Omega \|_{L^2(m)} \bigl( \|\p_X^2\p_Y\Omega\|_{L^2(m)}+\|\p_X\p_Y\Omega\|_{L^2(m-1)} \bigr)\\
&\leq C_m \|\p_Y\cN_t\Omega \|_{L^2(m)}  \|\p_X^2\p_Y\Omega\|_{L^2(m)}.
\end{align*}
 Whereas we deduce from \eqref{eq:lem3.1a}, \eqref{eq:lem3.1b}, \eqref{Sobolev in x} and \eqref{Sobolev in y} that
\begin{align*}
\|\p_Y\cN_t\Omega\|_{L^2(m)}\lesssim &\nu^{-1}\tt^{-2} \Bigl( \|\p_X\Delta_t^{-1}\Omega\|_{L^\oo}\|\p_Y^2\Omega\|_{L^2(m)}+\|\p_Y\Delta_t^{-1}\Omega\|_{L^\oo}\|\p_X\p_Y\Omega\|_{L^2(m)} \\
&+\|\p_X\p_Y\Delta_t^{-1}\Omega\|_{L^2}\|b^m\p_Y\Omega\|_{L^\oo}+\|\p_Y^2\Delta_t^{-1}\Omega\|_{L^2}\|b^m\p_X\Omega\|_{L^\oo}\Bigr)\\
\lesssim & \nu^{-1}{\|\Omega(1)\|_{L^2(m)}} \Bigl( \|\p_Y^2\Omega\|_{L^2(m)} +\tt^{-1}\|\p_X\p_Y\Omega\|_{L^2(m)}\\
&+\tt^{-\f12}\|\p_X\p_Y^2\Omega\|_{L^2(m)}^\f14 \|\p_X\p_Y\Omega\|_{L^2(m)}^\f14 \|\p_Y^2\Omega\|_{L^2(m)}^\f14 \|\p_Y\Omega\|_{L^2(m)}^\f14\\
&+\tt^{-\f32}\|\p_X^2\p_Y\Omega\|_{L^2(m)}^\f14 \|\p_X^2\Omega\|_{L^2(m)}^\f14 \|\p_X\p_Y\Omega\|_{L^2(m)}^\f14 \|\p_X\Omega\|_{L^2(m)}^\f14 \Bigr)\\
\lesssim & \nu^{-1} {\|\Omega(1)\|_{L^2(m)}}\Bigl( \|\p_Y^2\Omega\|_{L^2(m)} +\tt^{-1}\|\p_X\p_Y\Omega\|_{L^2(m)}\\
&+\bigl(\ln{t}/{t_0}\bigr)^\f12\tt^{-1}\|\p_X\p_Y^2\Omega\|_{L^2(m)}+\bigl(\ln{t}/{t_0}\bigr)^{-\f12}\|\p_Y\Omega\|_{L^2(m)}\Bigr)\\
&+\nu^{-1}{\|\Omega(1)\|_{L^2(m)}}\tt^{-1}\Bigl(\tt^{-1}\|\p_X^2\Omega\|_{L^2(m)}+\|\p_X\p_Y\Omega\|_{L^2(m)}\\
&+\bigl(\ln{t}/{t_0}\bigr)^\f12\tt^{-1}\|\p_X^2\p_Y\Omega\|_{L^2(m)}+\bigl(\ln{t}/{t_0}\bigr)^{-\f12}\|\p_X\Omega\|_{L^2(m)} \Bigr).
\end{align*}
As a result, it comes out
\begin{align*}
\int& \p_X \p_Y\cN_t\Omega \p_X\p_Y\Omega b^{2m} \\
\leq &C_m \nu^{-1}\|\Omega(1)\|_{L^2(m)} \bigl(\ln{t}/{t_0}\bigr)^{-\f72} \Bigl(\bigl(\ln{t}/{t_0}\bigr)^6 \|\p_X^2\p_Y\Omega\|_{L^2(m)}^2  \Bigr)^\f12 \Bigl( \|\p_Y\Omega\|_{L^2(m)}^2\\
&+\bigl(\ln{t}/{t_0}\bigr)\bigl( \|\p_Y^2\Omega\|_{L^2(m)}^2 +\tt^{-2}\|\p_X\p_Y\Omega\|_{L^2(m)}^2\bigr)+\bigl(\ln{t}/{t_0}\bigr)^2\tt^{-2}\|\p_X\p_Y^2\Omega\|_{L^2(m)}^2\Bigr)^\f12\\
&+C_m \nu^{-1}{\|\Omega(1)\|_{L^2(m)}} \bigl(\ln{t}/{t_0}\bigr)^{-\f72} \Bigl(\bigl(\ln{t}/{t_0}\bigr)^4 \tt^{-2} \|\p_X^2\p_Y\Omega\|_{L^2(m)}^2  \Bigr)^\f12 \Bigl( \bigl(\ln{t}/{t_0}\bigr)^2 \|\p_X\Omega\|_{L^2(m)}^2\\
 &+\bigl(\ln{t}/{t_0}\bigr)^3\bigl(\tt^{-2}\|\p_X^2\Omega\|_{L^2(m)}^2+\|\p_X\p_Y\Omega\|_{L^2(m)}^2 \bigr)+\bigl(\ln{t}/{t_0}\bigr)^4\tt^{-2}\|\p_X^2\p_Y\Omega\|_{L^2(m)}^2\Bigr)^\f12,
\end{align*}
from which and \eqref{def:D}, we infer
\begin{equation}\label{eq3.39}
\int \p_X \p_Y\cN_t\Omega \p_X\p_Y\Omega b^{2m} \leq C_m \nu^{-1}{\|\Omega(1)\|_{L^2(m)}}\bigl(\ln{t}/{t_0}\bigr)^{-\f72} D(t).
\end{equation}

By substituting \eqref{eq3.38} and \eqref{eq3.39} into \eqref{eq3.37}, we  conclude the proof of \eqref{eq:lem6}.
\end{proof}

\begin{lem}\label{S3lem9}
{\sl Let $m>1$ and $0<\sigma<\f12$. Then for $t_0\leq t\leq 2t_0 $, we have
\begin{equation}\label{eq:lem7}
\begin{aligned}
t\f{d}{dt}&\Bigl( \bigl(\ln{t}/{t_0}\bigr)^5 \bigl(\p_X^2 \Omega \big|\p_X\p_Y \Omega \bigr)_{L^2(m)} \Bigr)+\bigl(\ln{t}/{t_0}\bigr)^5\|\p_X^2\Omega\|_{L^2(m)}^2\\
\leq & 5\bigl(\ln{t}/{t_0}\bigr)^4\bigl(\p_X^2 \Omega \big|\p_X\p_Y \Omega \bigr)_{L^2(m)}+C_m\bigl(\ln{t}/{t_0}\bigr)^5\Bigl(\tt^{-2}\|\p_X^3\Omega\|_{L^2(m)}\|\p_X^2\p_Y\Omega\|_{L^2(m)}\\
&+\|\p_X^2\p_Y\Omega\|_{L^2(m)}\|\p_X\p_Y^2\Omega\|_{L^2(m)}\Bigr)+C_m \bigl(1+{\nu}^{-1}{\|\Omega(1)\|_{L^2(m)}} \bigr)\bigl(\ln {t}/{t_0}\bigr)^{\f12} D(t)\\
&+C_{m,\sigma} \nu^{-1} \tt^{-\f12} \bigl(\ln {t}/{t_0}\bigr)^{\f\sigma2+\f14} E(t)^\f12 D(t).
\end{aligned}\end{equation}
}\end{lem}
\begin{proof} In view of \eqref{eq:Om}, we write
\begin{equation}\label{eq3.42}
\begin{aligned}
t&\f{d}{dt}\Bigl( \bigl(\ln{t}/{t_0}\bigr)^5 \bigl(\p_X^2\Omega \big|\p_X\p_Y \Omega \bigr)_{L^2(m)} \Bigr)=5\bigl(\ln{t}/{t_0}\bigr)^4\bigl(\p_X^2 \Omega \big|\p_X\p_Y \Omega \bigr)_{L^2(m)} \\
&+\bigl(\ln{t}/{t_0}\bigr)^5 \Bigl(\int\p_X^2 \bigl(\cL_t\Omega+\cN_t\Omega \bigr)\p_X\p_Y\Omega b^{2m}
+ \int\p_X^2 \Omega \p_X\p_Y\bigl(\cL_t\Omega+\cN_t\Omega \bigr)b^{2m}\Bigr).
\end{aligned}
\end{equation}

By applying \eqref{[px;L]} twice, we find
\begin{equation}\label{[pxpx;L]}
\p_X^2 \cL_t=\cL_t\p_X^2 +\f1{1+\f{t^2}{3}} \p_X^2 +\f{t(2+\f{t^2}3)}{2(1+\f{t^2}3)\sqrt{1+\f{t^2}{12}}}\p_X\p_Y,
\end{equation}
from which and \eqref{[pypx;L]}, we infer
\begin{align*}
\int&\p_X^2 \cL_t\Omega\p_X\p_Y\Omega b^{2m}
+ \int\p_X^2 \Omega \p_X\p_Y\cL_t\Omega b^{2m}\\
=&\int\cL_t\p_X^2 \Omega\p_X\p_Y\Omega b^{2m}
+ \int\p_X^2 \Omega\cL_t \p_X\p_Y\Omega b^{2m}-\f{t\sqrt{1+\f{t^2}{12}}}{{1+\f{t^2}3}}\|\p_X^2\Omega\|_{L^2(m)}^2 \\
&+\f{t(2+\f{t^2}3)}{2(1+\f{t^2}3)\sqrt{1+\f{t^2}{12}}}\|\p_X\p_Y\Omega\|_{L^2(m)}^2 +\Bigl(\f1{1+\f{t^2}{3}}+\f{2+\f{t^2}3}{2(1+\f{t^2}{12})}\Bigr)\int \p_X^2\Omega \p_X\p_Y\Omega b^{2m}\\
&+\f{t(2+\f{t^2}3)}{4(1+\f{t^2}3)\sqrt{1+\f{t^2}{12}}} \int \p_X^2\Omega \p_Y^2\Omega b^{2m}.
\end{align*}
By using integration by parts, we have
\begin{align*}
\int\cL_t\p_X^2 \Omega\p_X\p_Y\Omega b^{2m}
&+ \int\p_X^2 \Omega\cL_t \p_X\p_Y\Omega b^{2m}
\lesssim_m \tt^{-2}\|\p_X^3\Omega\|_{L^2(m)}\|\p_X^2\p_Y\Omega\|_{L^2(m)}\\
&+\tt^{-2}\|\p_X^2\p_Y\Omega\|_{L^2(m)}^2+\|\p_X^2\p_Y\Omega\|_{L^2(m)}\|\p_X\p_Y^2\Omega\|_{L^2(m)}\\
&+\|\p_X^2\Omega\|_{L^2(m)}\|\p_X\p_Y\Omega\|_{L^2(m)}+\tt^{-2}\|\p_X^2\Omega\|_{L^2(m)}^2.
\end{align*}
Therefore, for $t\geq t_0$ with $t_0$ being large enough, we find
\begin{equation}\label{eq3.44}
\begin{aligned}
&\int\p_X^2 \cL_t\Omega\p_X\p_Y\Omega b^{2m} + \int\p_X^2 \Omega \p_X\p_Y\cL_t\Omega b^{2m}\leq-\|\p_X^2\Omega\|_{L^2(m)}^2\\
 &+C_m\Bigl(\tt^{-2}\|\p_X^2\Omega\|_{L^2(m)}^2+\|\p_X\p_Y\Omega\|_{L^2(m)}^2+\|\p_Y^2\Omega\|_{L^2(m)}^2+\tt^{-2}\|\p_X^2\p_Y\Omega\|_{L^2(m)}^2\Bigr)\\
&+C_m\Bigl(\tt^{-2}\|\p_X^3\Omega\|_{L^2(m)}\|\p_X^2\p_Y\Omega\|_{L^2(m)}+\|\p_X^2\p_Y\Omega\|_{L^2(m)}\|\p_X\p_Y^2\Omega\|_{L^2(m)}\Bigr).
\end{aligned}
\end{equation}

To deal with the nonlinear terms in \eqref{eq3.42}, we use integration by parts to write
\begin{align*}
\int&\p_X^2 \cN_t\Omega\p_X\p_Y\Omega b^{2m}
+ \int\p_X^2 \Omega \p_X\p_Y\cN_t\Omega b^{2m} \\
\leq &C_m\|\p_X\cN_t\Omega\|_{L^2(m)}\bigl(\|\p_X^2\p_Y\Omega\|_{L^2(m)}+\|\p_X\p_Y\Omega\|_{L^2(m-1)}+\|\p_X^2\Omega\|_{L^2(m-1)} \bigr)\\
\leq &C_m\|\p_X\cN_t\Omega\|_{L^2(m)}\|\p_X^2\p_Y\Omega\|_{L^2(m)}.
\end{align*}
Yet it is easy to observe from \eqref{eq:Oma} that
\begin{align*}
\|\p_X\cN_t\Omega\|_{L^2(m)}\lesssim &\nu^{-1}\tt^{-2} \Bigl( \|\p_X\Delta_t^{-1}\Omega\|_{L^\oo}\|\p_X\p_Y\Omega\|_{L^2(m)}+\|\p_Y\Delta_t^{-1}\Omega\|_{L^\oo}\|\p_X^2\Omega\|_{L^2(m)} \\
&\qquad+\|\p_X\p_Y\Delta_t^{-1}\Omega\|_{L^2}\|b^m\p_X\Omega\|_{L^\oo}+\|\p_X^2\Delta_t^{-1}\Omega\|_{L^\oo}\|\p_Y\Omega\|_{L^2(m)}\Bigr).
\end{align*}
It follows from  \eqref{eq:lem3.1a} that
\begin{align*}
\nu^{-1}&\tt^{-2} \bigl( \|\p_X\Delta_t^{-1}\Omega\|_{L^\oo}\|\p_X\p_Y\Omega\|_{L^2(m)}+\|\p_Y\Delta_t^{-1}\Omega\|_{L^\oo}\|\p_X^2\Omega\|_{L^2(m)}\bigr)\\
\leq & {C}\nu^{-1}{\|\Omega(1)\|_{L^2(m)}}\bigl( \|\p_X\p_Y\Omega\|_{L^2(m)}+\tt^{-1}\|\p_X^2\Omega\|_{L^2(m)}\bigr)\\
\leq &{C_m}\nu^{-1} {\|\Omega(1)\|_{L^2(m)}}  \bigl(\ln {t}/{t_0}\bigr)^{-\f32} D(t)^\f12
\end{align*}
Along the same line, we deduce from \eqref{eq:lem3.1b} that
\begin{align*}
\nu^{-1}&\tt^{-2} \|\p_X\p_Y\Delta_t^{-1}\Omega\|_{L^2}\|b^m\p_X\Omega\|_{L^\oo}\\
\leq & C\nu^{-1}{\|\Omega(1)\|_{L^2(m)}} \tt^{-\f12} \|\p_X\Omega\|_{L^2(m)}^\f14\|\p_X\p_Y\Omega\|_{L^2(m)}^\f14\|\p_X^2\Omega\|_{L^2(m)}^\f14  \|\p_X^2\p_Y\Omega\|_{L^2(m)}^\f14\\
\leq & C\nu^{-1}{\|\Omega(1)\|_{L^2(m)}} \bigl(\ln{t}/{t_0}\bigr)^{-\f32} \Bigl( \bigl(\ln{t}/{t_0}\bigr)\|\p_X\Omega\|_{L^2(m)}\Bigr)^\f14
\Bigl( \bigl(\ln{t}/{t_0}\bigr)^\f32\|\p_X\p_Y\Omega\|_{L^2(m)}\Bigr)^\f14 \\
&\qquad\times\Bigl( \bigl(\ln{t}/{t_0}\bigr)^\f32\tt^{-1}\|\p_X^2\Omega\|_{L^2(m)}\Bigr)^\f14
\Bigl( \bigl(\ln{t}/{t_0}\bigr)^2\tt^{-1}\|\p_X^2\p_Y\Omega\|_{L^2(m)}\Bigr)^\f14\\
\leq &C_m\nu^{-1}  {\|\Omega(1)\|_{L^2(m)}}\bigl(\ln {t}/{t_0}\bigr)^{-\f32} D(t)^\f12.
\end{align*}
Whereas we deduce from \eqref{eq3.6}  that for any $0<\sigma<\f12$,
\begin{align*}
\nu^{-1}&\tt^{-2} \|\p_X^2\Delta_t^{-1}\Omega\|_{L^\oo}\|\p_Y\Omega\|_{L^2(m)}\\
\leq & C_\sigma\nu^{-1} \tt^{\sigma-1} \|\p_X\Omega\|_{L^2(m)}^{\f12+\sigma}\|\p_X^2\Omega\|_{L^2(m)}^{\f12-\sigma}\|\p_Y\Omega\|_{L^2(m)}\\
\leq & C_\sigma \nu^{-1} \tt^{-\f12} \bigl(\ln {t}/{t_0}\bigr)^{\f\sigma2-\f74} \Bigl( \bigl(\ln {t}/{t_0}\bigr)\|\p_X\Omega\|_{L^2(m)} \Bigr)^{\f12+\sigma} \Bigl( \bigl(\ln {t}/{t_0}\bigr)^{\f32}\tt^{-1}\|\p_X^2\Omega\|_{L^2(m)} \Bigr)^{\f12-\sigma} \\
&\qquad\times\Bigl( \bigl(\ln {t}/{t_0}\bigr)^\f12\|\p_Y\Omega\|_{L^2(m)}\Bigr)\\
\leq & C_{m,\sigma} \nu^{-1} \tt^{-\f12} \bigl(\ln {t}/{t_0}\bigr)^{\f\sigma2-\f74} E(t)^\f12 D(t)^\f12
\end{align*}
Noticing that
$$
\|\p_X^2\p_Y\Omega\|_{L^2(m)} \leq C_m \bigl(\ln {t}/{t_0}\bigr)^{-3} D(t)^\f12,
$$
we get, by summarizing the above estimates, that
\begin{equation}\label{eq3.45}
\begin{aligned}
&\int\p_X^2 \cN_t\Omega\p_X\p_Y\Omega b^{2m}
+ \int\p_X^2 \Omega \p_X\p_Y\cN_t\Omega b^{2m}\\
&\leq C_m {\nu}^{-1}{\|\Omega(1)\|_{L^2(m)}} \bigl(\ln {t}/{t_0}\bigr)^{-\f92} D(t)+C_{m,\sigma} \nu^{-1} \tt^{-\f12} \bigl(\ln {t}/{t_0}\bigr)^{\f\sigma2-\f{19}4} E(t)^\f12 D(t).
\end{aligned}
\end{equation}

By substituting the estimates \eqref{eq3.44} and \eqref{eq3.45} into \eqref{eq3.42}, we arrive at \eqref{eq:lem7}. This
finishes the proof of Lemma \ref{S3lem9}.
\end{proof}

\begin{lem}\label{S3lem10}
{\sl Under the assumptions of Lemma \ref{S3lem9}, we have
\begin{equation}\label{eq:lem8}
\begin{aligned}
t &\f{d}{dt}\Bigl(\bigl(\ln{t}/{t_0}\bigr)^6\|\p_X^2\Omega(t)\|_{L^2(m)}^2\Bigr) +\bigl(\ln{t}/{t_0}\bigr)^6\tt^{-2}\|\p_X^3\Omega\|_{L^2(m)}^2\\
&+\bigl(\ln{t}/{t_0}\bigr)^6\|\p_X^2\p_Y\Omega\|_{L^2(m)}^2
\leq 6\bigl(\ln{t}/{t_0}\bigr)^5\|\p_X^2\Omega\|_{L^2(m)}^2\\
&+C_m\bigl(1+\nu^{-1}{\|\Omega(1)\|_{L^2(m)}}\bigr)E(t)+C_{m,\sigma} \nu^{-1} \tt^{-\f12} \bigl(\ln{t}/{t_0}\bigr)^{\f\sigma2+\f{1}4} E(t)^\f12 D(t).
\end{aligned}\end{equation}
}\end{lem}
\begin{proof} We first get,
by applying $\p_X^2$ to \eqref{eq:Om} and then taking  $L^2(m)$ inner product of the resulting equation with $\bigl(\ln{t}/{t_0}\bigr)^6\p_X^2\Omega$, that
\begin{equation}\label{eq3.47}\begin{aligned}
\f{t}2 \f{d}{dt}\Bigl(\bigl(\ln{t}/{t_0}\bigr)^6\|\p_X^2\Omega(t)\|_{L^2(m)}^2\Bigr) =&3\bigl(\ln{t}/{t_0}\bigr)^5\|\p_X^2\Omega\|_{L^2(m)}^2\\
&+\bigl(\ln{t}/{t_0}\bigr)^6\int \p_X^2\Bigl(\cL_t\Omega+\cN_t \Omega\Bigr) \p_X^2\Omega b^{2m}.
\end{aligned}
\end{equation}

By applying \eqref{[pxpx;L]} and \eqref{eq3.8}, we deduce  that for $t\geq t_0$ with $t_0$ being large enough,
\begin{equation}\label{eq3.48}
\begin{aligned}
\int \p_X^2\cL_t\Omega \p_X^2\Omega b^{2m}
\leq &\int \cL_t\p_X^2\Omega \p_X^2\Omega b^{2m} +2\int \bigl(|\p_X^2\Omega|+|\p_X\p_Y\Omega| \bigr)|\p_X^2\Omega| b^{2m}\\
\leq &-\tt^{-2}\|\p_X^3\Omega\|_{L^2(m)}^2-2\|\p_X^2\p_Y\Omega\|_{L^2(m)}^2\\
&+C_m \bigl(\|\p_X^2\Omega\|_{L^2(m)}^2+\|\p_X\p_Y\Omega\|_{L^2(m)}^2\bigr).
\end{aligned}
\end{equation}

For the nonlinear part in \eqref{eq3.47}, we first decompose it as
\begin{align*}
\int \p_X^2 \cN_t\Omega \p_X^2\Omega b^{2m}=&\f{1}{\nu (1+\f{t^2}{12})} \Bigl(\int \bigl(\p_Y \Delta_t^{-1}\Omega\p_X^3\Omega
- \p_X\Delta_t^{-1}\Omega \p_Y\p_X^2\Omega\bigr)\p_X^2\Omega b^{2m}\\
&+2\int \bigl(\p_X\p_Y \Delta_t^{-1}\Omega\p_X^2 \Omega
- \p_X^2\Delta_t^{-1}\Omega \p_X\p_Y\Omega\bigr)\p_X^2\Omega b^{2m}\\
&+\int \bigl(\p_X^2\p_Y \Delta_t^{-1}\Omega\p_X \Omega
- \p_X^3\Delta_t^{-1}\Omega \p_Y\Omega\bigr)\p_X\Omega b^{2m}\Bigr)
\eqdef D_1+D_2+D_3.
\end{align*}
It follows from a similar derivation of \eqref{eq3.9} that
$$
D_1\leq C_{m} {\nu}^{-1}{\|\Omega(1)\|_{L^2(m)}} \|\p_X^2\Omega\|_{L^2(m)}^2.
$$
For  $D_2$, we write
\begin{align*}
D_2 \leq 2\nu^{-1} \tt^{-2}\Bigl(\|\p_X\p_Y \Delta_t^{-1}\Omega\|_{L^\oo}\|\p_X^2\Omega\|_{L^2(m)}^2
+\|\p_X^2\Delta_t^{-1}\Omega\|_{L^\oo}\|\p_X\p_Y\Omega\|_{L^2(m)}\|\p_X^2\Omega\|_{L^2(m)}\Bigr),
\end{align*}
from which and  \eqref{eq3.6}, we infer
\begin{align*}
D_2  \leq &C_\sigma\nu^{-1} \tt^{\sigma-1} \Bigl(  \|\p_Y\Omega\|_{L^2(m)}^{\f12+\sigma}\|\p_X\p_Y\Omega\|_{L^2(m)}^{\f12-\sigma}\|\p_X^2\Omega\|_{L^2(m)}^2\\
& + \|\p_X\Omega\|_{L^2(m)}^{\f12+\sigma}\|\p_X^2\Omega\|_{L^2(m)}^{\f12-\sigma}\|\p_X\p_Y\Omega\|_{L^2(m)}\|\p_X^2\Omega\|_{L^2(m)}\Bigr)\\
\leq &C_\sigma \nu^{-1} \tt^{-\f12}\bigl(\ln{t}/{t_0}\bigr)^{\f\sigma2-\f{23}4} \Bigl(  \bigl((\ln{t}/{t_0})^\f12\|\p_Y\Omega\|_{L^2(m)}\bigr)^{\f12+\sigma}\bigl((\ln{t}/{t_0})^\f12\tt^{-1}\|\p_X\p_Y\Omega\|_{L^2(m)}\bigr)^{\f12-\sigma}\\
&\times\bigl((\ln{t}/{t_0})^3\|\p_X^2\Omega\|_{L^2(m)}\bigr)^{\f12-\sigma}\bigl((\ln{t}/{t_0})^\f52\|\p_X^2\Omega\|_{L^2(m)}\bigr)^{\f32+\sigma}+\bigl((\ln{t}/{t_0}) \|\p_X\Omega\|_{L^2(m)}\bigr)^{\f12+\sigma}\\
&\times\bigl((\ln{t}/{t_0})^\f32\tt^{-1}\|\p_X^2\Omega\|_{L^2(m)}\bigr)^{\f12-\sigma}\bigl((\ln{t}/{t_0})^\f32\|\p_X\p_Y\Omega\|_{L^2(m)}\bigr)\bigl((\ln{t}/{t_0})^3\|\p_X^2\Omega\|_{L^2(m)}\bigr)\Bigr)\\
\leq &C_\sigma \nu^{-1} \tt^{-\f12} \bigl(\ln{t}/{t_0}\bigr)^{\f\sigma2-\f{23}4} \Bigl( \bigl(\ln{t}/{t_0}\bigr)\|\p_Y\Omega\|_{L^2(m)}^2+ \bigl(\ln{t}/{t_0}\bigr)^6\|\p_X^2\Omega\|_{L^2(m)}^2 \Bigr)^\f12 \\
& \times \Bigl(\bigl(\ln{t}/{t_0}\bigr)\tt^{-2}\|\p_X\p_Y\Omega\|_{L^2(m)}^2+\bigl(\ln{t}/{t_0}\bigr)^5 \|\p_X^2\Omega\|_{L^2(m)}^2+\bigl(\ln{t}/{t_0}\bigr)^2 \|\p_X\Omega\|_{L^2(m)}^2\\
&+ \bigl(\ln{t}/{t_0}\bigr)^3\tt^{-2}\|\p_X^2\Omega\|_{L^2(m)}^2+\bigl(\ln{t}/{t_0}\bigr)^3\|\p_X\p_Y\Omega\|_{L^2(m)}^2\Bigr)\\
\leq &C_{m,\sigma} \nu^{-1} \tt^{-\f12} \bigl(\ln{t}/{t_0}\bigr)^{\f\sigma2-\f{23}4} E(t)^\f12 D(t).
\end{align*}
Along the same line, we get, by using
 \eqref{eq3.6}, that
\begin{align*}
D_3  \leq & \nu^{-1} \tt^{-2}\Bigl(\|\p_X^2\p_Y \Delta_t^{-1}\Omega\|_{L^\oo}\|\p_X\Omega\|_{L^2(m)}+\|\p_X^3\Delta_t^{-1}\Omega\|_{L^\oo}\|\p_Y\Omega\|_{L^2(m)}\Bigr)\|\p_X^2\Omega\|_{L^2(m)}\\
\leq &C_\sigma\nu^{-1} \tt^{\sigma-1} \|\p_X^2\Omega\|_{L^2(m)}\Bigl(  \|\p_X\p_Y\Omega\|_{L^2(m)}^{\f12+\sigma}\|\p_X^2\p_Y\Omega\|_{L^2(m)}^{\f12-\sigma}\|\p_X\Omega\|_{L^2(m)}\\
& + \|\p_X^2\Omega\|_{L^2(m)}^{\f12+\sigma}\|\p_X^3\Omega\|_{L^2(m)}^{\f12-\sigma}\|\p_Y\Omega\|_{L^2(m)}\Bigr)\\
\leq &C_\sigma \nu^{-1} \tt^{-\f12}\bigl(\ln{t}/{t_0}\bigr)^{\f\sigma2-\f{23}4} \Bigl(\bigl(\ln{t}/{t_0}\bigr)^3\|\p_X^2\Omega\|_{L^2(m)}\Bigr) \Bigl(  \bigl((\ln{t}/{t_0})^\f32\|\p_X\p_Y\Omega\|_{L^2(m)}\bigr)^{\f12+\sigma}\\
&\times\bigl((\ln{t}/{t_0})^2\tt^{-1}\|\p_X^2\p_Y\Omega\|_{L^2(m)}\bigr)^{\f12-\sigma}\bigl((\ln{t}/{t_0})\|\p_X\Omega\|_{L^2(m)}\bigr)\\
&+\bigl((\ln{t}/{t_0})^\f52 \|\p_X^2\Omega\|_{L^2(m)}\bigr)^{\f12+\sigma}\bigl((\ln{t}/{t_0})^3\tt^{-1}\|\p_X^3\Omega\|_{L^2(m)}\bigr)^{\f12-\sigma}\|\p_Y\Omega\|_{L^2(m)}\Bigr)\\
\leq &C_\sigma \nu^{-1} \tt^{-\f12} \bigl(\ln{t}/{t_0}\bigr)^{\f\sigma2-\f{23}4} \Bigl( \bigl(\ln{t}/{t_0}\bigr)^6\|\p_X^2\Omega\|_{L^2(m)}^2 \Bigr)^\f12  \Bigl( \bigl(\ln{t}/{t_0}\bigr)^3\|\p_X\p_Y\Omega\|_{L^2(m)}^2\\
&+\bigl(\ln{t}/{t_0}\bigr)^4\tt^{-2}\|\p_X^2\p_Y\Omega\|_{L^2(m)}^2+\bigl(\ln{t}/{t_0}\bigr)^2 \|\p_X\Omega\|_{L^2(m)}^2+\bigl(\ln{t}/{t_0}\bigr)^5 \|\p_X^2\Omega\|_{L^2(m)}^2\\
&+\bigl(\ln{t}/{t_0}\bigr)^6\tt^{-2}\|\p_X^3\Omega\|_{L^2(m)}^2+ \|\p_Y\Omega\|_{L^2(m)}^2\Bigr)\\
\leq &C_{m,\sigma} \nu^{-1} \tt^{-\f12} \bigl(\ln{t}/{t_0}\bigr)^{\f\sigma2-\f{23}4} E(t)^\f12 D(t).
\end{align*}
By summarizing
the above estimates, we arrive at
\begin{equation}\label{eq3.49}
\begin{aligned}
\int \p_X^2 \cN_t\Omega \p_X^2\Omega b^{2m} \leq &C_m\nu^{-1}{\|\Omega(1)\|_{L^2(m)}}\|\p_X^2\Omega\|_{L^2(m)}^2 \\
&+C_{m,\sigma} \nu^{-1} \tt^{-\f12} \bigl(\ln{t}/{t_0}\bigr)^{\f\sigma2-\f{23}4} E(t)^\f12 D(t).
\end{aligned}
\end{equation}

By inserting the estimates \eqref{eq3.48} and \eqref{eq3.49} into \eqref{eq3.47}, we  conclude the proof of \eqref{eq:lem8}.
\end{proof}

Now we are in a position to complete the estimate for the energy functionals defined by \eqref{def:E} and \eqref{def:D}.

\begin{prop}\label{S3prop2}
{\sl Let $m>1$ and $0<\sigma<\f12$. Let $(c_1,c_2,c_3,c_4,c_5,c_6,c_7)$ satisfy \eqref{assumptions on c1-7}.
Then for $t_0\leq t\leq 2t_0,$ the energy functionals: $E(t)$ and $D(t),$ defined by \eqref{def:E} and \eqref{def:D} satisfy
\begin{equation}\label{eq3.51}
\begin{aligned}
t\f{d}{dt}E(t)+\f12D(t)\leq C_m &\bigl(1+\nu^{-1}{\|\Omega(1)\|_{L^2(m)}}\bigr)\bigl( E(t)+\bigl(\ln{t}/{t_0}\bigr)^\f12D(t)\bigr)\\
&+C_{m,\sigma} \nu^{-1} \tt^{-\f12} \bigl(\ln{t}/{t_0}\bigr)^{\f\sigma2+\f{1}4} E(t)^\f12 D(t).
\end{aligned}
\end{equation}
}\end{prop}
\begin{proof} In view of \eqref{def:E} and \eqref{def:D}, we get,
by summarizing the estimates obtained in the previous lemmas in such a way that $\eqref{eq:lem1}+c_1\eqref{eq:lem2}+c_2\eqref{eq:lem3}+c_3\eqref{eq:lem4}+c_4\eqref{eq:lem5}+c_5\eqref{eq:lem6}+c_6\eqref{eq:lem7}+c_7\eqref{eq:lem8}$, that
\beq\label{S3eq6}
\begin{split}
&t\f{d}{dt}E(t)+\f34D(t)
\leq C_m \bigl(1+\nu^{-1}{\|\Omega(1)\|_{L^2(m)}}\bigr)\bigl( E(t)+\bigl(\ln{t}/{t_0}\bigr)^\f12D(t)\bigr)\\
&+C_\sigma \nu^{-1} \tt^{-\f12} \bigl(\ln{t}/{t_0}\bigr)^{\f\sigma2+\f{1}4} E(t)^\f12 D(t)+4c_1\bigl(\ln{t}/{t_0}\bigr)\|\p_X\Omega\|_{L^2(m)}\|\p_Y\Omega\|_{L^2(m)}\\
&+2c_2\bigl(\ln{t}/{t_0}\bigr)\bigl(\p_X \Omega \big|\p_Y \Omega \bigr)_{L^2(m)}+C_mc_2\bigl(\ln{t}/{t_0}\bigr)^2\bigl(\tt^{-2}\|\p_X^2\Omega\|_{L^2(m)}\|\p_X\p_Y\Omega\|_{L^2(m)}\\
&+\|\p_Y^2\Omega\|_{L^2(m)}\|\p_X\p_Y\Omega\|_{L^2(m)}\bigr)+8c_4\bigl(\ln{t}/{t_0}\bigr)^2\|\p_X\p_Y\Omega\|_{L^2(m)}\|\p_Y^2\Omega\|_{L^2(m)}
\\
&+4c_5\bigl(\ln{t}/{t_0}\bigr)^4\|\p_X^2\Omega\|_{L^2(m)}\|\p_X\p_Y\Omega\|_{L^2(m)}+5c_6\bigl(\ln{t}/{t_0}\bigr)^4\bigl(\p_X^2 \Omega \big|\p_X\p_Y \Omega \bigr)_{L^2(m)}\\&+C_mc_6\bigl(\ln{t}/{t_0}\bigr)^5\bigl(\tt^{-2}\|\p_X^3\Omega\|_{L^2(m)}\|\p_X^2\p_Y\Omega\|_{L^2(m)}
+\|\p_X^2\p_Y\Omega\|_{L^2(m)}\|\p_X\p_Y^2\Omega\|_{L^2(m)}\bigr),
\end{split}\eeq
where we used the facts:
$$
c_1\leq \f14, \quad c_3\leq \f{c_2}{12}, \quad c_4\leq \f{c_1}{8}, \quad c_5\leq \f{c_3}{16}, \quad c_7\leq \f{c_6}{24}.
$$

Next we handle the terms in \eqref{S3eq6} which does not contain  $E(t)$ or $D(t).$  If $c_1$ and $c_2$ satisfy moreover $256c_1^2\leq c_2$ and $64c_2\leq 1$, then one has
\begin{align*}
&4c_1\bigl(\ln{t}/{t_0}\bigr)\|\p_X\Omega\|_{L^2(m)}\|\p_Y\Omega\|_{L^2(m)}+2c_2\bigl(\ln{t}/{t_0}\bigr)\bigl(\p_X \Omega \big|\p_Y \Omega \bigr)_{L^2(m)}\\
\leq &\f{1}4\|\p_Y\Omega\|_{L^2(m)}^2+\f{c_2}4\bigl(\ln{t}/{t_0}\bigr)^2\|\p_X\Omega\|_{L^2(m)}^2.
\end{align*}
In case $64C_m^2c_2^2\leq c_1c_3$, we have
\begin{align*}
&C_mc_2\bigl(\ln{t}/{t_0}\bigr)^2\bigl(\tt^{-2}\|\p_X^2\Omega\|_{L^2(m)}\|\p_X\p_Y\Omega\|_{L^2(m)}
+\|\p_Y^2\Omega\|_{L^2(m)}\|\p_X\p_Y\Omega\|_{L^2(m)}\bigr)\\
&\leq  \f{c_3}{16}\bigl(\ln{t}/{t_0}\bigr)^3\bigl( \tt^{-2}\|\p_X^2\Omega\|_{L^2(m)}^2+\|\p_X\p_Y\Omega\|_{L^2(m)}^2\bigr)\\
&\quad+\f{c_1}{16}\bigl(\ln{t}/{t_0}\bigr)\bigl(\tt^{-2}\|\p_X\p_Y\Omega\|_{L^2(m)}^2+\|\p_Y^2\Omega\|_{L^2(m)}^2 \bigr).
\end{align*}
Whereas in case $1024c_4^2\leq c_1c_3 $, one has
\begin{align*}
8c_4\bigl(\ln{t}/{t_0}\bigr)^2\|\p_X\p_Y\Omega\|_{L^2(m)}\|\p_Y^2\Omega\|_{L^2(m)} \leq \f{c_3}{8}\bigl(\ln{t}/{t_0}\bigr)^3\|\p_X\p_Y\Omega\|_{L^2(m)}^2+\f{c_1}{8}\bigl(\ln{t}/{t_0}\bigr)\|\p_Y^2\Omega\|_{L^2(m)}^2.
\end{align*}
If $256c_5^2\leq c_3c_6$ and $400c_6\leq c_3$, there holds
\begin{align*}
&4c_5\bigl(\ln{t}/{t_0}\bigr)^4\|\p_X^2\Omega\|_{L^2(m)}\|\p_X\p_Y\Omega\|_{L^2(m)}
+5c_6\bigl(\ln{t}/{t_0}\bigr)^4\bigl(\p_X^2 \Omega \big|\p_X\p_Y \Omega \bigr)_{L^2(m)}\\
&\leq \f{c_6}4 \bigl(\ln{t}/{t_0}\bigr)^5\|\p_X^2\Omega\|_{L^2(m)}^2+\f{c_3}{16}\bigl(\ln{t}/{t_0}\bigr)^3\|\p_X\p_Y\Omega\|_{L^2(m)}^2.
\end{align*}
If $4C_m^2c_6^2\leq c_5c_7$, one has
\begin{align*}
&C_mc_6\bigl(\ln{t}/{t_0}\bigr)^5\bigl(\tt^{-2}\|\p_X^3\Omega\|_{L^2(m)}\|\p_X^2\p_Y\Omega\|_{L^2(m)}
+\|\p_X^2\p_Y\Omega\|_{L^2(m)}\|\p_X\p_Y^2\Omega\|_{L^2(m)}\bigr)\\
&\leq  \f{c_7}{4}\bigl(\ln{t}/{t_0}\bigr)^6\bigl( \tt^{-2}\|\p_X^3\Omega\|_{L^2(m)}^2+\|\p_X^2\p_Y\Omega\|_{L^2(m)}^2\bigr)\\
&\quad+\f{c_5}{4}\bigl(\ln{t}/{t_0}\bigr)^4\bigl(\tt^{-2}\|\p_X^2\p_Y\Omega\|_{L^2(m)}^2+\|\p_X\p_Y^2\Omega\|_{L^2(m)}^2 \bigr).
\end{align*}

Noticing that under the assumption \eqref{assumptions on c1-7}, all the previous assumptions on $
 c_1,\cdots,c_6,$ are satisfied. Therefore, by substituting the above estimates into \eqref{S3eq6}, we conclude the
 proof of \eqref{eq3.51}.
\end{proof}

Let us turn to the proof of \eqref{eq1.14}.

\begin{prop}\label{S3prop3}
{\sl Let $1<m\leq m_0$, $\delta>0$ and $0<\sigma<\f12$. Suppose \eqref{eq:Om} has a solution $\Omega(t)\in L^2(m_0)$. Then, for any $|(a,b)|\leq 2$ and $t\geq T_0$ for some universal time $T_0$, there holds
\begin{equation}\label{eq:prop3.3}
\begin{aligned}
\|\p_X^a\p_Y^b\Omega(t)\|_{L^2(m)}
\leq &C_{m,m_0,\delta,\sigma} \bigl( 1 +{\nu}^{-1}{\|\Omega(1)\|_{L^2(m_0)}}\bigr)^{m+\f{2(3a+b)(m+1)}{1+2\sigma}} \\
&\qquad\times\tt^{\f12 +\f{m(1+2\delta)}{m_0}\left(\f12+\f{3a+b}{1+2\sigma}\right)}\|\Omega(1)\|_{L^2(m_0)}.
\end{aligned}
\end{equation}
}\end{prop}
\begin{proof}
We first fix $t_0$ to be a large enough time and prove that \eqref{eq:prop3.3} holds near $t_0$. Let us define
\begin{equation}\label{def:T^*}
T^*\eqdefa \inf \bigl\{\ T>t_0 |\ \sup_{t\in [t_0,T]}E(t)<2E(t_0)  \ \bigr\}.
\end{equation}
Observing that $T^*>t_0$ is well-defined.
We define $T_1^*$ to be determined by
\begin{equation}\label{def:T_1^*}
C_m \bigl(1+\nu^{-1}{\|\Omega(1)\|_{L^2(m)}}\bigr)\bigl(\ln{T_1^*}/{t_0}\bigr)^\f12+C_{m,\sigma} \nu^{-1} \langle t_0\rangle^{-\f12} \bigl(\ln{T_1^*}/{t_0}\bigr)^{\f\sigma2+\f{1}4} \bigl(2E(t_0)\bigr)^\f12=\f12.
\end{equation}
Then, we get that for all $t_0\leq t\leq \min(T^*,T^*_1)$,
$$
C_m \bigl(1+\nu^{-1}{\|\Omega(1)\|_{L^2(m)}}\bigr)\bigl(\ln{t}/{t_0}\bigr)^\f12D(t)+C_{m,\sigma} \nu^{-1} \langle t\rangle^{-\f12} \bigl(\ln{t}/{t_0}\bigr)^{\f\sigma2+\f{1}4} \bigl(E(t)\bigr)^\f12 D(t)\leq \f12 D(t),
$$
from which and \eqref{eq3.51}, we infer
$$
t\f{d}{dt} E(t)\leq C_m \bigl(1+\nu^{-1}{\|\Omega(1)\|_{L^2(m)}}\bigr)E(t), \qquad\forall \quad t\in [t_0,\min(T^*,T^*_1)].
$$
Applying Gronwall's inequality gives rise to
$$
E(t)\leq E(t_0)e^{C_m \bigl(1+\nu^{-1}{\|\Omega(1)\|_{L^2(m)}}\bigr)\bigl(\ln {t}/{t_0}\bigr)}, \qquad\forall \ t\in [t_0,\min(T^*,T^*_1)].
$$
We denote $T^*_2$ to be determined by
\begin{equation}\label{def:T_2^*}
C_m \bigl(1+\nu^{-1}{\|\Omega(1)\|_{L^2(m)}}\bigr)\bigl(\ln {T_2^*}/{t_0})=\ln \f32.
\end{equation}
As a consequence, we obtain
$$
E(t)\leq \f32 E(t_0), \qquad\forall \quad t\in [t_0,\min(T^*,T^*_1,T^*_2)].
$$
Comparing the above inequality with \eqref{def:T^*},  we get, by using a standard continuous argument, that $T^*>\min(T^*_1,T^*_2)$.

On the other hand, we deduce from  Corollary \ref{col3.1} that,
$$
E(t_0)= \|\Omega(t_0)\|_{L^2(2m)}\leq C_{m,m_0,\delta} \bigl( 1 +\bigl({\nu}^{-1}{\|\Omega(1)\|_{L^2(m_0)}}\bigr)^{2m}\bigr) \tt^{1 +\f{m}{m_0}(1+2\delta)}\|\Omega(1)\|_{L^2(m_0)}^2,
$$
from which we infer
$$
E(t)\leq C_{m,m_0,\delta} \bigl( 1 +\bigl({\nu}^{-1}{\|\Omega(1)\|_{L^2(m_0)}}\bigr)^{2m}\bigr) \langle t_0\rangle^{1 +\f{m}{m_0}(1+2\delta)}\|\Omega(1)\|_{L^2(m_0)}^2, \qquad \forall t\in [t_0,f_0(t_0)],
$$
where
\begin{align*}
f_0(t_0)\eqdefa t_0 \exp \Bigl( \f1{C_{m,m_0,\delta,\sigma}} \langle t_0\rangle^{-\f{m}{m_0}\f{2(1+2\delta)}{1+2\sigma}} \bigl( 1 +\bigl({\nu}^{-1}{\|\Omega(1)\|_{L^2(m_0)}}\bigr)^{2m}\bigr)^{-\f{4(m+1)}{1+2\sigma}} \Bigr),
\end{align*}
so that for any $T\leq f(t_0),$ there holds
\begin{align*}
&C_m \bigl(1+\nu^{-1}{\|\Omega(1)\|_{L^2(m)}}\bigr)\bigl(\ln{T}/{t_0}\bigr)^\f12+C_{m,\sigma} \nu^{-1} \langle t_0\rangle^{-\f12} \bigl(\ln{T}/{t_0}\bigr)^{\f\sigma2+\f{1}4} \bigl(2E(t_0)\bigr)^\f12\leq \f12,\\
& C_m \bigl(1+\nu^{-1}{\|\Omega(1)\|_{L^2(m)}}\bigr)\bigl(\ln {T}/{t_0})\leq \ln \f32,
\end{align*}
which implies $f_0(t_0)\leq \min (T_1^*,T_2^*)$.

Now for any $t$ being large enough, we  choose some $t_0<t$ so that
\begin{equation}\label{eq3.55}
\ln{t}/{t_0}={C_{m,m_0,\delta,\sigma}}^{-1} \langle t_0\rangle^{-\f{m}{m_0}\f{2(1+2\delta)}{1+2\sigma}}  \bigl( 1 +\bigl({\nu}^{-1}{\|\Omega(1)\|_{L^2(m_0)}}\bigr)^{2m}\bigr)^{-\f{4(m+1)}{1+2\sigma}}.
\end{equation}
Due to the fact: $t_0<t\leq f_0(t_0)$,  we deduce that
$$
E(t)\leq C_{m,m_0,\delta} \bigl( 1 +\bigl({\nu}^{-1}{\|\Omega(1)\|_{L^2(m_0)}}\bigr)^{2m}\bigr) \langle t_0\rangle^{1 +\f{m}{m_0}(1+2\delta)}\|\Omega(1)\|_{L^2(m_0)}^2,
$$
from which and \eqref{def:E},  we conclude the proof of \eqref{eq:prop3.3}.
\end{proof}

\section{Long time behaviour}\label{section 4}

The goal of this section is to present the proof of Theorem \ref{Thm3}.
In subsection \ref{subsection 4.1}, we recall some basic properties of the operator $\cL_\oo,$ which is defined by \eqref{def:Fokker-Planck operator},
 on $L^2(G)$ and prove \eqref{eq1.11a} on $L^2(m)$. With an explicit solution formula of the semigroup generated by $\cL_\oo,$  we can derive the optimal decay rate of the semigroup without using spectral theory.
In subsection \ref{subsection 4.2}, by using  \eqref{eq1.11a} and the estimates in Theorem \ref{Thm2}, we present the proof of Theorem \ref{Thm3}.

\subsection{Semigroup properties of $\cL_\oo$}\label{subsection 4.1}

For simplicity, we first recall some basic properties of the operator  $\cL_\oo$ on $L^2(G)$, which can be found in \cite{helffer2013spectral,nier2005hypoelliptic} and the references therein.

When considering this operator in  $L^2(G),$ with  the inner product of which is defined by
$$
\bigl( f \big| g \bigr)_{L^2(G)}=\int_{\R^2} f \bar{g} {G}^{-1},
$$
where $G=\f1{4\pi}e^{-\f{X^2+Y^2}4}$ is the Gaussian function, one find that $\cL_\oo$ satisfies
$$
e^{\f{X^2+Y^2}8}\cL_\oo f=\cL_{FK}(e^{\f{X^2+Y^2}8}f)
$$
where
$
\cL_{FK}\eqdefa 4\p_Y^2-\f{Y^2}{4}+1-\f{\sqrt{3}}2(X\p_Y-Y\p_X)
$ stands for the standard Fokker-Planck operator
on $L^2$ space. We refer to Chapter 15 of \cite{helffer2013spectral} and the references therein
 for the properties of $\cL_{FK}.$

Noticing that
$$
[\p_X-\sqrt{3}\p_Y;\cL_\oo]=\f32(\p_X-\sqrt{3}\p_Y)\andf [\sqrt{3}\p_X-\p_Y;\cL_\oo]=\f12(\p_X-\sqrt{3}\p_Y),
$$
we observe that $\cL_\oo$ has a sequence of eigenfunctions:
$$
\cL_\oo \bigl( (\p_X-\sqrt{3}\p_Y)^a (\sqrt{3}\p_X-\p_Y)^b G\bigr)=-\f{3a+b}2 (\p_X-\sqrt{3}\p_Y)^a (\sqrt{3}\p_X-\p_Y)^b G.
$$
 These eigenvalues are all the spectrum of $\cL_\oo$ in $L^2(G).$ Although these eigenfunctions are not orthogonal to each other (otherwise the operator will be self-adjoint), they are all orthogonal to the kernel $G$. Such orthogonality property enables one to prove that (see, for example, Theorem 15.4 in \cite{helffer2013spectral})
$$
\|e^{\tau\cL_\oo}f-M(f)G\|_{L^2(G)}\leq C_\epsilon e^{-(\f12-\epsilon)\tau}\|f\|_{L^2(G)}.
$$

In this paper, we consider the similar properties of the semigroup $e^{\tau\cL_\oo}$ on $L^2(m)$. Let us denote $u(\tau)\eqdefa e^{\tau\cL_\oo}u_0,$ which satisfies
$$
\p_\tau u=4\p_Y^2 u+2Y\p_Yu+2u+\f{\sqrt{3}}2(X\p_Y-Y\p_X)u,
$$
with initial data $u|_{t=0}=u_0$. By applying the Fourier transform to the above equation, one has
\beq\label{S4eq1}
\p_\tau \hat{u}+\f{\sqrt{3}}2 \eta\p_\xi \hat{u}+\bigl(2\eta-\f{\sqrt{3}}2 \xi\bigr)\p_\eta \hat{u} =-4\eta^2\hat{u}.
\eeq
 We are going to solve the above equation by the characteristic method. To do so, we define
$$
\left\{\begin{array}{l}
\displaystyle \f{d}{d\tau}{\Xi}(\tau,\xi,\eta)=\f{\sqrt{3}}2{\Upsilon}(\tau,\xi,\eta), \\
\displaystyle  \f{d}{d\tau}{\Upsilon}(\tau,\xi,\eta)=-\f{\sqrt{3}}2{\Xi}(\tau,\xi,\eta)+2{\Upsilon}(\tau,\xi,\eta), \\
\displaystyle  {\Xi}(0,\xi,\eta)=\xi,\qquad {\Upsilon}(0,\xi,\eta)=\eta,
\end{array}\right.
$$
which can be solved explicitly as
\beq\label{S4eq2}
\quad \left\{\begin{array}{l}
\displaystyle {\Xi}(\tau,\xi,\eta)=\bigl(\f32\xi-\f{\sqrt{3}}2\eta\bigr)e^{\f\tau2}+\bigl(-\f12\xi+\f{\sqrt{3}}2\eta\bigr)e^{\f32\tau}, \\
\displaystyle  {\Upsilon}(\tau,\xi,\eta)=\bigl(\f{\sqrt{3}}2\xi-\f12\eta\bigr)e^{\f\tau2}+\bigl(-\f{\sqrt{3}}2\xi+\f32\eta\bigr)e^{\f32\tau}.
\end{array}\right.
\eeq
Let $g(\tau,\xi,\eta)\eqdefa \hat{u}(\tau,{\Xi}(\tau,\xi,\eta),{\Upsilon}(\tau,\xi,\eta)).$ We deduce from \eqref{S4eq1} that
$$
\f{d}{d\tau} g=-4{\Upsilon}^2 g.
$$
Observing that
\begin{align*}
-4&\int_0^\tau {\Upsilon}^2(s,\xi,\eta)\,ds\\
=&-\bigl(\sqrt{3}\xi-\eta\bigr)^2(e^\tau-1)-\sqrt{3}(\sqrt{3}\xi-\eta)(\sqrt{3}\eta-\xi)(e^{2\tau}-1)-(\sqrt{3}\eta-\xi)^2(e^{3\tau}-1)\\
=&-(\sqrt{3}{\Xi}-{\Upsilon})^2(1-e^{-\tau})-\sqrt{3}(\sqrt{3}{\Xi}-{\Upsilon})(\sqrt{3}{\Upsilon}-{\Xi})(1-e^{-2\tau})-(\sqrt{3}{\Upsilon}-{\Xi})^2(1-e^{-3\tau})\\
=&-(1-e^{-\tau})^3{\Xi}^2-2\sqrt{3}e^{-\tau}(1-e^{-\tau})^2{\Xi}{\Upsilon}-(1-e^{-\tau})(1+3e^{-2\tau}){\Upsilon}^2,
\end{align*}
we obtain
\begin{align*}
\hat{u}(\tau,{\Xi}(\tau,\xi,\eta),{\Upsilon}(\tau,\xi,\eta)) =&g(\tau,\xi,\eta)\\
=&\hat{u_0}(\xi,\eta)e^{-(1-e^{-\tau})^3{\Xi}^2-2\sqrt{3}e^{-\tau}(1-e^{-\tau})^2{\Xi}{\Upsilon}-(1-e^{-\tau})(1+3e^{-2\tau}){\Upsilon}^2},
\end{align*}
from which and \eqref{S4eq2}, we infer
\begin{equation}\label{Fourier formula}
\hat{u}(t,\xi,\eta)=e^{\Phi(\tau,\xi,\eta)} \hat{u}_0(\xi_0(\tau,\xi,\eta),\eta_0(\tau,\xi,\eta)),
\end{equation}
where the backward characteristic flow $\xi_0$ and $\eta_0$ are determined by
\begin{equation}\label{xi0,eta0}
\quad \left\{\begin{array}{l}
\displaystyle \xi_0=\bigl(\f32\xi-\f{\sqrt{3}}2\eta\bigr)e^{-\f\tau2}+\bigl(-\f12\xi+\f{\sqrt{3}}2\eta\bigr)e^{-\f32\tau}, \\
\displaystyle \eta_0=\bigl(\f{\sqrt{3}}2\xi-\f12\eta\bigr)e^{-\f\tau2}+\bigl(-\f{\sqrt{3}}2\xi+\f32\eta\bigr)e^{-\f32\tau},
\end{array}\right.
\end{equation}
and  $\Phi(\tau,\xi,\eta)$  by
\begin{equation}
\Phi(t,\xi,\eta)=-(1-e^{-\tau})^3{\xi}^2-2\sqrt{3}e^{-\tau}(1-e^{-\tau})^2{\xi}{\eta}-(1-e^{-\tau})(1+3e^{-2\tau}){\eta}^2.
\end{equation}

With \eqref{Fourier formula}, we can prove the following semigroup estimate:

\begin{prop}\label{S4prop1}
{\sl For $0\leq m\leq 2$, we have
\begin{equation}\label{eq4.4}
\|e^{\tau\cL_\oo}u_0\|_{L^2(m)}\leq C e^{\left(1-\f{m}2\right)\tau}\|u_0\|_{L^2(m)}.
\end{equation}
For $m\geq 2$, $e^{\tau\cL_\oo}$ is a bounded operator on $L^2(m)$ and satisfies
\begin{equation}\label{eq4.5}
\|e^{\tau\cL_\oo}u_0-M(u_0)G\|_{L^2(m)}\leq C e^{-\f{m-2}2\tau}\|u_0-M(u_0)G\|_{L^2(m)},
\end{equation}
for $2\leq m\leq 3$, and when $m\geq 3$,
\begin{equation}\label{eq4.6}
\|e^{\tau\cL_\oo}u_0-M(u_0)G\|_{L^2(m)}\leq C_m e^{-\f\tau2}\|u_0-M(u_0)G\|_{L^2(m)},
\end{equation}
where $M(u_0)=\int u_0 $ denotes the mass of $u_0$.
}\end{prop}
\begin{proof}
Noticing that $L^2(m)$ norm in the physical spaces is equivalent to the Sobolev  $H^m$ norm in  the Fourier variables $(\xi,\eta)$,
 below we shall perform $H^m$ estimate of \eqref{Fourier formula}.

First, we begin with the $L^2$ estimate. It's easy to observe that $\Phi\leq 0$ and
$$
\det\Bigl(\f{\p(\xi_0,\eta_0)}{\p(\xi,\eta)}\Bigr)=e^{-2\tau}.
$$
So that we get, by using changes of variables, that
\begin{equation}\label{eq4.7}
\|e^{\tau\cL_\oo}u_0\|_{L^2}=\|\hat{u}(\tau)\|_{L^2}\leq\|\hat{u}_0(\xi_0,\eta_0)\|_{L^2}=e^\tau\|\hat{u}_0\|_{L^2}=e^\tau\|{u}_0\|_{L^2}.
\end{equation}
One can also prove the above inequality by a trivial energy estimate. The reason why we present such a proof here is to point out that on the Fourier side, the $e^\tau$ growth comes from the change of variables in $L^2$ norms.

Next, we consider the $L^2(m)$ estimate of the semigroup $e^{\tau\cL_\oo}$ with $m\geq 2$, and prove that such weight (regularity in Fourier side) is enough to overcome the $e^{\tau}$ growth. To do so, we only need to consider $\tau\geq 100$, because when $\tau\leq 100$, one has the trivial estimate:
$$
\|e^{\tau\cL_\oo}u_0\|_{L^2(m)}=\|\hat{u}(\tau)\|_{H^m}\leq C\|e^{-\Phi}\|_{H^m}\|\hat{u}_0(\xi_0,\eta_0)\|_{H^m}\leq Ce^{C\tau}\|\hat{u}_0\|_{H^m}\leq C\|u_0\|_{L^2(m)}.
$$
Yet for such large enough $\tau\geq 100$, we have $$
\Phi\leq -\f{\xi^2+\eta^2}4,$$
and obviously
$$
\|e^{\Phi}\|_{H^m}\leq C_m.
$$
Then we get, by using Moser type inequality, that for $\tau\geq 100$,
\beq \label{S4eq3}
\begin{split}
\|\hat{u}(\tau)\|_{H^m}&\lesssim \|e^{-\Phi}\|_{L^\oo}\|\hat{u}_0(\xi_0,\eta_0)\|_{\dot{H}^m}+\|\hat{u}_0(\xi_0,\eta_0)\|_{L^\oo}\|e^{-\Phi}\|_{H^m} \\
&\lesssim_m \|\hat{u}_0(\xi_0,\eta_0)\|_{\dot{H}^m}+\|\hat{u}_0(\xi_0,\eta_0)\|_{L^\oo}.
\end{split}\eeq
While for any integer $k\in\N$, by using a change of variables, we find
$$
\|\hat{u}_0(\xi_0,\eta_0)\|_{\dot{H}^k}\leq Ce^{\left(1-\f{k}2\right)\tau}\|\hat{u}_0\|_{\dot{H}^k},
$$
which, together with the interpolation argument, implies that
$$
\|\hat{u}_0(\xi_0,\eta_0)\|_{\dot{H}^m} \leq Ce^{\left(1-\f{m}2\right)\tau}\|\hat{u}_0\|_{\dot{H}^m} \leq Ce^{\left(1-\f{m}2\right)\tau}\|{u}_0\|_{L^2(m)}.
$$
Whereas it follows from the Sobolev embedding theorem that
$$
\|\hat{u}_0(\xi_0,\eta_0)\|_{L^\oo} =\|\hat{u}_0\|_{L^\oo}\leq C \|\hat{u}_0\|_{H^m}\leq C\|u_0\|_{L^2(m)}.
$$
By substituting the above estimates into \eqref{S4eq3}, we obtain
\begin{equation}\label{eq4.8}
\|e^{\tau\cL_\oo}u_0\|_{L^2(m)}\leq C_m\|u_0\|_{L^2(m)}.
\end{equation}

Next, let us investigate the long-time behavior of the semigroup $e^{\tau\cL_\oo}.$
Once again, we only consider the large time case for $\tau\geq 100$. Since the Gaussian function $G$ is the only steady solution of the linear equation, we may assume (by subtracting the mass multiple of $G$) that the initial data $u_0$ satisfies
$$M(u_0)=\int u_0=0,$$
which also implies $\hat{u}_0(0,0)=0$. Inspired by the above proof of \eqref{eq4.8}, we only need to improve the estimate of the low frequencies part of $\hat{u}_0(\xi_0,\eta_0)$ by using the fact that  $(\xi_0,\eta_0)\rightarrow (0,0)$ as $\tau\rightarrow+\oo$. Precisely,
  for some integer $k\geq 3$, we write
\begin{align*}
\|u\|_{L^2(k)}=&\|\hat{u}\|_{H^k}\leq \sum_{|\alpha|\leq k}\|\p^\alpha u\|_{L^2} \leq C_k \sum_{|\alpha|+|\beta|\leq k}\|\p^\alpha e^{-\Phi} \p^\beta \hat{u}_0(\xi_0,\eta_0)\|_{L^2} \\
\leq & C_k\Bigl( \sum_{|\alpha|\leq k}\|\p^\alpha e^{-\Phi} \hat{u}_0(\xi_0,\eta_0)\|_{L^2}+\sum_{|\alpha|\leq k-1, |\beta|=1}\|\p^\alpha e^{-\Phi} \p^\beta \hat{u}_0(\xi_0,\eta_0)\|_{L^2}\\
&+\sum_{2\leq|\beta|\leq k-1,|\alpha|\leq k-|\beta|}\|\p^\alpha e^{-\Phi} \p^\beta \hat{u}_0(\xi_0,\eta_0)\|_{L^2}+\sum_{|\beta|= k}\| e^{-\Phi} \p^\beta \hat{u}_0(\xi_0,\eta_0)\|_{L^2}\Bigr)\\
\eqdefa & L_1+L_2+L_3+L_4.
\end{align*}
Due to $\hat{u}_0(0,0)=0,$  we write
\begin{align*}
|\hat{u}_0(\xi_0,\eta_0)|&=|\int_0^1 \xi_0 (\p_\xi\hat{u}_0)(s\xi_0,s\eta_0)+\eta_0(\p_\eta\hat{u}_0)(s\xi_0,s\eta_0) ds|\\
&\leq C(|\xi_0|+|\eta_0|)\|\nabla \hat{u}_0\|_{L^\oo} \leq C e^{-\f\tau2}(|\xi|+|\eta|)\|\nabla \hat{u}_0\|_{L^\oo},
\end{align*}
then we get, by using the Sobolev embedding inequality, that
$$
L_1\leq C_k e^{-\f\tau2} \|\nabla \hat{u}_0\|_{H^2} \sum_{|\alpha|\leq k}\|\p^\alpha e^{-\Phi} (|\xi|+|\eta|)\|_{L^2} \leq C_k e^{-\f\tau2} \| \hat{u}_0\|_{H^3}.
$$
For $L_2$, we observe that the derivative on $\hat{u}_0(\xi_0,\eta_0)$ leads to such decay as
$$
\|\nabla \hat{u}_0(\xi_0,\eta_0)\|_{L^\oo}\leq \|(\nabla\xi_0,\nabla\eta_0)\|_{L^\oo}\|\nabla \hat{u}_0\|_{L^\oo}\leq e^{-\f\tau2}\|\nabla \hat{u}_0\|_{L^\oo},
$$
which implies
$$
L_2\leq C_k e^{-\f\tau2} \|\nabla \hat{u}_0\|_{H^2} \sum_{|\alpha|\leq k-1}\|\p^\alpha e^{-\Phi} \|_{L^2} \leq C_k e^{-\f\tau2} \| \hat{u}_0\|_{H^3}.
$$
Along the same line, we deduce that
\begin{align*}
L_3\leq & C_k\sum_{|\alpha|\leq k}\|\p^\alpha e^{-\Phi} \|_{L^4} \sum_{2\leq|\beta|\leq k-1} e^{-\f{|\beta|}2\tau}\| (\p^\beta \hat{u}_0)(\xi_0,\eta_0)\|_{L^4}\\
\leq & C_k\sum_{2\leq|\beta|\leq k-1} e^{\f{1-|\beta|}2\tau}\| \p^\beta \hat{u}_0\|_{L^4}\\
\leq & C_k e^{-\f\tau2}\|  \hat{u}_0\|_{H^k}.
\end{align*}
Whereas due to $k\geq 3,$ we infer
$$
L_4\leq C_k\|e^{-\Phi}\|_{L^\oo}\|\hat{u}_0(\xi_0,\eta_)\|_{\dot{H}^k}\leq C_k e^{(1-\f{k}2)\tau}\|\hat{u}_0\|_{\dot{H}^k} \leq C_k e^{-\f\tau2}\|\hat{u}_0\|_{H^k}.
$$
By summarizing the above estimates, we achieve
$$
\|u(\tau)\|_{L^2(k)}\leq C_k e^{-\f\tau2}\|\hat{u}_0\|_{H^k}\leq C_k e^{-\f\tau2}\|u_0\|_{L^2(k)}.
$$
By interpolation between $L^2(k)$ with integers $k\geq 3$, we have proven that for all $m\geq 3$
\begin{equation}\label{eq4.9}
\|e^{\tau\cL_\oo}u_0\|_{L^2(m)}\leq C_m e^{-\f\tau2} \|u_0\|_{L^2(m)},
\end{equation}
which finishes the proof of \eqref{eq4.6}.

 Finally \eqref{eq4.4} follows from  interpolation between \eqref{eq4.7} and \eqref{eq4.8},  and \eqref{eq4.5} from
 interpolation between \eqref{eq4.8} and \eqref{eq4.9}. This finishes the proof of Proposition \ref{S4prop1}.
\end{proof}

\subsection{Proof of Theorem \ref{Thm3}}\label{subsection 4.2}
In this subsection, we present the proof of  Theorem \ref{Thm3}.

\begin{proof}[Proof of Theorem \ref{Thm3}]
Let us denote
$$
\tilde{\Omega}\eqdef \Omega-\alpha G,
$$
where $\alpha=M(\Omega)$ is the mass of $\Omega$. It is easy to observe from  \eqref{eq:Om} that $M(\Omega)$ is a conserved
quantity, so that $\alpha=M(\Omega(1)),$ which is independent of time. Then in view of \eqref{eq:Om}  and $\cL_t G=0,$  we
find
\beq \label{S4eq4}
t\p_t \tilde{\Omega}=\cL_t \tilde{\Omega}+\cN_t (\tilde{\Omega}+\alpha G).
\eeq

Since we are investigating the long-time behavior of $\tilde{\Omega}(t),$ we focus on the timespan $[t_0,+\oo)$ beginning at some large time $t_0$. By virtue of \eqref{S4eq4}, we write
\begin{equation}
\tilde{\Omega}(t)=e^{(\ln {t}/{t_0})\cL_\oo}\tilde{\Omega}(t_0)+\int_{t_0}^t e^{\left(\ln {t}/s\right)\cL_\oo} \bigl( (\cL_s-\cL_\oo) \tilde{\Omega}(s)+\cN_s (\tilde{\Omega}(s)+\alpha G) \bigr)\,\f{ds}s.
\end{equation}

We first get, by using \eqref{eq4.6} and \eqref{eq:col3.1}, that
\begin{equation}\label{eq4.11}
\begin{aligned}
&\bigl\|e^{(\ln {t}/{t_0})\cL_\oo}\tilde{\Omega}(t_0)\bigr\|_{L^2(m)}
\leq C_m e^{-\f12\left(\ln {t}/{t_0}\right)}\|\tilde{\Omega}(t_0)\|_{L^2(m)}\\
&\leq  C_{m,m_0,\delta} \bigl(\f{t}{t_0}\bigr)^{-\f12}\bigl( 1 +\nu^{-1}{\|\Omega(1)\|_{L^2(m_0)}}\bigr)^{m} \langle t_0\rangle^{\f12 +\f{m}{m_0}\left(\f12+\delta\right)}\|\Omega(1)\|_{L^2(m_0)}\\
&\leq  C_{m,m_0,\delta} t^{-\f12}\bigl( 1 +{\nu}^{-1}{\|\Omega(1)\|_{L^2(m_0)}}\bigr)^{m} \langle t_0\rangle^{1 +\f{m}{m_0}\left(\f12+\delta\right)}\|\Omega(1)\|_{L^2(m_0)}.
\end{aligned}\end{equation}

Whereas in view of \eqref{eq:Oma} and \eqref{def:Fokker-Planck operator}, we  compute
\begin{align*}
\cL_t-\cL_\oo=&\f1{1+\f{t^2}3}\Bigl(\p_X-\f{t}2\bigl(1+\f{t^2}{12}\bigr)^{-\f12}\p_Y\Bigr)^2 -\f{3}{1+\f{t^2}{12}} \p_Y^2\\
&+\f1{2(1+\f{t^2}3)}\Bigl(X-\f{t}2\bigl(1+\f{t^2}{12}\bigr)^{-\f12}Y\Bigr)\Bigl(\p_X-\f{t}2\bigl(1+\f{t^2}{12}\bigr)^{-\f12}\p_Y\Bigr)
-\f{3}{2(1+\f{t^2}{12})} Y\p_Y \\
&+\Bigl(\f{3}{\sqrt{1+\f{t^2}{12}}(\sqrt{1+\f{t^2}{12}}+t)}-\f{3t}{\sqrt{4+\f{t^2}3}(3+t^2)} \Bigr)  (X\p_Y-Y\p_X) -\f{12}{12+t^2}
\end{align*}
from which we infer
$$
|(\cL_s-\cL_\oo) \tilde{\Omega}(s)|\leq C \langle s\rangle^{-2} \bigl( |\nabla^2 \tilde{\Omega}|+|(X,Y)||\nabla \tilde{\Omega}|+|\tilde{\Omega}| \bigr).
$$
Then we get, by using \eqref{eq4.6} and \eqref{eq:prop3.3}, that for $m_0\geq m+1$
\begin{equation}\notag
\begin{aligned}
\bigl\|\int_{t_0}^t& e^{\left(\ln {t}/s\right)\cL_\oo} (\cL_s-\cL_\oo) \tilde{\Omega}(s)\,\f{ds}s\bigr\|_{L^2(m)}\\
\leq & C_m\int_{t_0}^t e^{-\f12\left(\ln {t}/s\right)} \langle s\rangle^{-2} \bigl( \|\nabla^2 \tilde{\Omega}(s)\|_{L^2(m)} +\|\nabla \tilde{\Omega}(s)\|_{L^2(m+1)}+\| \tilde{\Omega}(s)\|_{L^2(m)} \bigr)\,\f{ds}{s}\\
\leq & C_{m,m_0,\delta,\sigma}\int_{t_0}^t e^{-\f12\left(\ln {t}/s\right)} \langle s\rangle^{-2}
\Bigl( \bigl( 1 +\nu^{-1}{\|\Omega(1)\|_{L^2(m_0)}}\bigr)^{m+\f{12(m+1)}{1+2\sigma}} \langle s\rangle^{\f12 +\f{m(1+2\delta)}{m_0}\left(\f12+\f{6}{1+2\sigma}\right)}\\
&+\bigl( 1 +\nu^{-1}{\|\Omega(1)\|_{L^2(m_0)}}\bigr)^{m+1+\f{6(m+2)}{1+2\sigma}} \langle s\rangle^{\f12 +\f{(m+1)(1+2\delta)}{m_0}\left(\f12+\f3{1+2\sigma}\right)} \Bigr)\|\Omega(1)\|_{L^2(m_0)}\,\f{ds}{s}\\
\leq & C_{m,m_0,\delta,\sigma}t^{-\f12}\bigl( 1 +\nu^{-1}{\|\Omega(1)\|_{L^2(m_0)}}\bigr)^{m+\f{12(m+1)}{1+2\sigma}}\|\Omega(1)\|_{L^2(m_0)}\int_{t_0}^t   \langle s\rangle^{-2 +\f{m(1+2\delta)}{m_0}\left(\f12+\f{6}{1+2\sigma}\right)}\, ds,
\end{aligned}
\end{equation}
where we used that for all $m\geq 3$, $\delta>0$ and $0<\sigma<\f12$,
\begin{align*}
&m+1+\f{6(m+2)}{1+2\sigma}<m+\f{12(m+1)}{1+2\sigma}
\andf \\
&\f12 +\f{(m+1)(1+2\delta)}{m_0}\left(\f12+\f3{1+2\sigma}\right)<\f12 +\f{m(1+2\delta)}{m_0}\left(\f12+\f{6}{1+2\sigma}\right).
\end{align*}
Noticing that $m_0$ is taken to be so large  that
$$
\f{m(1+2\delta)}{m_0}\bigl(\f12+\f{6}{1+2\sigma}\bigr)<1,
$$
we achieve
\begin{equation}\label{eq4.12}
\begin{aligned}
\|\int_{t_0}^t& e^{\left(\ln {t}/s\right)\cL_\oo} (\cL_s-\cL_\oo) \tilde{\Omega}(s)\f{ds}s\bigr\|_{L^2(m)}
\leq  C_{m,m_0,\delta,\sigma}t^{-\f12}\\
&\times\bigl( 1 +\nu^{-1}{\|\Omega(1)\|_{L^2(m_0)}}\bigr)^{m+\f{12(m+1)}{1+2\sigma}} \langle t_0\rangle^{-1 +\f{m(1+2\delta)}{m_0}\left(\f12+\f{6}{1+2\sigma}\right)} \|\Omega(1)\|_{L^2(m_0)}.
\end{aligned}
\end{equation}

While we get, by using the divergence-free structure and \eqref{eq4.6}, that
\begin{align*}
\bigl\|\int_{t_0}^t& e^{\left(\ln {t}/s\right)\cL_\oo} \cN_s (\tilde{\Omega}(s)+\alpha G)\f{ds}s\bigr\|_{L^2(m)}
\leq C_m\int_{t_0}^t e^{-\f12\left(\ln {t}/s\right)} \|\cN_s (\tilde{\Omega}(s)+\alpha G)\|_{L^2(m)}\f{ds}s\\
\leq & C_m \Bigl( \alpha^2 t^{-\f12}\int_{t_0}^t  \|\cN_s ( G)\|_{L^2(m)}s^{-\f12}ds \\
&+\alpha t^{-\f12}\nu^{-1}\int_{t_0}^t  \|\p_Y \Delta_s^{-1}G\p_X \tilde{\Omega}(s)- \p_X\Delta_s^{-1}G\p_Y \tilde{\Omega}(s)\|_{L^2(m)}\langle s\rangle^{-\f52}ds\\
&+\alpha t^{-\f12}\nu^{-1}\int_{t_0}^t  \|\p_Y \Delta_s^{-1}\tilde{\Omega}(s)\p_X G- \p_X\Delta_s^{-1}\tilde{\Omega}(s)\p_Y G\|_{L^2(m)}\langle s\rangle^{-\f52}ds\\
&+ t^{-\f12}\int_{t_0}^t  \|\cN_s ( \tilde{\Omega}(s))\|_{L^2(m)}s^{-\f12}ds\Bigr) \eqdefa M_1+M_2+M_3+M_4.
\end{align*}
Next, we estimate term by term above.

It follows from \eqref{eq3.6} and $\alpha\leq \|\Omega(1)\|_{L^2(m)}$ that
\begin{equation}\notag
\begin{aligned}
M_1 \leq &C_m \alpha^2 t^{-\f12} \nu^{-1} \int_{t_0}^t  \bigl(\|\p_Y \Delta_s^{-1}G\|_{L^\oo}\|\p_X G\|_{L^2(m)}+ \|\p_X\Delta_s^{-1}G\|_{L^\oo}\|\p_Y G\|_{L^2(m)}\bigr)\langle s\rangle^{-\f52}\,ds\\
\leq & C_{m,\sigma_0}  \|\Omega(1)\|_{L^2(m)}^2 t^{-\f12} \nu^{-1} \int_{t_0}^t  \langle s\rangle^{\sigma_0-\f32}\,ds \\
\leq & C_{m,\sigma_0} t^{-\f12}  {\nu}^{-1}{\|\Omega(1)\|_{L^2(m)}} \langle t_0\rangle^{\sigma_0-\f12} \|\Omega(1)\|_{L^2(m)},
\end{aligned}
\end{equation}
where $0<\sigma_0<\f12$ is a small constant. For simplicity, we just take $\sigma_0=\f14$ to conclude
\begin{equation}\label{eq4.13}
M_1\leq C_m t^{-\f12} {\nu}^{-1}{\|\Omega(1)\|_{L^2(m)}} \|\Omega(1)\|_{L^2(m)}.
\end{equation}

While we get, by using \eqref{eq3.6} that
\begin{align*}
M_2 \leq &C_m \alpha t^{-\f12} \nu^{-1} \int_{t_0}^t  \bigl(\|\p_Y \Delta_s^{-1}G\|_{L^\oo}\|\p_X \tilde{\Omega}(s)\|_{L^2(m)}+ \|\p_X\Delta_s^{-1}G\|_{L^\oo}\|\p_Y \tilde{\Omega}(s)\|_{L^2(m)}\bigr)\langle s\rangle^{-\f52}ds\\
\leq &C_m  t^{-\f12}\nu^{-1}\|\Omega(1)\|_{L^2(m)} \int_{t_0}^t  \Bigl(\|\p_X \tilde{\Omega}(s)\|_{L^2(m)}+ \langle s\rangle\|\p_Y \tilde{\Omega}(s)\|_{L^2(m)}\Bigr)\langle s\rangle^{\sigma_1-\f52}ds\\
\leq &C_m  t^{-\f12}\nu^{-1}\|\Omega(1)\|_{L^2(m)} \int_{t_0}^t  \Bigl(\|\p_X \tilde{\Omega}(s)\|_{L^2(m)}+ \langle s\rangle\|\p_Y^2 \tilde{\Omega}(s)\|_{L^2(m)}^\f12\| \tilde{\Omega}(s)\|_{L^2(m)}^\f12\Bigr)\langle s\rangle^{\sigma_1-\f52}ds,
\end{align*}
where we used the following inequality, which can be derived by using integration by parts,
\begin{equation}\label{eq4.15}
\|\p_Y \tilde{\Omega}\|_{L^2(m)}\leq C_m\|\p_Y^2 \tilde{\Omega}(s)\|_{L^2(m)}^\f12\| \tilde{\Omega}(s)\|_{L^2(m)}^\f12.
\end{equation}
Then by applying \eqref{eq:prop3.3} and
\begin{equation}\label{eq4.14}
\| \tilde{\Omega}(s)\|_{L^2(m)} \leq s^{-\f12} \sup_{s\in[t_0,t]} \bigl(s^\f12\|\tilde{\Omega}(s)\|_{L^2(m)}\bigr),
\end{equation}
 we find
\begin{align*}
M_2 \leq & C_{m,m_0,\delta,\sigma} t^{-\f12}\nu^{-1}\|\Omega(1)\|_{L^2(m)} \int_{t_0}^t  \Bigl( \bigl( 1 +\nu^{-1}\|\Omega(1)\|_{L^2(m_0)}\bigr)^{m+\f{6(m+1)}{1+2\sigma}}\|\Omega(1)\|_{L^2(m_0)}\\
&\times\langle s\rangle^{\sigma_1-2 +\f{m(1+2\delta)}{m_0}(\f12+\f{3}{1+2\sigma})}+ \bigl( 1 +\nu^{-1}\|\Omega(1)\|_{L^2(m_0)}\bigr)^{\f{m}2+\f{2(m+1)}{1+2\sigma}} \|\Omega(1)\|_{L^2(m_0)}^\f12\\
&\times \sup_{s\in[t_0,t]} \bigl(s^\f12\|\tilde{\Omega}(s)\|_{L^2(m)}\bigr)^\f12 \langle s\rangle^{\sigma_1-\f54 +\f{m(1+2\delta)}{2m_0}(\f12+\f{2}{1+2\sigma})} \Bigr)ds.
\end{align*}
Let us choose $0<\sigma_1<\f12$ to be so small that
$$
\sigma_1-2 +\f{m(1+2\delta)}{m_0}\bigl(\f12+\f{3}{1+2\sigma}\bigr)<-1 \andf \sigma_1-\f54 +\f{m(1+2\delta)}{2m_0}\bigl(\f12+\f{2}{1+2\sigma}\bigr)<-1,
$$
we deduce that
\begin{equation}\label{eq4.16}
\begin{aligned}
M_2 \leq & C_{m,m_0,\delta,\sigma,\sigma_1} t^{-\f12} \Bigl( \bigl( 1 +\nu^{-1}\|\Omega(1)\|_{L^2(m_0)}\bigr)^{m+1+\f{6(m+1)}{1+2\sigma}} \langle t_0\rangle^{\sigma_1-1 +\f{m(1+2\delta)}{m_0}\left(\f12+\f{3}{1+2\sigma}\right)}\\
&\times\|\Omega(1)\|_{L^2(m_0)}+ \bigl( 1 +\nu^{-1}\|\Omega(1)\|_{L^2(m_0)}\bigr)^{\f{m}2+1+\f{2(m+1)}{1+2\sigma}} \langle t_0\rangle^{\sigma_1-\f14 +\f{m(1+2\delta)}{2m_0}\left(\f12+\f{2}{1+2\sigma}\right)}\\
&\times\|\Omega(1)\|_{L^2(m_0)}^\f12  \sup_{s\in[t_0,t]} (s^\f12\|\tilde{\Omega}(s)\|_{L^2(m)})^\f12\Bigr) \\
\leq &\f{t^{-\f12}}8\sup_{s\in[t_0,t]} (s^\f12\|\tilde{\Omega}(s)\|_{L^2(m)})+C_{m,m_0,\delta,\sigma,\sigma_1} t^{-\f12} \|\Omega(1)\|_{L^2(m_0)} \\
&\times \Bigl(\bigl( 1 +\nu^{-1}\|\Omega(1)\|_{L^2(m_0)}\bigr)^{m+1+\f{6(m+1)}{1+2\sigma}}\langle t_0\rangle^{\sigma_1-1 +\f{m(1+2\delta)}{m_0}(\f12+\f{3}{1+2\sigma})}\\
 &\qquad+\bigl( 1 +\nu^{-1}\|\Omega(1)\|_{L^2(m_0)}\bigr)^{m+2+\f{4(m+1)}{1+2\sigma}} \langle t_0\rangle^{2\sigma_1-\f12 +\f{m(1+2\delta)}{m_0}\left(\f12+\f{2}{1+2\sigma}\right)}\Bigr).
\end{aligned}
\end{equation}

For  $M_3$, we get, by applying \eqref{eq3.6}, that
\begin{align*}
M_3 \leq &C_m \alpha t^{-\f12} \nu^{-1} \int_{t_0}^t  \bigl(\|\p_Y \Delta_s^{-1}\tilde{\Omega}(s)\|_{L^\oo}\|\p_X G\|_{L^2(m)}+ \|\p_X\Delta_s^{-1}\tilde{\Omega}(s)\|_{L^\oo}\|\p_Y G\|_{L^2(m)}\bigr)\langle s\rangle^{-\f52}ds\\
\leq &C_m  t^{-\f12}\nu^{-1}\|\Omega(1)\|_{L^2(m)} \int_{t_0}^t  \| \tilde{\Omega}(s)\|_{L^2(m)}^{\f12+\sigma_2}\|\p_X \tilde{\Omega}(s)\|_{L^2(m)}^{\f12-\sigma_2}\langle s\rangle^{\sigma_2-\f32}\,ds,
\end{align*}
from which, \eqref{eq4.14} and \eqref{eq:prop3.3}, we infer
\begin{align*}
M_3 \leq  &C_{m,m_0,\delta,\sigma,\sigma_2}  t^{-\f12}  \bigl( 1 +\nu^{-1}\|\Omega(1)\|_{L^2(m_0)}\bigr)^{1+\left(\f12-\sigma_2\right)\left(m+\f{6(m+1)}{1+2\sigma}\right)}\|\Omega(1)\|_{L^2(m_0)}^{\f12-\sigma_2}\\
&\quad\times\sup_{s\in[t_0,t]} \bigl(s^\f12\|\tilde{\Omega}(s)\|_{L^2(m)}\bigr)^{\f12+\sigma_2}\int_{t_0}^t  \langle s\rangle^{-\f32+\left(\f12-\sigma_2\right)\f{m(1+2\delta)}{m_0}\left(\f12+\f{3}{1+2\sigma}\right)}\,ds\\
\leq  &C_{m,m_0,\delta,\sigma,\sigma_2}  t^{-\f12}  \bigl( 1 +\nu^{-1}\|\Omega(1)\|_{L^2(m_0)}\bigr)^{1+(\f12-\sigma_2)(m+\f{6(m+1)}{1+2\sigma})}\|\Omega(1)\|_{L^2(m_0)}^{\f12-\sigma_2}\\
&\quad\times\sup_{s\in[t_0,t]} \bigl(s^\f12\|\tilde{\Omega}(s)\|_{L^2(m)}\bigr)^{\f12+\sigma_2}  \langle t_0\rangle^{-\f12+\left(\f12-\sigma_2\right)\f{m(1+2\delta)}{m_0}\left(\f12+\f{3}{1+2\sigma}\right)},
\end{align*}
where we choose $0<\sigma_2<\f12$ to be so close  to $\f12$  that
$$
\bigl(\f12-\sigma_2\bigr)\f{m(1+2\delta)}{m_0}\left(\f12+\f{3}{1+2\sigma}\right)<\f12.
$$
By applying Young's inequality, we obtain
\begin{equation}\label{eq4.17}
\begin{aligned}
&M_3 \leq \f{t^{-\f12}}8 \sup_{s\in[t_0,t]} (s^\f12\|\tilde{\Omega}(s)\|_{L^2(m)}) +C_{m,m_0,\delta,\sigma,\sigma_2}  t^{-\f12}  \\
&\ \times\bigl( 1 +\nu^{-1}\|\Omega(1)\|_{L^2(m_0)}\bigr)^{\f{2}{1-2\sigma_2}+m+\f{6(m+1)}{1+2\sigma}}\langle t_0\rangle^{-\f1{1-2\sigma_2}+\f{m(1+2\delta)}{m_0}\left(\f12+\f{3}{1+2\sigma}\right)}\|\Omega(1)\|_{L^2(m_0)}.
\end{aligned}\end{equation}

Finally, we consider the fully nonlinear part $M_4$. By applying \eqref{eq3.6} with $0<\sigma_3<\f12$, we find
\begin{align*}
M_4 \leq &C_m  t^{-\f12} \nu^{-1} \int_{t_0}^t  \bigl(\|\p_Y \Delta_s^{-1}\tilde{\Omega}(s)\|_{L^\oo}\|\p_X \tilde{\Omega}(s)\|_{L^2(m)}\\
&\qquad\qquad\qquad+ \|\p_X\Delta_s^{-1}\tilde{\Omega}(s)\|_{L^\oo}\|\p_Y \tilde{\Omega}(s)\|_{L^2(m)}\bigr)\langle s\rangle^{-\f52}\,ds\\
\leq &C_m  t^{-\f12} \nu^{-1} \int_{t_0}^t  \| \tilde{\Omega}(s)\|_{L^2(m)}^{\f12+\sigma_3}\|\p_X \tilde{\Omega}(s)\|_{L^2(m)}^{\f12-\sigma_3} \\
&\qquad\qquad\qquad\times\bigl(\|\p_X \tilde{\Omega}(s)\|_{L^2(m)}+ \langle s\rangle\|\p_Y \tilde{\Omega}(s)\|_{L^2(m)}\bigr)\langle s\rangle^{\sigma_3-\f52}\,ds\\
\leq &C_m  t^{-\f12} \nu^{-1} \int_{t_0}^t  \| \tilde{\Omega}(s)\|_{L^2(m)}^{1+\sigma_3}\|\p_X \tilde{\Omega}(s)\|_{L^2(m)}^{\f12-\sigma_3} \\
&\qquad\qquad\qquad\times\bigl(\|\p_X^2 \tilde{\Omega}(s)\|_{L^2(m)}^\f12+ \langle s\rangle\|\p_Y^2 \tilde{\Omega}(s)\|_{L^2(m)}^\f12\bigr)\langle s\rangle^{\sigma_3-\f52}\,ds,
\end{align*}
where we used \eqref{eq4.15} and
$$
\|\p_X \tilde{\Omega}(s)\|_{L^2(m)} \leq \|\p_X^2 \tilde{\Omega}(s)\|_{L^2(m)}^\f12 \|\tilde{\Omega}(s)\|_{L^2(m)}^\f12.
$$
Then we  get, by applying \eqref{eq4.14} and \eqref{eq:prop3.3}, that
\begin{equation}\begin{aligned}\label{eq4.18}
M_4 \leq &C_{m,m_0,\delta,\sigma,\sigma_3}  t^{-\f12} \sup_{s\in[t_0,t]} \bigl(s^\f12\|\tilde{\Omega}(s)\|_{L^2(m)}\bigr)\\
&\times\Bigl( \bigl( 1 +\nu^{-1}\|\Omega(1)\|_{L^2(m_0)}\bigr)^{m+1+(9-6\sigma_3)\f{m+1}{1+2\sigma}} \int_{t_0}^t \langle s\rangle^{\sigma_3-\f52+\f{m(1+2\delta)}{2m_0}\left(1+\f{9-6\sigma_3}{1+2\sigma}\right)} \,ds\\
&\qquad+\bigl( 1 +\nu^{-1}\|\Omega(1)\|_{L^2(m_0)}\bigr)^{m+1+(5-6\sigma_3)\f{m+1}{1+2\sigma}} \int_{t_0}^t \langle s\rangle^{\sigma_3-\f32+\f{m(1+2\delta)}{2m_0}\left(1+\f{5-6\sigma_3}{1+2\sigma}\right)} \,ds\Bigr)\\
\leq &C_{m,m_0,\delta,\sigma,\sigma_3}  t^{-\f12} \sup_{s\in[t_0,t]} (s^\f12\|\tilde{\Omega}(s)\|_{L^2(m)})\\
&\times\Bigl( \bigl( 1 +\nu^{-1}\|\Omega(1)\|_{L^2(m_0)}\bigr)^{m+1+(9-6\sigma_3)\f{m+1}{1+2\sigma}}  \langle t_0\rangle^{\sigma_3-\f32+\f{m(1+2\delta)}{2m_0}\left(1+\f{9-6\sigma_3}{1+2\sigma}\right)} \\
&\qquad+\bigl( 1 +\nu^{-1}\|\Omega(1)\|_{L^2(m_0)}\bigr)^{m+1+(5-6\sigma_3)\f{m+1}{1+2\sigma}} \langle t_0\rangle^{\sigma_3-\f12+\f{m(1+2\delta)}{2m_0}\left(1+\f{5-6\sigma_3}{1+2\sigma}\right)} \Bigr),
\end{aligned}
\end{equation}
where we choose  $\sigma_3$ to be so small that
\begin{align*}
&\sigma_3-\f32+\f{m(1+2\delta)}{2m_0}\left(1+\f{9-6\sigma_3}{1+2\sigma}\right)<0, \andf \\
&\sigma_3-\f12+\f{m(1+2\delta)}{2m_0}\left(1+\f{5-6\sigma_3}{1+2\sigma}\right)<0.
\end{align*}
To control \eqref{eq4.18}, we need to use some smallness from the negative power of the large time $t_0$. To do so, we take a large time $T_1$ such that,
$$
C_{m,m_0,\delta,\sigma,\sigma_3} \bigl( 1 +\nu^{-1}\|\Omega(1)\|_{L^2(m_0)}\bigr)^{m+1+(9-6\sigma_3)\f{m+1}{1+2\sigma}}  \langle T_1\rangle^{\sigma_3-\f32+\f{m(1+2\delta)}{2m_0}\left(1+\f{9-6\sigma_3}{1+2\sigma}\right)}\leq \f18,
$$
and
$$
C_{m,m_0,\delta,\sigma,\sigma_3} \bigl( 1 +\nu^{-1}\|\Omega(1)\|_{L^2(m_0)}\bigr)^{m+1+(5-6\sigma_3)\f{m+1}{1+2\sigma}}  \langle T_1\rangle^{\sigma_3-\f12+\f{m(1+2\delta)}{2m_0}\left(1+\f{5-6\sigma_3}{1+2\sigma}\right)}\leq \f18,
$$
which results in
\begin{align*}
&T_1\geq C_{m,m_0,\delta,\sigma,\sigma_3} \\
&\ \ \times\bigl( 1 +\nu^{-1}\|\Omega(1)\|_{L^2(m_0)}\bigr)^{(m+1)\max\Bigl(\f{1+\f{9-6\sigma_3}{1+2\sigma}}{\f12-\sigma_3-\f{m(1+2\delta)}{2m_0}(1+\f{9-6\sigma_3}{1+2\sigma})},\f{1+\f{5-6\sigma_3}{1+2\sigma}}{\f12-\sigma_3-\f{m(1+2\delta)}{2m_0}(1+\f{5-6\sigma_3}{1+2\sigma})}\Bigr)}.
\end{align*}
Noticing that as $\delta\rightarrow0$, $\sigma\rightarrow \f12$ and $\sigma_3 \rightarrow0$. As a result, it comes out
$$
(m+1)\max\Bigl(\f{1+\f{9-6\sigma_3}{1+2\sigma}}{\f12-\sigma_3-\f{m(1+2\delta)}{2m_0}(1+\f{9-6\sigma_3}{1+2\sigma})},\f{1+\f{5-6\sigma_3}{1+2\sigma}}{\f12-\sigma_3-\f{m(1+2\delta)}{2m_0}(1+\f{5-6\sigma_3}{1+2\sigma})}\Bigr) \rightarrow \f{7(m+1)}{1-\f{7m}{2m_0}},
$$
so that we can take $\delta,\sigma_3$ small enough and $\sigma_3$ close enough to $\f12$ and choose
\begin{equation}\label{def:T_0}
T_1\eqdefa C_{m,m_0} \bigl( 1 +\nu^{-1}\|\Omega(1)\|_{L^2(m_0)}\bigr)^{\f{7(m+1)}{1-\f{7m}{2m_0}}+\epsilon},
\end{equation}
where $\epsilon$ can be taken arbitrarily small. We conclude that for $t_0\geq T_1$, one has
\begin{equation}\label{eq4.20}
M_4 \leq \f{t^{-\f12}}4 \sup_{s\in[t_0,t]} \bigl(s^\f12\|\tilde{\Omega}(s)\|_{L^2(m)}\bigr).
\end{equation}

Therefore by summarizing the estimates \eqref{eq4.11}, \eqref{eq4.12}, \eqref{eq4.13}, \eqref{eq4.16}, \eqref{eq4.17} and \eqref{eq4.20},
 we deduce that for $t_0\geq T_1$
\begin{align*}
t^{\f12}&\|\tilde{\Omega}(t)\|_{L^2(m)}
\leq \f12\sup_{s\in[t_0,t]} \bigl(s^\f12\|\tilde{\Omega}(s)\|_{L^2(m)}\bigr) \\
&+C_{m,m_0,\delta} \bigl( 1 +\nu^{-1}\|\Omega(1)\|_{L^2(m_0)}\bigr)^{m} \langle t_0\rangle^{1 +\f{m}{m_0}\bigl(\f12+\delta\bigr)}\|\Omega(1)\|_{L^2(m_0)}\\
&+C_{m,m_0,\delta,\sigma}\bigl( 1 +\nu^{-1}\|\Omega(1)\|_{L^2(m_0)}\bigr)^{m+\f{12(m+1)}{1+2\sigma}}\langle t_0\rangle^{-1 +\f{m(1+2\delta)}{m_0}\left(\f12+\f{6}{1+2\sigma}\right)} \|\Omega(1)\|_{L^2(m_0)}\\
&+C_{m,m_0,\delta,\sigma,\sigma_1}  \Bigl(\bigl( 1 +\nu^{-1}\|\Omega(1)\|_{L^2(m_0)}\bigr)^{m+1+\f{6(m+1)}{1+2\sigma}}\langle t_0\rangle^{\sigma_1-1 +\f{m(1+2\delta)}{m_0}\left(\f12+\f{3}{1+2\sigma}\right)}\\
& \quad +\bigl( 1 +\nu^{-1}\|\Omega(1)\|_{L^2(m_0)}\bigr)^{m+2+\f{4(m+1)}{1+2\sigma}} \langle t_0\rangle^{2\sigma_1-\f12 +\f{m(1+2\delta)}{m_0}(\f12+\f{2}{1+2\sigma})}\Bigr)\|\Omega(1)\|_{L^2(m_0)}\\
&+C_{m,m_0,\delta,\sigma,\sigma_2}  t^{-\f12} \bigl( 1 +\nu^{-1}\|\Omega(1)\|_{L^2(m_0)}\bigr)^{\f{2}{1-2\sigma_2}+m+\f{6(m+1)}{1+2\sigma}}\\
&\qquad\times\langle t_0\rangle^{-\f1{1-2\sigma_2}+\f{m(1+2\delta)}{m_0}\left(\f12+\f{3}{1+2\sigma}\right)}\|\Omega(1)\|_{L^2(m_0)}.
\end{align*}
Then we get, by choosing $t_0=T_1$ with $T_1$ being defined in \eqref{def:T_0} and taking supremum over $t\in [T_1, T],$ that
\begin{align*}
&\sup_{t\in[T_1,T]} (t^\f12\|\tilde{\Omega}(t)\|_{L^2(m)}) \leq C_{m,m_0,\delta,\sigma,\sigma_1,\sigma_2} \Bigl( \bigl( 1 +\nu^{-1}\|\Omega(1)\|_{L^2(m_0)}\bigr)^{m+\bigl(1 +\f{m}{m_0}(\f12+\delta)\bigr)\bigl(\f{7(m+1)}{1-\f{7m}{2m_0}}+\epsilon\bigr)} \\
&\quad+\bigl( 1 +\nu^{-1}\|\Omega(1)\|_{L^2(m_0)}\bigr)^{m+\f{12(m+1)}{1+2\sigma}+\bigl(-1 +\f{m(1+2\delta)}{m_0}(\f12+\f{6}{1+2\sigma}) \bigr)\bigl(\f{7(m+1)}{1-\f{7m}{2m_0}}+\epsilon \bigr)} \\
&\quad+\bigl( 1 +\nu^{-1}\|\Omega(1)\|_{L^2(m_0)}\bigr)^{m+1+\f{6(m+1)}{1+2\sigma}+\bigl(\sigma_1-1 +\f{m(1+2\delta)}{m_0}(\f12+\f{3}{1+2\sigma}) \bigr)\bigl(\f{7(m+1)}{1-\f{7m}{2m_0}}+\epsilon \bigr)}\\
&\quad+\bigl( 1 +\nu^{-1}\|\Omega(1)\|_{L^2(m_0)}\bigr)^{m+2+\f{4(m+1)}{1+2\sigma}+\bigl( 2\sigma_1-\f12 +\f{m(1+2\delta)}{m_0}(\f12+\f{2}{1+2\sigma})\bigr)\bigl(\f{7(m+1)}{1-\f{7m}{2m_0}}+\epsilon \bigr)}\\
&\quad+\bigl( 1 +\nu^{-1}\|\Omega(1)\|_{L^2(m_0)}\bigr)^{\f{2}{1-2\sigma_2}+m+\f{6(m+1)}{1+2\sigma}+\bigl(-\f1{1-2\sigma_2}+\f{m(1+2\delta)}{m_0}(\f12+\f{3}{1+2\sigma}) \bigr)\bigl(\f{7(m+1)}{1-\f{7m}{2m_0}}+\epsilon \bigr)}\Bigr)\|\Omega(1)\|_{L^2(m_0)}.
\end{align*}
To simplify these complicated powers, one can focus on the limit case that when $\delta\rightarrow0$, $\sigma\rightarrow\f12$, $\sigma_1\rightarrow0$, $\sigma_2\rightarrow0$, and $\epsilon\rightarrow0$, the largest power is
$$
m+\left(1+\f{m}{2m_0}\right)\f{7(m+1)}{1-\f{7m}{2m_0}}=\f{8(m+1)}{1-\f{7m}{2m_0}}-1,
$$
which comes from the first term, reflecting the evolution of initial data from $T_0$. Indeed, the limits of the other powers are all negative, because they come from nonlinear parts, and we use the most challenging part to determine $T_1$. Therefore, for any small $\e>0$, we can take $(\delta, \sigma_1,\sigma_2)$ to be small enough and $\sigma$ to be close enough to $\f12$ so that
\begin{equation}\label{eq4.21}
\sup_{t\in[T_1,T]} \bigl(t^\f12\|\tilde{\Omega}(t)\|_{L^2(m)}) \leq C_{m,m_0,\epsilon} \bigl(  1 +\nu^{-1}\|\Omega(1)\|_{L^2(m_0)}\bigr)^{\f{8(m+1)}{1-\f{7m}{2m_0}}-1+\e}\|\Omega(1)\|_{L^2(m_0)}.
\end{equation}

We point out that when $m$ and $m_0$ satisfy the relation
$$
m_0>\f72 m \qquad \Leftrightarrow \qquad \f{7m}{2m_0}<1,
$$
one can take $(\delta,\sigma_1,\sigma_2,\sigma_3)$ small enough and $\sigma$ close enough to $\f12$ such that all the assumptions of these parameters used in the proof are satisfied. This completes the proof of Theorem \ref{Thm3}.
\end{proof}

\section*{Acknowledgement}
  P. Zhang is partially supported by the National Key R$\&$D Program of China under grant 2021YFA1000800 and by National Natural Science Foundation of China under Grant 12421001, 12494542, and 12288201.

\section*{Declarations}

\subsection*{Conflict of interest} The authors declare that there are no conflicts of interest.

\subsection*{Data availability}
This article has no associated data.

\end{document}